\documentclass[12pt]{article}

\usepackage{a4wide}
\usepackage[french,english]{babel}
\usepackage[utf8]{inputenc}
\usepackage[T1]{fontenc}
\usepackage{amsthm}
\usepackage{amssymb}
\usepackage{amsmath}
\usepackage{stmaryrd}
\usepackage{imakeidx}
\usepackage{graphicx}
\usepackage{caption}
\usepackage{mathtools}
\theoremstyle{plain}
\usepackage[colorlinks=true,breaklinks=true,linkcolor=black]{hyperref} 
\usepackage{amsfonts}
\usepackage{appendix}
\usepackage[normalem]{ulem}
\usepackage{cancel}
\usepackage{comment}
\usepackage{enumerate}
\usepackage{tikz-cd}
\usepackage[
ordering=Kac,
edge/.style=blue,
indefinite-edge={draw=green,fill=white,densely dashed},
indefinite-edge-ratio=5,
mark=o,
root-radius=.06cm]
{dynkin-diagrams}

\newtheorem{Theorem}{Theorem}[section] 
\newtheorem{Proposition}[Theorem]{Proposition}
\newtheorem{Corollary}[Theorem]{Corollary}
\newtheorem{Definition}[Theorem]{Definition}

\newtheorem{Lemma}[Theorem]{Lemma}
\newtheorem{Example}[Theorem]{Example}
\newtheorem{Remark}[Theorem]{Remark}

\newtheorem{Definition/Proposition}[Theorem]{Definition/Proposition}
\newtheorem{Theorem*}{Theorem}
\newtheorem{Corollary*}[Theorem*]{Corollary}
\usepackage{dsfont}
\date{}
\makeindex[columns=3]
\usepackage{mathrsfs}
\usepackage[all]{xy}

\usepackage{float}

\makeatletter
\renewcommand\theequation%
{\thesection.\arabic{equation}}
\@addtoreset{equation}{section}
\makeatother

\theoremstyle{Definition}
\newtheorem{proposition*}[thmbis]{Proposition}

\theoremstyle{Definition}
\newtheorem{proposition**}[thmter]{Proposition}

\makeatletter \@addtoreset{figure}{section}\makeatother

\newcommand{\R}{\mathbb{R}}
\newcommand{\A}{\mathbb{A}}

\newcommand{\N}{\mathbb{N}}
\newcommand{\Z}{\mathbb{Z}}
\newcommand{\Q}{\mathbb{Q}}

\newcommand{\G}{\mathbb{G}}

\newcommand{\htt}{\mathrm{ht}}

\newcommand{\qp}{\varpi}

\newcommand{\efface}[1]{}

\newcommand{\pr}{\mathrm{proj}}

\newcommand{\cC}{\mathcal{C}}

\newcommand{\cD}{\mathcal{D}}

\newcommand{\cQ}{\mathcal{Q}}

\newcommand{\sT}{\mathscr{T}}

\newcommand{\bt}{\mathbf{t}}

\newcommand{\bx}{\mathbf{x}}

\newcommand{\by}{\mathbf{y}}

\newcommand{\bz}{\mathbf{z}}

\renewcommand{\phi}{\varphi}
\renewcommand{\emptyset}{\varnothing}

\renewcommand{\tilde}[1]{\widetilde{#1}}

\makeatletter
\def\Ddots{\mathinner{\mkern1mu\raise\p@
\vbox{\kern7\p@\hbox{.}}\mkern2mu
\raise4\p@\hbox{.}\mkern2mu\raise7\p@\hbox{.}\mkern1mu}}
\makeatother

\DeclareMathOperator{\supp}{Supp}

\DeclareMathOperator{\sgn}{sgn}

\newcommand{\Inv}{\mathrm{Inv}}

\newcommand{\vv}{\mathrm{v}}

\usepackage[left,pagewise]{lineno}

\hypersetup{
pdfauthor={Hebert, Philippe},
pdftitle={Quantum roots Bruhat order},
}

\title{Quantum roots for  Kac-Moody root systems and finiteness properties of the Kac-Moody affinized Bruhat order}
\author{Auguste \textsc{Hébert} \\Université de Lorraine, Institut Élie Cartan de Lorraine, F-54000 Nancy, France\\ UMR 7502,
auguste.hebert@univ-lorraine.fr \\  \and Paul \textsc{Philippe} \\ Universit\"at M\"unster, Einsteinstr. 62, 48149 M\"unster, Germany\\  paul.philippe@uni-muenster.de}

\begin{document}
\maketitle

\begin{abstract}
The notion of quantum root comes up in the study of the quantum cohomology of generalized flag varieties, quantum Bruhat graphs or in the study of the affinized Bruhat order on an affinized Weyl semigroup. In this paper, we provide a complete classification of quantum roots associated to general Kac-Moody root systems. In particular we prove that the set of quantum roots of a given root system is always finite. We then deduce several finiteness properties of the affinized Bruhat order, generalizing results of A. Welch. In particular, we prove that the intervals for this order are finite.
\end{abstract}

\section{Introduction}
Kac-Moody root systems are particular arrangements of vectors in a free $\Z$-lattice, whose structure is encoded in a generalised Cartan matrix; they appear in various contexts in combinatorics and algebra. For instance, finite Kac-Moody root systems arise from the theory of reductive Lie algebras and reductive groups (such as $\mathfrak{sl}_n$ and $\operatorname{SL}_n$), and the infinite Kac-Moody root systems permit to define their infinite-dimensional counterparts, namely Kac-Moody-Lie algebras and Kac-Moody groups. Certain roots within a given root system, called quantum roots, are of particular interest. Without being named as such, quantum roots of reductive root systems first came up in the computation of quantum cohomology of flag manifolds and related varieties (\cite{brenti1999mixedbruhat},\cite{braverman2011quantum}). They are intrinsically connected to combinatorics of affine Weyl groups (this is already apparent in Peterson's lectures on quantum Cohomology \cite[Lecture 12]{peterson1997lectures}) and they are essential to the definition of quantum Bruhat graphs (introduced by Postnikov \cite{postnikov2004quantum}). The term "quantum root" seems to first appear in the work of Lenart et al. \cite{lenart2012uniform}, in which they apply the quantum Bruhat graph to the analysis of  certain Kirillov-Reshetikhin crystals. More recently, in his doctoral thesis, Schremmer \cite{schremmer2022thesis} established another wide range of applications of quantum Bruhat graphs associated to reductive root systems. 

The notion of quantum root is also related to the notion of minuscule element defined by D. Peterson (\cite{stembridge2001minuscule}), see Proposition~\ref{p_minuscule}. This last concept has applications in various areas, for example in the theory of Kac--Moody homogeneous spaces (see \cite{chaput2012littlewood}) or in the study of Schubert varieties (see \cite[Chapter 9]{billey2000singular}, to cite but a few.

In this paper, we study quantum roots associated to Kac-Moody (infinite) root systems. We provide a complete classification of these roots, and in particular prove that the set of quantum roots is always finite. As observed by the second author in a previous work \cite[Proposition 3.19]{philippe2023grading}, describing these roots is important in order to understand the affinized Bruhat order on the Kac-Moody affine Weyl semi-group. This is a combinatorial object which naturally replaces the extended affine Weyl group when looking at non-reductive Kac-Moody groups over discretely valued fields. In the second part of this paper, we deduce several finiteness properties of the affinized Weyl semi-group, which we use in a follow up work to define and compute affine Kazhdan-Lusztig $R$-polynomials for Kac-Moody groups \cite{hebert2024affine}. Independently, Dean has considered quantum Bruhat graphs associated to Kac-Moody root systems \cite{dean2025quantum}, mainly motivated by the definition of a Demazure product on affinized Weyl semi-groups. This relies on properties of the quantum Bruhat graph which remain conjectural in the non-reductive setting, but were proven by Dean for the type $A_1$ affine. This type can be worked out very explicitly,  especially because only the two simple roots are quantum (this is a particular case of Lemma~\ref{l_no_-1_coefficient}). We expect a good understanding of quantum roots to be necessary to extend these results to the general Kac-Moody setting.

Let us now introduce more precisely the main objects of study and results presented in this paper.

\paragraph{Quantum roots}

 Let $\Phi,\Phi^\vee$ be a (free and cofree) Kac-Moody real root system and its dual. Let $(\alpha_i)_{i\in I}$ and $(\alpha_i^\vee)_{i\in I}$ be its simple roots and coroots, where $I$ is a finite index set. There is a natural bijection $\Phi\rightarrow \Phi^\vee: \alpha\mapsto \alpha^\vee$ which sends $\alpha_i$ to $\alpha_i^\vee$. Each coroot $\alpha^\vee$ decomposes uniquely as $\alpha^\vee=\varepsilon\sum_{i\in I} n_i \alpha_i^\vee$ with $\varepsilon\in \{1,-1\}$ and $(n_i)_{i\in I}\in \mathbb N^I$. The height of $\alpha^\vee$ is then defined as $\htt(\alpha^\vee):=\varepsilon\sum_{i\in I}n_i$. Intuitively, the height computes the size of a given (co)root, for instance a root is simple if and only if its coroot has height $1$. Another quantitative value associated to such a system can be defined through its vectorial Weyl group $W^v$. This object is a Coxeter group whose set of simple reflections bijects with $(\alpha_i)_{i\in I}$. It acts on $\Phi$ (and dually on $\Phi^\vee$) and we have $\Phi=W^v.\{\alpha_i\}_{i\in I}$. To any root $\alpha$ is associated a unique reflection $s_\alpha$, whose length $\ell(s_\alpha)$ also measures the size of $\alpha$.

In general, we have that \begin{equation}2\htt(\alpha^\vee)\geq \ell(s_\alpha)+1\end{equation} by Lemma~\ref{l_inequality_length}. We say that a root is quantum when this is an equality.

Whereas there is a compact characterization of quantum roots in reductive root systems (see \cite[Lemma 4.2]{lenart2012uniform}), the situation seems very different in the non-reductive case. For example, every positive root is quantum in a reductive simply laced root system, but this property breaks even for the simplest infinite root systems. Nonetheless, we provide a complete classification of quantum roots using sequences of Dynkin diagrams.

As a corollary, we obtain:
\begin{Theorem}(see Theorem~\ref{t_finiteness_almost_simple_roots})\label{Theorem: finiteness quantum roots intro}
    Any Kac-Moody root system admits finitely many quantum roots.
\end{Theorem}

From Theorem~\ref{Theorem: finiteness quantum roots intro} we deduce finiteness properties for the affinized Bruhat order. In the following paragraphs we introduce this affinized Bruhat order and explain its connection with quantum roots.

\paragraph{Affine Weyl semi group $W^+$}

 In their study of Kac-Moody groups over non-Archimedean local fields, A. Braverman D. Kazhdan and M. Patnaik consider a semigroup $W^+$, called the affinized Weyl semigroup. It is defined as follows.
When $\Phi$ is the root system of a Kac-Moody group $\mathbb G$ (or of a Kac-Moody Lie algebra $\mathfrak g$), $\Phi^\vee$ lies in the cocharacter lattice $Y$ and we have $\Phi\subset \operatorname{Hom}_\mathbb Z(Y,\Z)$. The action of $W^v$ extends to $Y$. The cocharacter lattice admits a $W^v$-stable subcone $Y^+$, called its integral Tits cone, which plays a role in the representation theory of the associated Kac-Moody Lie algebra. The affinized Weyl semigroup is then defined as $W^+:=Y^+\rtimes W^v$.

When $\Phi$ is associated to a reductive group (hence is finite) then $Y^+=Y$ so $W^+$ is the extended affine Weyl group associated to $\mathbb G$. It is a finite extension of a Coxeter group,  which is very useful to study  $\mathbb G$ over non-archimedean local fields. In the general Kac-Moody case, we expect $W^+$ to play an analogous role but its structure is less understood. It is not a group so it does not admit any Coxeter structure, but a few key notions of Coxeter theory were extended to $W^+$. One of these notions is the Bruhat-like partial pre-order $<$ defined by  Braverman, Kazhdan and Patnaik in \cite[Appending B2]{braverman2016iwahori}. In \cite{muthiah2018iwahori}, D. Muthiah proves that this relation is a partial order, and in \cite{muthiah2019bruhat} Muthiah and D. Orr define a strictly compatible $\mathbb Q$-valued length function $\ell^a:W^+\rightarrow \Q$. We call this partial order the affinized Bruhat order.

\paragraph{Covers for the affinized Bruhat order and quantum roots}
The affinized Bruhat length gives a characterization of covers:   Muthiah and Orr (in the affine ADE case) and the second named author (in the general case) have proved that $\by$ covers $\bx$ (see Definition~\ref{Definition: covers}) if and only if $\by>\bx$ and $\ell^a(\by)=\ell^a(\bx)+1$ (\cite{muthiah2019bruhat}, \cite{philippe2023grading}). The notion of quantum roots naturally comes up in a more explicit description of covers for the affinized Bruhat order. This is sketched in Lemma~\ref{Lemma : form of covers} (see \cite[Proposition 3.20]{philippe2023grading} for a complete statement). Therefore our classification of quantum roots also provides insights on the covers for the affinized Bruhat order. In particular using Theorem~\ref{Theorem: finiteness quantum roots intro} we prove:

\begin{Theorem}\label{thm_finiteness_intro}(see Theorem~\ref{Theorem : finiteness of covers} and Theorem~\ref{thm_finiteness_co-covers})
   Let $\bx\in W^+$. Then $\bx$ admits finitely many covers. If furthermore $\bx$ is spherical, i.e the fixator in $W^v$ of the associated cocharacter  is finite, then $\bx$  also admits finitely many co-covers.
\end{Theorem}

\begin{Corollary}\label{Corollary: finiteness intervals introduction}
    Let $\bx,\by\in W^+$. Then the interval $[\bx,\by]$ is finite.
\end{Corollary}

In the reductive case, these theorems are easily obtained by considering reduced decompositions. In the Kac-Moody framework these theorems were first proved by A. Welch in the untwisted affine ADE case using different methods  in \cite{welch2022classification}.

When $W^v$ admits an infinite proper parabolic subgroup (this implies that $\G$ is not affine), we also prove the existence of elements of $W^+$ admitting infinitely many co-covers, see Lemma~\ref{l_reciprocal_thm_cocovers}.

Our main motivation for the study of the combinatorics of $W^+$ is the construction of an affine Kazhdan-Lusztig theory for Kac-Moody groups. This theory was initiated by Muthiah in \cite{muthiah2019double}, where he uses Corollary~\ref{Corollary: finiteness intervals introduction}   to give a conjectural definition of spherical $R$-polynomials  in the affine ADE case.  We  use our results to define an Iwahori version of Muthiah's polynomials in the general Kac-Moody framework (see \cite{hebert2024affine}).

\paragraph{Organization of the paper}

In section~\ref{s_almost_simple_roots}, we define the quantum roots and we prove that the set of quantum roots of a Kac-Moody datum is finite (see Theorem~\ref{t_finiteness_almost_simple_roots}). We also obtain a general classification of quantum roots (see Theorem~\ref{Theorem: Classification quantum roots}). We then describe the set of quantum roots in particular cases. 

In section~\ref{s_almost_simple_roots_covers}, we explain the link between covers for the affinized Bruhat order and quantum roots. We then prove Theorem~\ref{thm_finiteness_intro}. We also show that the spherical condition in Theorem~\ref{thm_finiteness_intro} is necessary by constructing non-spherical elements with infinitely many co-covers (see Proposition~\ref{Proposition: nonspherical infinite cocovers}). We end the section with an explicit construction of covers associated to any given quantum root, thus showing that the notion of quantum root is well-suited. Note that this section uses the finiteness of the set of quantum roots (Theorem~\ref{Theorem: finiteness quantum roots intro}) but is relatively independent of Section~\ref{s_almost_simple_roots}: for the reader ready to admit this theorem,  the only prerequisites to this section are contained in Subsection~\ref{ss_gen_def_not}.

\paragraph{Acknowledgement}
 We thank Dinakar Muthiah and Stéphane Gaussent for helpful conversations on the subject. We thank Pierre-Emmanuel Chaput for indicating to us the notion of minuscule elements and for discussions on this topic. We  thank the referees for their careful reading, valuable comments and suggestions.

\tableofcontents

\section{Quantum roots}\label{s_almost_simple_roots}

\subsection{General definitions and notations}\label{ss_gen_def_not}

We let $\N\subset \mathbb \Z$ denote the set of non-negative integers. For $a,b\in \Z$, we denote by $\llbracket a,b\rrbracket$ the integer interval $[a,b]\cap \Z$.

\paragraph{Generalized Cartan matrices and Dynkin Diagrams}
\begin{Definition}
 Let $I$ be a finite set. A \textbf{generalized Cartan matrix} over $I$ is a matrix $A=(a_{i,j})_{(i,j)\in I\times I}\in M_I(\mathbb Z)$ with integral coefficients satisfying the following properties:
\begin{enumerate}
    \item $\forall i\in I \; a_{i,i}=2$
    \item $\forall (i,j)\in I^2,\; i\neq j \implies a_{i,j}\leq 0$
    \item $\forall (i,j)\in I^2,\; a_{i,j}=0 \implies a_{j,i}= 0$.
\end{enumerate}
   
\end{Definition}
\begin{Definition}
    A \textbf{Dynkin diagram} is a tuple $(I,E,w)$ where $I$ is a finite set (its vertices), $E$ is a symmetric subset of $I\times I \setminus \{(i,i)\mid i \in I\}$ (its edges) and $w$ is an application $E \rightarrow \mathbb Z_{>0}$ (its weight function). In other words, it is a positively weighted oriented graph such that $(i,j) \in E \iff (j,i) \in E$. 

    Let $\Gamma=(I,E,w)$ be a Dynkin diagram. Then any subset $J\subset I$ admits an induced structure of Dynkin diagram: $(J,E\cap (J\times J), w|_J)$, where $w|_J$ is the restriction of $w$ to $E\cap (J\times J)$. By subdiagram (or subgraph) of $\Gamma$, we mean any Dynkin diagram of this form. 

    A \textbf{leaf} in a Dynkin diagram is any vertex which lies in at most one unoriented edge. That is to say, it is any $i\in I$ for which there is at most one $j\in I$ such that $(i,j)\in E$.
\end{Definition}

Dynkin diagrams are graphic representations of generalized Cartan matrices: Suppose that $A=(a_{i,j})_{(i,j)\in I\times I}$ is a generalized Cartan matrix indexed by $I$. Then, setting $E=\{(i,j)\in I^2 \mid a_{i,j}<0\}$ and $w: (i,j)\mapsto -a_{i,j}$, we obtain a Dynkin diagram $(I,E,w)$. Conversely, to any Dynkin diagram $(I,E,w)$ we associate a generalized Cartan matrix $(a_{i,j})_{(i,j)\in I\times I }$ indexed by $I$, setting \[a_{i,j}=\begin{cases}2 &\text{ if } i=j \\
    -w(i,j) &\text{ if } (i,j)\in E \\
    
    0 &\text{ otherwise.}
\end{cases}\]

\begin{Definition}[Connectivity and $1$-star-convexity]
    A Dynkin diagram $(I,E,w)$ is said to be \textbf{connected} if it is connected as a graph. That is to say if, for any $i,j\in I$ there is a sequence $i=i_0,i_1,\dots,i_n=j$ such that $(i_k,i_{k+1})\in E$ for all $k\in \llbracket0,n-1\rrbracket$. 
    
    For a general Dynkin diagram $\Gamma$, a \textbf{connected component} of $\Gamma$ is any connected subdiagram which is maximal for the inclusion. Therefore a Dynkin diagram is connected if and only if it has exactly one connected component.

    A connected Dynkin diagram is a \textbf{Dynkin tree} if it does not admit any closed circuit: For any sequence $i_0,\dots, i_n$ such that $(i_k,i_{k+1})\in E$ for all $k\in \llbracket0,n-1\rrbracket$ and $i_{k-1}\neq i_{k+1}$ for all $k \in \llbracket1,n-1\rrbracket$, we have $i_0\neq i_n$. A Dynkin tree is a \textbf{Dynkin segment} if there are at most two edges going out of any given vertex. 
    
    We introduce a refinement of the notion of connectivity:   We say that a Dynkin diagram $\Gamma=(I,E,w)$ is \textbf{$1$-star-convex} if there exists a vertex $i_0\in I$ such that, for any $j\in I$, there is a sequence $(i_1,\dots,i_n)$ with $i_n=j$ such that, for all $k\in \llbracket0,n-1\rrbracket$,
    \begin{equation}
        (i_k,i_{k+1})\in E \text{ and } w(i_k,i_{k+1})=1.
    \end{equation}
In this case, we say that $\Gamma$ is $1$-star-convex at $i_0$.

\end{Definition}

    By symmetry of $E$, a $1$-star-convex Dynkin diagram is always connected, but the converse is not true.

\paragraph{Kac-Moody root systems}
Let $\mathcal D = (A,X,Y,(\alpha_i)_{i \in I},(\alpha_i^\vee)_{i \in I})$ be a Kac-Moody root datum  in the sense of \cite[\S 8]{remy2002groupes}. It is a quintuplet such that:
\begin{itemize}
    \item $I$ is a finite indexing set and $A=(a_{ij})_{(i,j)\in I\times I}$ is a generalized Cartan matrix.
    \item $X$ and $Y$ are two dual free $\mathbb Z$-modules of finite rank, we write $\langle , \rangle$ the duality bracket.  
    \item $(\alpha_i)_{i\in I}$ (resp. $(\alpha_i^\vee)_{i \in I}$) is a family of linearly independent elements of $X$ (resp. $Y$), the simple roots (resp. simple coroots).
    
    \item For all $(i,j) \in I^2$ we have $\langle \alpha_i^\vee,\alpha_j \rangle = a_{ij}$.
\end{itemize}
Let $(I,E,w)$ be the Dynkin diagram corresponding to $A$, with vertices $I$, which is fixed once and for all. In this article, every Dynkin diagram we will consider will be a subdiagram of $(I,E,w)$. Since they are identified with subsets of $I$, any $J\subseteq I$ now designates indistinctively the subset $J$ of $I$, or the corresponding subdiagram of $(I,E,w)$.

\paragraph{Height function}
We define a height function on $Y$ as follows: Let $\rho\in X\otimes \Q$ be any element such that $\langle \alpha_i^\vee,\rho\rangle =1$ for all $i\in I$, then, for any $\lambda\in Y$, the \textbf{height} of $\lambda$ is: \index{h@$\htt$} $\htt(\lambda)=\langle \lambda,\rho\rangle$. This definition depends on the choice of $\rho$, but its restriction to $Q^\vee=\bigoplus\limits_{i\in I} \mathbb Z\alpha_i^\vee$ does not: $\htt(\sum_{i\in I} n_i \alpha_i^\vee)=\sum_{i\in I} n_i$.

\paragraph{Vectorial Weyl group} For every $i \in I$ set $r_i \in \operatorname{Aut}_\mathbb Z(X):  x \mapsto x- \langle \alpha_i^\vee, x \rangle \alpha_i$. The generated group \index{w@$W^v$} $W^v=\langle r_i \mid i \in I \rangle$ is the \textbf{vectorial Weyl group} of the Kac-Moody root datum. The duality bracket $\langle Y,X\rangle$ induces a contragredient action of $W^v$ on $Y$, explicitly $r_i(y)=y-\langle y,\alpha_i\rangle \alpha_i^\vee$. By construction the duality bracket is then $W^v$-invariant.

The group $W^v$ is a Coxeter group, in particular it has a Bruhat order $<$ and a length function $\ell$ compatible with the Bruhat order.

\paragraph{Real roots} Let $\Phi=W^v.\{\alpha_i \mid i \in I \}$ be the set of real roots of $\mathcal D$, it is a, possibly infinite, root system (see \cite[1.2.2 Definition]{kumar2002kac}).
In particular let $\Phi_+=\Phi\cap \oplus_{i\in I} \mathbb N \alpha_i$  be the set of positive real roots, then $\Phi=\Phi_+ \sqcup -\Phi_+$. We write $\Phi_-=-\Phi_+$ the set of negative roots.

The set $\Phi^\vee=W^v.\{\alpha_i^\vee \mid i \in I\}$ is the set of \textbf{coroots}, and its subset  $\Phi^\vee_+=\Phi^\vee \cap \oplus_{i\in I} \mathbb N \alpha_i^\vee$ is the set of \textbf{positive coroots}.

To each root $\beta$ corresponds a unique coroot $\beta^\vee$: if $\beta=w(\alpha_i)$ then $\beta^\vee=w(\alpha_i^\vee)$. This map $\beta\mapsto\beta^\vee$ is well defined, bijective between $\Phi$ and $\Phi^\vee$ and sends positive roots to positive coroots. Note that $\langle \beta^\vee,\beta\rangle = 2$ for all $\beta \in \Phi$.

Moreover to each root $\beta$ one associates a reflection $s_\beta \in W^v$: if $\beta=w(\pm\alpha_i)$ then $s_\beta:= wr_iw^{-1}$\index{s@$s_\beta$}. It is well-defined, independently of  the choices of $w$ and $i$. Explicitly it is the map $x \mapsto x-\langle \beta^\vee , x \rangle \beta$. We have $s_\beta=s_{-\beta}$ and the map $\beta\mapsto s_\beta$ forms a bijection between  the set of positive roots and the set $\{wr_iw^{-1}\mid (w,i)\in W^v\times I\}$ of reflections of $W^v$.

\paragraph{Dynkin sequence associated to a positive root}
Let $\beta\in \Phi_+$. Since $(\alpha^\vee_i)_{i\in I}$ are linearly independent, there is a unique decomposition $\beta^\vee= \sum_{i\in I} N_i(\beta)\alpha_i^\vee$, and for all $i \in I$, $N_i(\beta)\in \mathbb Z_{\geq0}$. For any $n \in \Z_{\geq 0}$, let us define
\begin{equation}\label{eq: I_n definition}
    I_n(\beta)=\{i\in I \mid N_i(\beta) \geq n \}.
\end{equation} We consider $I_n(\beta)$ as a subdiagram of $\Gamma$. We call $(I_n(\beta))_{n\geq 1}$ the \textbf{Dynkin sequence} of $\beta$. The Dynkin sequence associated to a root is, by definition, non-increasing (for the inclusion) and finitely supported (in the sense that $I_n=\emptyset$ for $n$ large enough). Note that the Dynkin sequence of a simple root $\alpha_i$ is given by $I_1(\alpha_i)=\{i\}$ and $I_n(\alpha_i)=\emptyset$ for $n\geq 2$.
We will provide a classification of quantum roots using Dynkin sequences.

Note that the Dynkin sequence $(I_n(\beta))_{n\geq 1}$ fully determines $\beta$. Indeed, we have $\beta^\vee=\sum_{n=1}^{+\infty} \sum_{i\in I_n(\beta)}\alpha_i^\vee$, which proves that $\beta^\vee$ and thus $\beta$ is determined by its Dynkin sequence.

\paragraph{Inversion sets}
For any $w \in W^v$, let $\Inv(w)=\Phi_+ \cap w^{-1}.\Phi_-= \{ \alpha \in \Phi_+ \mid w(\alpha) \in \Phi_-\}$\index{i@$\Inv(w)$}. Theses sets are strongly connected to the Bruhat order, as by \cite[1.3.13]{kumar2002kac},
$$\forall \alpha\in\Phi_+, \; \alpha \in \Inv(w) \iff ws_\alpha < w \iff s_\alpha w^{-1} < w^{-1}.$$
Moreover, they are related to the length $\ell$ on $W$: $\ell(w)=|\Inv(w)|$ (\cite[1.3.14]{kumar2002kac}).

\subsection{Quantum roots: definition,  motivation and notations}\label{ss_Quantum_roots}

\begin{Definition}\label{d_almost_simple_root}
    Let $\beta\in \Phi_+$. We say that $\beta$ is \textbf{quantum} if for every $\gamma\in \Inv(s_\beta)\setminus \{\beta\}$, we have $\langle \beta^\vee,\gamma\rangle=1$.  We denote by $\cQ(\Phi_+)$\index{q@$\cQ(\Phi_+)$} the set of quantum roots of $\Phi_+$. \
\end{Definition}

 Note that $\Inv(r_i)=\{\alpha_i\}$ for every $i\in I$ and thus every simple root is quantum.
\begin{Remark}
    This definition is not the one we presented in the introduction. By Lemma~\ref{l_characterization_almost_simpleness} below, the two definitions are equivalent.
\end{Remark}

Note that the notion of  quantum root is related to the notion of minuscule element attributed to  Peterson (see \cite[2]{stembridge2001minuscule}). J. Stembridge assumes the generalized Cartan matrix to be symmetrizable, but this assumption is not necessary to give a sense to this definition, so we drop it here.
\begin{Proposition}\label{p_minuscule}
\begin{enumerate}
    \item Let $\beta\in \Phi_+$. Write $\beta=w(\alpha_i)$,  with $i\in I$ and  $w\in W^{\vv}$ of minimal possible length. Then $\{\langle \beta^\vee,\gamma\rangle\mid \gamma\in \Inv(s_\beta)\setminus \{\beta\}\}=\{-\langle \alpha_i^\vee,\gamma\rangle\mid \gamma\in \Inv(w)\}$. 

    \item  Let $\beta\in \Phi_+$. Write $\beta=w(\alpha_i)$, with $i\in I$ and $w\in W^{\vv}$ of minimal possible length. Then $\beta\in \cQ(\Phi_+)$ if and only if $w$ is $\alpha_i^\vee$-minuscule in the sense of \cite[(2.1)]{stembridge2001minuscule}.
\end{enumerate}
    
\end{Proposition}

\begin{proof}
   1) Write $w=r_{i_n}\ldots r_{i_1}$, with $n=\ell(w)$ and $i_1,\ldots,i_n\in I$. Then by \cite[1.3.14 Lemma]{kumar2002kac}, we have: \[\Inv(s_\beta)=\{\alpha_{i_n},r_{i_n}(\alpha_{i_{n-1}}),\ldots, r_{i_n}\ldots r_{i_2}(\alpha_{i_1}),w(\alpha_i), wr_i(\alpha_{i_1}),\ldots, w r_{i}r_{i_1}\ldots r_{i_{n-1}}(\alpha_i)\}.\] Set $E_1=\{r_{i_n}\ldots r_{k}(\alpha_{i_{k-1}})\mid k\in \llbracket 2,n+1\rrbracket\}$ and $E_2=\{w r_ir_{i_1}\ldots r_{i_{k-1}}(\alpha_{i_{k-1}})\mid k\in \llbracket 0,n-1\rrbracket\}$. Then $\Inv(s_\beta)=E_1\sqcup \{\beta\}\sqcup E_2.$ 
    
    Let $\gamma \in E_1$. Write $\gamma=r_{i_n}\ldots r_{i_k}(\alpha_{i_{k-1}})$, where $k\in \llbracket 2,n+1\rrbracket$. Then $\langle \beta^\vee,\gamma\rangle=\langle w(\alpha_i^\vee),\gamma\rangle=\langle \alpha_i^\vee,w^{-1} r_{i_n}\ldots r_{i_{k}}(\alpha_{i_{k-1}})\rangle=\langle \alpha_i^\vee, r_{i_1}\ldots r_{i_{k-1}}(\alpha_{i_{k-1}})\rangle=-\langle \alpha_i^\vee,r_{i_1}\ldots r_{i_{k-2}}(\alpha_{i_{k-1}})\rangle$. Therefore $\{\langle \beta^\vee,\gamma\rangle \mid \gamma\in E_1\}=\{-\langle \alpha_i^\vee,\gamma\rangle \mid \gamma\in \Inv(w)\}$. 
    
    Let $\gamma$ in $E_2$. Write $\gamma=wr_i r_{i_1}\ldots r_{i_{k-1}}(\alpha_{i_{k-1}})$, with $k\in \llbracket 0,n-1\rrbracket$. Then $\langle \beta^\vee,\gamma\rangle=\langle \alpha_i^\vee,r_{i} r_{i_1}\ldots r_{i_{k-1}}(\alpha_{i_{k-1}})\rangle=-\langle \alpha_i^\vee,r_{i_1}\ldots r_{i_{k-1}}(\alpha_{i_{k-1}})\rangle$. Therefore $\{\langle \beta^\vee,\gamma\rangle \mid \gamma\in E_2\}=\{-\langle \alpha_i^\vee,\gamma\rangle \mid \gamma\in \Inv(w)\}$. This proves 1). 

    2) If $k\in \llbracket 0,n-1\rrbracket$, then $\langle \alpha_i^\vee,r_{i_1}\ldots r_{i_{k-1}}(\alpha_{i_{k-1}})\rangle=\langle r_{i_k}\ldots r_{i_1}(\alpha_i^\vee),\alpha_{i_{k+1}}\rangle$ and thus Definition~\ref{d_almost_simple_root} for $w(\alpha_i)$ is equivalent to \cite[(2.1)]{stembridge2001minuscule} for $w$ and $\alpha_i^\vee$.
    
\end{proof}

\begin{Remark}
In \cite{proctor1999minuscule} and \cite{proctor1999dynkin}, R. Proctor classifies $\lambda$-minuscule elements, for $\lambda\in Y$ dominant (i.e $\langle\lambda,\alpha_i\rangle\geq 0$, for all $i\in I$), in the simply laced case. Here the case we study is a bit orthogonal: in general simple coroots are not dominant, and furthermore, when the generalized Cartan matrix is indecomposable not of finite type, they do not lie in the Tits cone. It would be interesting to have a theory unifying these two situations.
\end{Remark}

\medskip 

The aim of this section is to study quantum roots and in particular to prove that for every Kac-Moody datum, the number of quantum roots is finite (see Theorem~\ref{t_finiteness_almost_simple_roots}). This result is a key step to prove our finiteness results on covers and co-covers (see Theorems~\ref{Theorem : finiteness of covers} and \ref{thm_finiteness_co-covers}). We then classify the quantum roots. We end up this section with examples: we describe the set of quantum roots for particular choices of generalized Cartan matrices.

\medskip

Let us describe the main steps of the proofs of Theorem~\ref{t_finiteness_almost_simple_roots} and Theorem~\ref{Theorem: Classification quantum roots}. Let $\beta\in \cQ(\Phi_+)$.  We write \begin{equation}\label{e_reduced_writting_beta}
    \beta=r_{i_L}\ldots r_{i_2}(\alpha_{i_1}), 
\end{equation}  with $i_1, \ldots,i_L\in I$ and $L$ minimal possible. \begin{enumerate}
    \item We begin by proving that up to renumbering, we can assume that if $L_1=|I_1(\beta)|$, we have $\{i_1,\ldots,i_{L_1}\}=I_1(\beta)$ (see Lemma~\ref{l_I1_placed_beggining}). 

    \item We prove that $I_1(\beta)$ is a $1$-star convex tree (see Proposition~\ref{p_characterization_I1}).

    \item We prove that the "orders of appearance" in \eqref{e_reduced_writting_beta} are very constrained. If $i,j\in I_2(\beta)$ belong to the same connected component and if the second appearance of $i$ (when we run the right hand side of \eqref{e_reduced_writting_beta} from the right) precedes the second appearance of $j$, then for $n\in \Z_{\geq 3}$, \[j\in I_n(\beta)\Rightarrow i\in I_n(\beta),\] and the $n$-th appearance of $i$ precedes the $n$-th appearance of $j$ (see Lemma~\ref{l_description_chains_In}).

    \item We prove that the connected components of $I_2(\beta)$ are as follows. Let $\cC$ be a connected component of $I_2(\beta)$ and $j$ be the element of $\cC$ whose second appearance in \eqref{e_reduced_writting_beta} is the earliest. Then $\cC$ is the union of at most three segments based at $j$ (see Lemma~\ref{l_description_chains_I2} and Proposition~\ref{p_connected_components_I2}). 

    \item Using the results described above, we  obtain restrictions on the form of the Dynkin sequence $(I_n(\beta))$. We deduce  that if $k=\max \{n\in \Z\mid I_n(\beta)\neq \emptyset\}$, then $k\leq \max(6,k'+1)$ where $k'$  is the maximal length of a segment appearing in point 4 (see Proposition~\ref{p_majoration_coefficients_AS_roots}). As $k'$ is bounded by the size of $I$, we deduce   a majoration of the heigth of $\beta^\vee$ and then deduce that $\cQ(\Phi_+)$ is finite.

    \item Refining the results of 5., we describe all the possible forms of the Dynkin sequences. We then  construct for all these Dynkin sequences a quantum root to which it is associated (see Theorem~\ref{Theorem: Classification quantum roots}).
\end{enumerate}

Let $\beta\in \Phi_+$.  We write $\beta=w(\alpha_{i_1})$, with $i_1\in I$ and $w\in W^v$ and we assume that $L:=\ell(w)+1$\index{l@$L$}  is the minimal possible. We fix a reduced decomposition $\underline{w}=(r_{i_L},\ldots ,r_{i_2})$ of $w$, with $i_L,\ldots,i_2\in I$.

By \cite[1.3.14 Lemma]{kumar2002kac} we have \begin{equation}\label{e_Inversion_set}
    \Inv(s_\beta)=\{\alpha_{i_L},r_{i_{L}}(\alpha_{i_{L-1}}),\ldots,r_{i_L}\ldots r_{i_{2}}(\alpha_{i_1}),r_{i_L}\ldots r_{i_1}(\alpha_{i_2}),\ldots, r_{i_L}\ldots r_{i_1} \ldots r_{i_{L-1}}(\alpha_{i_L})\}.
\end{equation}

 Set \[\gamma^\vee_1=\alpha_{i_1}^\vee,\ \gamma^\vee_2=r_{i_2}(\alpha_{i_1}^\vee),\  \ldots,\ \gamma^\vee_L=r_{i_L}\ldots r_{i_2}(\alpha_{i_1}^\vee)=\beta^\vee.\]

\begin{Lemma}\label{l_inequality_length}
    Let $\beta\in \Phi_+$. Then $\ell(s_\beta)\leq 2\htt(\beta^\vee)-1$.
\end{Lemma}

\begin{proof}
    By \cite[Corollary 1.10]{philippe2023grading}, we have $2\htt(\beta^\vee)=2+\sum\limits_{\gamma\in \Inv(s_\beta)\setminus\{\beta\}} \langle \beta^\vee,\gamma\rangle$ and all the terms in the right hand side are positive integers. Therefore $2\htt(\beta^\vee)-2\geq |\Inv(s_\beta)|-1=\ell(s_\beta)-1$. 
\end{proof}

\begin{Lemma}\label{l_characterization_almost_simpleness}
Let $\beta\in \Phi_+$.  Using the same notation as above, set $\beta_1=r_{i_L}\ldots r_{i_3}(\alpha_{i_2})$, $\beta_2=r_{i_L}\ldots r_{i_4}(\alpha_{i_3}),\ldots, \beta_{L-1}=\alpha_{i_L}$. Then the following conditions are equivalent: \begin{enumerate}
    \item $\beta$ is quantum,
    
    \item for every $t\in \llbracket 1,L-1\rrbracket$, $\langle \beta^\vee,\beta_t\rangle=1$,

    \item \begin{equation}\label{e_roots}
\forall t\in \llbracket 1,L\rrbracket, \ \gamma^\vee_t=\sum_{k=1}^t \alpha_{i_k}^\vee \text{ and }\langle \gamma_{t-1}^\vee,\alpha_{i_t}\rangle=-1,\text{ if }t\neq 1, 
\end{equation} 
\item $\ell(s_\beta)=2\htt(\beta^\vee)-1$.
\end{enumerate}
\end{Lemma}

\begin{proof}
    Let $t\in \llbracket 2,L\rrbracket$. We have: \[\langle \beta^\vee,\beta_{t-1}\rangle=\langle r_{i_L}\ldots r_{i_2}(\alpha_{i_1}^\vee),r_{i_L}\ldots r_{i_{t+1}}(\alpha_{i_{t}})\rangle=\langle r_{i_t}\ldots r_{i_2}(\alpha_{i_1}^\vee),\alpha_{i_t}\rangle=-\langle r_{i_{t-1}}\ldots r_{i_2}(\alpha_{i_1}^\vee),\alpha_{i_t}\rangle.\]

    We thus have: \[\begin{aligned}\langle r_{i_L}\ldots r_{i_2}(\alpha_{i_1}^\vee),r_{i_L}\ldots r_{i_1} \ldots r_{i_{t-1}}(\alpha_{i_t})\rangle &=\langle \alpha_{i_1}^\vee,r_{i_1}\ldots r_{i_{t-1}}(\alpha_{i_t})\rangle\\
    &=-\langle r_{i_{t-1}}\ldots r_{i_2}(\alpha_{i_1}^\vee),\alpha_{i_t}\rangle=\langle \beta^\vee,\beta_{t-1}\rangle.\end{aligned}\] 
 Therefore \[\{\langle\beta^\vee,\gamma\rangle\mid \gamma\in \Inv(s_\beta)\setminus\{\beta\}\}=\{\langle\beta^\vee,\beta_t\rangle\mid t\in \llbracket 2,L \rrbracket\},\] which proves the equivalence between (1) and (2).

    Let $t\in \llbracket 1,L-1\rrbracket$. We have: \[\langle \beta^\vee,\beta_t\rangle=\langle r_{i_L} \ldots r_{i_2}(\alpha_{i_1}^\vee),r_{i_L}\ldots r_{i_{t+2}}(\alpha_{i_{t+1}})\rangle=\langle r_{i_{t+1}}\ldots r_{i_2}(\alpha_{i_1}^\vee),\alpha_{i_{t+1}}\rangle=-\langle \gamma_t^\vee,\alpha_{i_{t+1}}\rangle.\] Therefore \[\gamma_{t+1}^\vee=r_{i_{t+1}}(\gamma_t^\vee)=\gamma_t^\vee-\langle \gamma_t^\vee,\alpha_{i_{t+1}}\rangle \alpha_{i_{t+1}}^\vee=\gamma_t^\vee+\langle \beta^\vee,\beta_t\rangle\alpha_{i_{t+1}}^\vee.\]

    The equivalence between (1) and (3) follows.

Let $\beta\in \Phi_+$ satisfying (3).    We have $\gamma_L^\vee=\beta^\vee$ and $\htt(\gamma_L^\vee)=L$. As $\ell(s_\beta)=2L-1$, we deduce that (3) implies (4). Let  $\beta^\vee\in \Phi^\vee_+$. Asume that $\beta^\vee$ satisfies (4). Then  by \cite[Corollary 1.10]{philippe2023grading}, we have $2\htt(\beta^\vee)=\sum_{\gamma\in \Inv(s_\beta)} \langle \beta^\vee,\gamma\rangle=2+\sum_{\gamma\in \Inv(s_\beta)\setminus\{\beta^\vee\}} \langle \beta^\vee,\gamma\rangle=\ell(s_\beta)+1$ and all the terms of this sum are positive. Therefore $\langle \beta^\vee,\gamma\rangle=1$ for all $\gamma\in \Inv(s_\beta)\setminus \{\beta\}$, which proves that $\beta$ is quantum.
    \end{proof}

\begin{Remark}\label{r_height_almost_simple_coroot}
By Lemma~\ref{l_characterization_almost_simpleness}, for all $t\in \llbracket 1,L\rrbracket$, $\gamma_t=r_{i_t}\dots r_{i_2}(\alpha_{i_1})$ is quantum and $\htt(\gamma_t^\vee)=t$.
\end{Remark}
The following proposition indicates that the set $\cQ(\Phi_+)$ can be constructed recursively. This will be used extensively in Section~\ref{subsection: classification} in order to prove Theorem~\ref{Theorem: Classification quantum roots}.

\begin{Proposition}\label{Proposition: simplereflectionquantumroot}

\begin{enumerate}
    \item If $\beta\in \cQ(\Phi_+)$ is a quantum root which is not simple, then there exists $i\in I$ such that $r_i(\beta)\in \cQ(\Phi_+)$ and $\htt(r_i(\beta^\vee))=\htt(\beta^\vee)-1$.
    \item Let $\beta\in \cQ(\Phi_+)$ and $i\in I$. Then $r_i(\beta)\in \cQ(\Phi_+)$ if and only if $|\langle \beta^\vee,\alpha_i\rangle|\leq 1$. 

\end{enumerate}

\end{Proposition}
\begin{proof}
1. We simply take $i=i_L$ with the notation of Formula \eqref{e_Inversion_set}, then by Remark~\ref{r_height_almost_simple_coroot}, $r_i(\beta)=\gamma_{L-1}$ is a quantum root.

2. Let $\beta\in \cQ(\Phi_+)$ and $i\in I$.  Let $\tilde{\beta}=r_i(\beta)$.  We have $s_{\tilde{\beta}}=r_i s_{\beta}r_i$ and thus: \begin{equation}\label{e_length}
   \ell(s_{\tilde{\beta}})\in \{\ell(s_\beta)-2,\ell(s_\beta),\ell(s_{\beta})+2\} .
\end{equation}  We have $\htt(\tilde{\beta^\vee})=\htt(\beta^\vee)-\langle \beta^\vee,\alpha_i\rangle$. If $r_i(\beta)\in \cQ(\Phi_+)$, then by Lemma~\ref{l_characterization_almost_simpleness}, we have $\ell(s_{\tilde{\beta}})=2\htt(\tilde{\beta}^\vee)-1=2\htt(\beta^\vee)-1-2\langle \beta^\vee,\alpha_i\rangle=\ell(s_\beta)-2\langle \beta^\vee,\alpha_i\rangle$. Using \eqref{e_length} we deduce that if $r_i(\beta)\in \cQ(\Phi_+)$, then $|\langle \beta^\vee,\alpha_i\rangle|\leq 1.$

Let now $\beta\in \cQ(\Phi_+)$ and $i\in I$. If $\langle \beta^\vee,\alpha_i\rangle=0$, then $r_i(\beta)=\beta\in \cQ(\Phi_+)$. Assume now $\langle \beta^\vee,\alpha_i\rangle =1$. Set $\tilde{\beta}=r_i(\beta)$. We have $\htt(\tilde{\beta}^\vee)=\htt(\beta^\vee)-1$ and by Lemma~\ref{l_inequality_length},  \[\ell(s_{\tilde{\beta}})\leq 2\htt(\tilde{\beta})-1=2\htt(\beta^\vee)-3=\ell(s_{\beta})-2.\] By \eqref{e_length} we deduce $\ell(s_{\tilde{\beta}})=\ell(s_{\beta})-2=2\htt(\tilde{\beta})-1$ which proves that $\tilde{\beta}=r_i(\beta)$ is quantum by Lemma~\ref{l_characterization_almost_simpleness}.

Finally, assume that $\langle \beta^\vee,\alpha_i\rangle = -1$. Let $\tilde \beta:=r_i(\beta)$. Note that the associated coroot is $\tilde \beta^\vee =r_i(\beta^\vee)=\beta^\vee+\alpha_i^\vee$. Therefore $\htt(\tilde \beta^\vee)=\htt(\beta^\vee)+1$ and, by Lemma~\ref{l_characterization_almost_simpleness}, $\tilde \beta$ is quantum if and only if $$\ell(s_{\tilde \beta})=\ell(r_i s_\beta r_i)=\ell(s_\beta)+2.$$
 Since $r_i$ is a simple reflection, $\{\ell(r_i s_\beta),\ell(s_\beta r_i)\}\subset \{\ell(s_\beta)+1,\ell(s_\beta)-1\}$. Moreover since $s_\beta$ is a reflection, $\ell(r_i s_\beta)=\ell((r_is_\beta)^{-1})=\ell(s_\beta r_i)$.
    
    We have $\langle \beta^\vee,\alpha_i\rangle <0$, so $s_\beta(\alpha_i)\in \Phi_+$. Hence $\alpha_i\notin \Inv(s_\beta)$, $s_\beta r_i>s_\beta$ and thus $\ell(s_\beta r_i)=\ell(r_i s_\beta)=\ell(s_\beta)+1$.
    
    Note that $s_\beta(\alpha_i)=\alpha_i-\langle \beta^\vee,\alpha_i\rangle \beta=\beta+\alpha_i$. Moreover since $\langle \beta^\vee,\alpha_i\rangle <0$, by \cite[Lemma 1.1.10]{bardy1996systemes} we also have $\langle \alpha_i^\vee,\beta\rangle <0$, that is to say $\langle \alpha_i^\vee,\beta\rangle+1 \leq 0$ (because it is an integer). Therefore, $$r_i s_\beta(\alpha_i)=r_i(\beta+\alpha_i)=r_i(\beta)-\alpha_i=\beta-(\langle \alpha_i^\vee,\beta\rangle +1)\alpha_i \in \Phi_+.$$
    So $\alpha_i\notin \Inv(r_i s_\beta)$ and $\ell(r_i s_\beta r_i)=\ell(r_i s_\beta)+1=\ell(s_\beta)+2$, which proves the proposition.
\end{proof}

From now on we assume that $\beta$ is quantum.

For $i\in I$ and $t \in \llbracket 1,L\rrbracket$, set $N_i(t)=|\{t'\in \llbracket 1,t\rrbracket \mid i_{t'}=i\}|$\index{n@$N_i(\cdot)$}. Then by  \eqref{e_roots}, \begin{equation}\label{e_coordinate_root}
\gamma^\vee_t=\sum_{i\in I_1} N_i(t) \alpha_i^\vee.
\end{equation}

Note that by \eqref{e_roots}, $N_i$ is non-decreasing for every $i\in I_1(\beta)$, and for every $t\in \llbracket 1,L-1\rrbracket$, $N_i(t)=N_i(t+1)$ for all but one element $k$ of $I_1(\beta)$, and we have $N_k(t+1)=N_k(t)+1$. 

Note that the Dynkin sequence of $\beta$ can be defined by $I_n(\beta):=\{i\in I\mid \max N_i\geq n\}$, with the induced diagram structure from $I$. The set $I_n(\beta)$ is also the set of elements which appear at least $n$ times in $(\underline{w},r_{i_1})$.

In this section, we have fixed a quantum root $\beta$, therefore to simplify notation we set \begin{equation}I_n:=I_n(\beta).\end{equation}\index{i@$I_1,I_2,I_n,I_n(\beta)$}

For $j\in I$ and $n\in \Z_{\geq 0}$, we set \begin{equation}\bt(j,n)=\min \{t\in \llbracket 1,L\rrbracket\mid N_j(t)=n\}-1\in \llbracket 0,L-1\rrbracket\cup \{\infty\}.\end{equation}\index{t@$\bt(\cdot,\cdot)$} We then have $N_j(\bt(j,n))=n-1$ and $N_j(\bt(j,n)+1)=n$, if $\bt(j,n)\neq \infty$.

\begin{Lemma}\label{l_I1_placed_beggining}
Let $L_1=|I_1|$\index{l@$L_1$}. Then up to changing the reduced decomposition of $w$, we may assume $I_1=\{i_1,\ldots,i_{L_1}\}$. 
\end{Lemma}

\begin{proof}
Assume $\{i_1,\ldots,i_{L_1}\}\neq I_1$. Then $(i_j)_{j\in \llbracket 1,L_1\rrbracket}$ is not injective.  Let $k\in \llbracket 1,L_1\rrbracket$ be minimal such that $i_k\in \{i_1,\ldots, i_{k-1}\}$. Let $k'\in \llbracket k+1,L\rrbracket$ be minimal such that $i_{k'}\notin \{i_1,\ldots,i_{k'-1}\}$. Then for every $k''\in \llbracket k,k'-1\rrbracket$, $N_{i_{k''}}(k')=N_{i_{k''}}(k'-1)\geq 2$ and $N_{i_{k'}}(k'-1)=0=N_{i_{k'}}(k')-1$. By assumption on $\beta$, we have  $\langle r_{i_{k'-1}}\ldots r_{i_2}(\alpha_{i_1}^\vee), \alpha_{i_{k'}}\rangle=-1$. Using \eqref{e_coordinate_root}, we deduce $\sum_{i\in I_1} N_i(k'-1) a_{i,i_{k'}}=-1$. As $N_{i_{k'}}(k'-1)=0$, all the terms appearing in the left hand side are non positive integers. Therefore only one is nonzero and equal to $-1$. Therefore $a_{i_{k''},i_{k'}}=0$ for every $k''\in \llbracket k, k'-1\rrbracket$. In other words, $r_{i_{k''}}$ commutes with $r_{i_{k'}}$ for every $k''\in \llbracket k,k'-1\rrbracket$. Therefore we can replace $r_{i_{k'}}r_{i_{k'-1}}\ldots r_{i_k} r_{i_{k-1}}\ldots r_{i_2}$ by $r_{i_{k'-1}}\ldots r_{i_k} r_{i_{k'}} r_{i_{k-1}}\ldots r_{i_2}$ in the reduced decomposition of $w$. Reiterating this process as many times as necessary, we end up with a reduced decomposition $\underline{w}$ such that $(i_j)_{j\in \llbracket 1,L_1\rrbracket}$ is injective and then $\{i_1,\ldots,i_{L_1}\}=I_1$. 
\end{proof}

From now on, we assume that the reduced decomposition $\underline{w}$ is chosen  so that $\{i_1,\ldots,i_{L_1}\}=I_1$, which is possible by Lemma~\ref{l_I1_placed_beggining}. 

\subsection{General structure of $I_1(\beta)$}
In this section, we study the Dynkin diagram structure of $I_1(\beta)$ when $\beta$ is a quantum root. We have fixed a quantum root $\beta$ and we have set $I_1=I_1(\beta)$.

For $i\in I$,  set \begin{equation}\label{e_definition_degree}
    \deg(i)=-\sum_{j\in I_1\setminus \{i\}} a_{j,i}\in \Z_{\geq 0}\end{equation}\index{d@$\deg$} and \[\supp(i)=\{j\in I_1\mid a_{j,i}\neq 0\}.\]\index{s@$\supp$}

\begin{Lemma}\label{l_I1_tree}
The graph $I_1$ is a tree: it is connected and without circuit.
\end{Lemma}

\begin{proof}
By \cite[Lemma 1.6]{kac1994infinite}, $I_1$ is connected ($I_1$ is the support of $\beta$ for the terminology of \cite{kac1994infinite}). Assume by contradiction that $I_1$ contains a circuit. Let $n\in \Z_{>0}$ and $\tau:\llbracket 0,n\rrbracket\rightarrow I_1$ be such that $\tau(0)=\tau(n)$ and $\{\tau(i) ,\tau(i+1)\}$ is an edge of $I_1$ for all $i\in \llbracket 0,n-1\rrbracket$. We assume moreover that $\tau(i)\neq \tau(0)$ for $i\in \llbracket 1,n-1\rrbracket$. Then every element of $\tau(\llbracket 0,n\rrbracket)$  has at least two neighbours in $I_1$. Let $i\in \llbracket 0,n-1\rrbracket$ be such that $\bt(\tau(i),1)=t:=\max\{\bt(\tau(j),1)\mid j\in \llbracket 0,n-1\rrbracket\}$. Then there exist $i_1,i_2\in I_1\setminus \{\tau(i)\}$ such that $a_{i_1,\tau(i)}, a_{i_2,\tau(i)}\neq 0$ and  $\gamma_t^\vee=\alpha_{i_1}^\vee+\alpha_{i_2}^\vee+x$, for  some $x\in \bigoplus_{j\in I_1\setminus \{\tau(i)\}}\N \alpha_j^\vee$. Then $\alpha_{\tau(i)}(\gamma^\vee_t)\leq \alpha_{\tau(i)}(\alpha_{i_1}^\vee+\alpha_{i_2}^\vee)\leq -2<-1$: a contradiction. Lemma follows.
\end{proof}

A first step in order to determine $\cQ(\Phi_+)$ is to understand what the support of quantum roots can be. That is to say, to determine $\{I_1(\beta) \mid \beta \in \cQ(\Phi_+)\}$. Proposition~\ref{p_characterization_I1} below characterizes this set.

\begin{Proposition}\label{p_characterization_I1}
    Let $T$ be a subset of $I$. Then the following properties are equivalent:
    \begin{enumerate}
        \item As a subdiagram of $I$, $T$ is a $1$-star-convex Dynkin tree.
        \item There exists $\beta \in \cQ(\Phi_+)$ such that $T=I_1(\beta)$.
        \item There exists $\beta \in \cQ(\Phi_+)$ such that $T=I_1(\beta)$ and $I_2(\beta)=\emptyset$. More precisely, $\beta$ is the unique element of $\Phi$ such that $\beta^\vee=\sum_{i\in T} \alpha_i^\vee$. 
    \end{enumerate}
\end{Proposition}

\begin{proof} 
 (2) $\Rightarrow (1)$  Let $\beta\in \cQ(\Phi_+)$ and $T=I_1(\beta)$. By Lemma~\ref{l_I1_placed_beggining} and Remark~\ref{r_height_almost_simple_coroot}, $\beta'^\vee:=\sum_{i\in T}\alpha_i^\vee$ is a quantum root. Therefore up to replacing $\beta$ by 
 $\beta'$, we can assume $\beta^\vee=\sum_{i\in T}\alpha_i^\vee$. Write $\beta=r_{i_L}\ldots r_{i_2}(\alpha_{i_1})$, where $L=|T|$ and $\{i_L,\ldots,i_1\}=T$. Let us prove that $T$ is $1$-star-convex at $i_1$. Let $t\in\llbracket 1,L-1\rrbracket$ and assume that, as a subdiagram of $T$, $\{i_1,\ldots,i_t\}$ is $1$-star-convex at $i_1$.  Then we have  $r_{i_t}\ldots r_{i_2}(\alpha_{i_1}^\vee)=\sum_{s=1}^{t} \alpha_{i_{s}}^\vee$ and  $r_{i_{t+1}}.(r_{i_t}\ldots r_{i_2}(\alpha_{i_1}^\vee))= (r_{i_t}\ldots r_{i_2}(\alpha_{i_1}^\vee))+\alpha_{i_{t+1}}^\vee$, by Lemma~\ref{l_characterization_almost_simpleness}.  Therefore $\langle \sum_{s=1}^{t} \alpha_{i_{s}}^\vee,\alpha_{i_{t+1}}\rangle=\sum_{s=1}^{t}\langle \alpha_{i_{s}}^\vee,\alpha_{i_{t+1}}\rangle=-1$. As this is a sum of non-positive integers, exactly one term is $-1$ and the others are $0$. Let $t_1\in \llbracket 1,t\rrbracket$ be such that $a_{i_{t_1},i_{t+1}}=-1$. By assumption, there exist $t'\in \llbracket 1,t\rrbracket$ and a sequence $(u_i)_{i\in \llbracket 1,t'\rrbracket}\in \{i_1,\ldots,i_t\}^{\llbracket 1,t'\rrbracket}$ such that $a_{u_i,u_{i+1}}=-1$ for all $i\in \llbracket 1,t'-1 \rrbracket$, $u_1=i_1$ and $u_{t'}=i_{t_1}$. Then $((u_i)_{i\in \llbracket 1,t'\rrbracket},i_{t+1})$ is a path in $I_1$ relating $i_1$ and $i_{t+1}$, which proves the $1$-star-convexity of $\{i_1,\ldots,i_{t+1}\}$ at $i_1$. By induction, we deduce that $T$ is $1$-star-convex. Using Lemma~\ref{l_I1_tree}, we deduce that  $I_1$ is a $1$-star-convex tree.

(1) $\Rightarrow$ (3)  Let $T$ be a subdiagram of $I$ which is $1$-star-convex and which is a tree. We assume that $|T|\geq 2$ and that for all $T'\subset I$ such that $|T'|<|T|$ and for which $T'$ is $1$-star-convex, there exists $\beta'\in \cQ(\Phi_+)$ such that $T'=I_1(\beta')$.

Let $j$ be a leaf of $T$. Since a finite tree with at least two vertices has at least two leaves, upon changing $j$ we can suppose that $T$ is $1$-star-convex at a node $j_0$ distinct from $j$. Write $\{j,k\}$ the only edge containing $j$, then any path from $j_0$ to $j$ contains $\{k,j\}$ and so $a_{k,j}=-1$ by $1$-star-convexity of $T$ at $j_0$. Moreover, $T\setminus \{j\}$ is $1$-star-convex (at $j_0$). Let $\beta_1\in \cQ(\Phi_+)$ be such that $(\beta_1)^\vee=\sum_{i\in T\setminus \{j\}} \alpha_i^\vee$. We have $\langle (\beta_1)^\vee,\alpha_j\rangle=-1$, thus $\beta^\vee=r_j(\beta_1^\vee)$ and thus by Proposition~\ref{Proposition: simplereflectionquantumroot}, $\beta=r_j(\beta_1)$ is quantum. Therefore (1) implies (3). As (3) obviously  implies (2), we deduce the lemma.

\end{proof}

\subsection{Order of appearance and general structure of $I_2(\beta)$}

Let $n\in \N$. For $i,j\in I_n$, we write $i \lhd_n j$\index{z@$\lhd_n$, $\lhd_2$} if $(i,j)$ is an edge of $I_n$, and if $\bt(i,n)<\bt(j,n)$. The first condition means that $a_{i,j}\neq 0$ and the second condition simply means that if we run the decomposition of $w$ from the right, we meet $i$ $n$ times before we meet  $j$ for the $n$-th time. We denote by $\leq_n$\index{z@$\leq_n$} the preorder generated by $\lhd_n$. Note that for $i,j\in I_n$, if $i\leq_n j$ and $j\leq_n i$, then $i$ and $j$ appears at the same time for the $n$-th time when we run $(\underline{w},i_1)$ from the right, and thus $i=j$. Therefore $\leq_n$ is an order on $I_n$. Note that although this is not obvious from the definition, this order  is independent of the chosen decomposition of $w$, as we will see in Remark~\ref{r_connected_components_In}. As $I_n$ is finite, all the chains are contained in a maximal one. 

\medbreak

In this section, we use $\leq_2$ to characterize the diagram $I_2$. In Lemma~\ref{l_description_chains_In} we then describe chains for $\leq_n$. This will permit to characterize the diagrams $I_n$ in Section~\ref{subsection: Dynkin sequences quantum roots}.

\begin{Lemma}\label{l_description_chains_I2}
Assume $I_2\neq \emptyset$. Let $j_1,\ldots,j_k\in I_2$\index{j@$j_1,\ldots,j_k$}\index{k@$k$} be a maximal chain for $\leq_2$, i.e: \begin{enumerate}
\item $j_1$ is minimal in $I_2$ for $\leq_2$,

\item $j_k$ is maximal in $I_2$ for $\leq_2$,

\item $j_1\lhd_2 j_2\lhd_2 \ldots \lhd_2 j_k$. 
\end{enumerate}

Then: \begin{enumerate}
\item $\deg(j_1)=3$, $|\supp(j_1)|\leq 4$.

\item for every $t\in \llbracket 2,k-1\rrbracket$, $\supp(j_t)=\{j_{t-1},j_t,j_{t+1}\}$, $\deg(j_t)=2$ and $a_{j_{t-1},j_t}=a_{j_{t+1},j_t}=-1$.

\item We have $\deg(j_k)\leq 2$, $|\supp(j_k)|\leq 3$ and $\supp(j_k)\supset \{j_{k-1},j_k\}$. 
\begin{enumerate}
\item If $|\supp(j_k)|=3$, then $\supp(j_k)=\{j_{k-1},j_k,j_{k+1}\}$ for some $j_{k+1}\in I_1\setminus I_2$, and we have $a_{j_{k-1},j_k}=a_{j_{k+1},j_k}=-1$. \item If $|\supp(j_{k})|=\{j_{k-1},j_k\}$, then  \begin{equation}\label{e_passage_3_before_2}
a_{j_{k-1},j_{k}}=-1\text{ and }\bt(j_{k-1},3)<\bt(j_k,2)<\bt(j_{k-1},4).
\end{equation}\end{enumerate}
\end{enumerate}

In particular, if $j\in I_2$ is such that $|\supp(j)|=2$, then $j$ is either minimal or maximal for $\leq_2$.
\end{Lemma}

\begin{proof}
Let $t_1=\bt(j_1,2)$. By \eqref{e_coordinate_root}, we have: $\gamma^\vee_{t_1}=\sum_{i\in I_1} N_i(t_1)\alpha_i^\vee$ and $N_i(t_1)\geq 1$ for every $i\in I_1$ by choice of $\underline{w}$.  Then $\alpha_{i_{t_1}}(\gamma^\vee_{t_1})=-1$. Let $i\in I_1$ and assume that $N_i(t_1)\geq 2$. Then as $i{\not\!\! \lhd}_2 j_1$, we have $a_{i,j_1}=0$. Therefore \[\begin{aligned} \alpha_{j_1}(\gamma^\vee_{t_1})&=-1\\ &=\alpha_{j_1}(\sum_{i\in \supp(j_1)} N_i(t_1) \alpha_i^\vee)\\
&=\alpha_{j_1}(\sum_{i\in \supp(j_1)} \alpha_{i}^\vee)\\
&=2+\sum_{i\in \supp(j_1)\setminus\{j_1\}} a_{i,j_1}\\
&=2-\deg(j_1),\end{aligned}\] and thus $\deg(j_1)=3$.

Let now $x\in \llbracket 2,k\rrbracket$. Let $t_x=\bt(j_x,2)$. We have: \[\alpha_{j_x}(\gamma^\vee_{t_x})=-1=\alpha_{j_x}\left(\sum_{i\in \supp(j_x)} N_i(t_x) \alpha_i^\vee\right)=2+\sum_{i\in \supp(j_x)\setminus \{j_x\}}N_i(t_x)a_{i,j_x}.\] Therefore $\sum_{i\in \supp(j_x)\setminus \{j_x\}}N_i(t_x)a_{i,j_x}=-3$, and this is a sum of non positive integers. Moreover, as $N_{j_{x-1}}$ is non decreasing and $t_x>t_{x-1}$,  we have $N_{j_{x-1}}(t_x)\geq 2$. Consequently, $|\supp(j_i)|\leq 3$. Assume $x\leq k-1$.  As $j_{x-1}\lhd_2 j_x\lhd_2 j_{x+1}$, we have $\supp(j_x)=\{j_{x-1},j_x,j_{x+1}\}$ and $a_{j_{x+1},j_x}=a_{j_{x-1},j_x}=-1$. In particular,  $\deg(j_x)=2$. 

Assume now $x=k$. Then if $|\supp(j_k)|=3$, we have $a_{j_{k+1},j_k}=-1$, for some $j_{k+1}\in I_1$. If $j_{k+1}\in I_2$, then we would have $j_k\lhd_2 j_{k+1}$, which would contradict our maximality assumption. Therefore, $j_{k+1}\in  I_1\setminus I_2$. Assume $|\supp(j_k)|=2$. Then $\supp(j_k)=\{j_{k-1},j_k\}$. Let $t=\bt(j_k,2)$. Then \[\alpha_{j_k}(\gamma^\vee_t)=-1=\alpha_{j_k}(N_{j_{k-1}}(t)\alpha_{j_{k-1}}^\vee+\alpha_{j_k}^\vee)=2-3,\] which proves \eqref{e_passage_3_before_2}. 
\end{proof}

In terms of Dynkin diagrams, Lemma~\ref{l_description_chains_I2} implies that the form of the connected components of $I_2$ are very constrained, see Proposition~\ref{p_connected_components_I2} and Figure~\ref{f_conneceted_components_I2}. 

 \begin{figure}[h]
 \centering
 \includegraphics[scale=0.33]{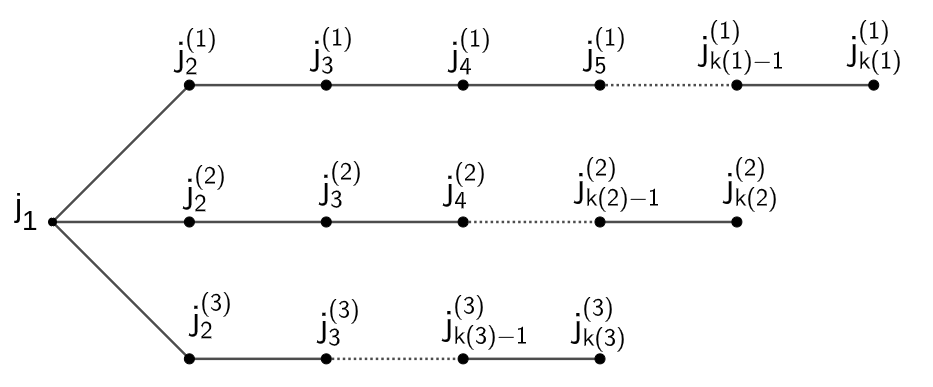}
 \caption{Generic form of a connected component of $I_2$, when $m=3$  . Actually, as we shall see in Lemma~\ref{Lemma: Class 4S}, if $j_1\in I_3$, then  we have $\min\{k(1),k(2),k(3)\}\leq 2$. }\label{f_conneceted_components_I2}
\end{figure}

\begin{Proposition}\label{p_connected_components_I2}
Let $\cC$ be a connected component of $I_2$. Then: \begin{enumerate}
    \item  $\cC$ admits a minimum $j_1$ and $m:=|\supp(j_1)|-1\leq 3$.

    \item Write $\supp(j_1)\setminus \{j_1\}=\{j_2^{(i)}\mid i\in \llbracket 1,m\rrbracket\}$. For $i\in \llbracket 1,m\rrbracket$, $\{x\in I_2\mid x\geq_2 j_2^{(i)}\}$ is a segment: we can write it $\{j_2^{(i)},j_3^{(i)},\ldots, j_{k(i)}^{(i)}\}$, for some $j_2^{(i)},\ldots,j_{k(i)}^{(i)}\in I_2$ and the edges  are the $(j_t^{(i)},j_{t+1}^{(i)})$ for $t\in \llbracket 1,k(i)-1\rrbracket$.
    
    \item We have  $\cC=\{j_1\}\sqcup \bigsqcup_{i=1}^m \{j_2^{(i)},j_{3}^{(i)},\ldots, ,j_{k(i)}^{(i)}\}$.

\item If $j\in \cC\setminus\{j_1\}$ then $|\supp(j)|\leq 3$ and $a_{j',j}=-1$ for all $j'\in \supp(j)$.
    \end{enumerate} 
\end{Proposition}

\begin{proof}
    Let $j_1\in \cC$ be minimum for $\leq_2$. By Lemma~\ref{l_description_chains_I2}, $m:=|\supp(j_1)|-1\leq 3$. Write $\supp(j_1)=\{j_1\}\cup  \{j_2^{(i)}\mid i\in \llbracket 1,m\rrbracket\}$. Let $i\in \llbracket 1,m\rrbracket$. By Lemma~\ref{l_description_chains_I2}, the set $E_i$ of the elements  of $I_2$ which are bigger than $j_2^{(i)}$ for $\leq_2$ is of the form $\{j_2^{(i)},j_3^{(i)},\ldots, j_{k(i)}^{(i)}\}$ for some $k(i)\in \Z_{\geq 2}$, and as a subgraph of $I_2$, $E_i$ is the line segment $[j_2^{(i)},j_{k(i)}^{(i)}]$. In particular, if $\cC'=\{j_1\}\cup \bigcup_{i=1}^m E_i$, then for each $x\in \cC'$, we have $\supp(x)\cap I_2\subset \cC'$. Therefore $\cC'$ is connected and $\cC=\cC'$. Moreover $\cC$ is a connected subgraph of the tree $I_1$ (by Proposition~\ref{p_characterization_I1}) and thus it is a tree. Therefore the $E_i$ are disjoint. Finally, by Lemma~\ref{l_description_chains_I2} 2-3, for all $j\in \cC'$ we have $|\supp(j)|\leq 3$ and, for all $j'\in \supp(j)$, $a_{j',j}=-1$. Lemma follows. 
\end{proof}

\begin{Remark}\label{r_independence_leq_2}
    The order $\leq_2$ does not depend on the choice of the  reduced decomposition made in the beginning of subsection~\ref{ss_Quantum_roots}, provided that it satisfies Lemma~\ref{l_I1_placed_beggining}. Indeed, if 2 elements of $I_2$ are comparable for $\leq_2$, they belong to the same connected component of $I_2$. Thus it suffices to prove that for each connected component $\cC$ of $I_2$  the restriction  of $\leq_2$ to $\cC$ is independent of the choice of the reduced decomposition. By Lemma~\ref{l_description_chains_I2}  and Proposition~\ref{p_connected_components_I2}, the minimum $j_1$ for $\leq_2$ is the unique element of $\cC$ whose degree is $3$. Then the elements $j_2^{(1)}, \ldots, j_2^{(m)}$ (where $m=|\supp(j_1)-1|$) that cover $j_1$ are the elements of $\supp(j_1)\setminus \{j_1\}$. If $i\in \llbracket 1,m\rrbracket$, the only element (if it exists) covering $j_2^{(i)}$ is the element of $\supp(j_2^{(i)})\setminus \{j_1,j_2^{(i)}\}$ etc. All these elements are defined uniquely from the graph $I_2$, which is independent of the choice of the reduced decomposition. 
\end{Remark}

With the above notation, we set: \begin{equation}\label{e_notation_k_kprime}
    k'=k,\text{ if }|\supp(j_k)|=2\text{ and }k'=k+1\text{ if }\supp(j_k)=\{j_{k-1},j_k,j_{k+1}\}.
\end{equation}\index{k@$k'$} The following lemma implies that \begin{equation}\label{e_inequalities_N}
    N_{j_1}(t)\geq N_{j_2}(t)\geq \ldots \geq N_{j_{k'}}(t), \forall t\in \llbracket L_1+1,L\rrbracket. 
\end{equation}

\begin{Lemma}\label{l_description_chains_In}
Let $j_1,\ldots,j_k\in I_2$ be a maximal chain for $\leq_2$. Let $n\in \Z_{\geq 3}$. Then:\begin{enumerate}
\item there exists $k_n\in \llbracket 0,k\rrbracket$ such that $\{i\in \llbracket 1,k\rrbracket\mid j_i\in I_n\}=\llbracket 1,k_n\rrbracket$,

\item we have $j_1\lhd_n j_2\lhd_n\ldots \lhd_n j_{k_n}$,

\item the sequence $(k_n)$ is non-increasing.

\end{enumerate}
Suppose moreover that  for all $n\geq 2$, $\bt(j_1,n)\geq\bt(j_2,n-1)$. Then:
\begin{enumerate}
\setcounter{enumi}{3}
    \item we have $\supp(j_k)=\{j_{k-1},j_k,j_{k+1}\}$,
    \item for all $i\in \llbracket 1,k\rrbracket$ and $n\geq 2$\begin{equation}\label{eq: k_n decreasing}
     \bt(j_i,n)\geq \bt(j_{i+1},n-1),
\end{equation}
    \item the sequence $(k_n)$ is decreasing until it reaches $0$. In particular $\max N_{j_1} \leq k+1$.

\end{enumerate} 
\end{Lemma}
\begin{proof}
We assume that $E:=\{x\in \llbracket 1,k\rrbracket\mid j_x\in I_n\}$ is non-empty. Let $y=|E|$. For $x\in \llbracket 1,y\rrbracket$, let $\tilde{t}_x\in \llbracket 2,L\rrbracket$ be such that $|\{x'\in \llbracket 1,k\rrbracket\mid N_{j_{x'}}(\tilde{t}_{x})\geq n\}|\geq x$ and such that $\tilde{t}_x$ is minimal for this property. By definition, $\tilde{t}_1<\tilde{t}_2<\ldots <\tilde{t}_y$.

 Let $x\in \llbracket 1,k\rrbracket$ be such that $N_{j_x}(\tilde{t}_1)=n$. We then have $N_{j_x}(\tilde{t}_1-1)=n-1$ and $\alpha_{j_x}(\gamma^\vee_{\tilde{t}_1-1})=-1$.   Assume $x\in \llbracket 2,k'-1\rrbracket$. Then by Lemma~\ref{l_description_chains_I2} we have  $\supp(j_x)=\{j_{x-1},j_x,j_{x+1}\}$ and $a_{j_{x-1},j_x}=-1=a_{j_{x+1},j_x}$. We have $N_{j_{x-1}}(\tilde{t}_1-1),N_{j_{x+1}}(\tilde{t}_1-1)\leq n-1$. Then: \begin{align*}
     \alpha_{j_x}(\gamma^\vee_{\tilde{t}_1-1})=&\alpha_{j_x}\left(N_{j_{x-1}}(\tilde{t}_1-1)\alpha_{j_{x-1}}^\vee+N_{j_{x}}(\tilde{t}_1-1)\alpha_{j_{x}}^\vee+N_{j_{x+1}}(\tilde{t}_1-1)\alpha_{j_{x+1}}^\vee\right)\\
     &\geq 2(n-1)-2(n-1)=0:
 \end{align*} a contradiction. Therefore $x\notin \llbracket 2,k'-1\rrbracket$. Assume $k=k'$ and $x=k$. Then $N_{j_{x-1}}(\tilde{t}_1-1)\leq n-1$. By \eqref{e_passage_3_before_2}, $a_{j_{k-1},j_k}=-1$ and  we reach a contradiction again. Therefore, $x=1$.

Let $x\in \llbracket 2,y\rrbracket$ and $i\in \llbracket 1,k\rrbracket$ be such that $N_{j_i}(\tilde{t}_x-1)=n-1=N_{j_i}(\tilde{t}_x)-1$. We assume that for all $x'\in \llbracket 1,x-1\rrbracket$,
 \[\{z\in \llbracket 1,k\rrbracket \mid N_{j_z}(\tilde{t}_{x'})\geq n\}=\llbracket 1,x'\rrbracket \text{ and }\{z\in \llbracket 1,k\rrbracket \mid N_{j_z}(\tilde{t}_{x'})\leq n-1\}=\llbracket x'+1,k\rrbracket.\] Then $i\geq x$. Assume $i\in \llbracket x+1,y\rrbracket$ and $i<k'$. Then: \begin{equation}\label{e_alpha_intermediate}
\alpha_{j_i}(\gamma^\vee_{\tilde{t}_x-1})=-1=\alpha_{j_i}\left(N_{j_{i-1}}(\tilde{t}_x-1)\alpha_{j_{i-1}}^\vee+N_{j_{i}}(\tilde{t}_x-1)\alpha_{j_{i}}^\vee+N_{j_{i+1}}(\tilde{t}_x-1)\alpha_{j_{i+1}}^\vee\right).
\end{equation} By choice of $i$, $N_{j_i}(\tilde{t}_x-1)=n-1$ and by our assumption $N_{j_{i-1}}(\tilde{t}_x-1),N_{j_{i+1}}(\tilde{t}_x-1)\leq n-1$. Therefore $\alpha_{j_i}(\gamma^\vee_{\tilde{t}_x-1})\geq 2(n-1)-2(n-1)=0$: a contradiction. Therefore $i=x$ (the case $i=k'=k$ is similar).  Thus we proved that for every $x\in \llbracket 1,y\rrbracket$, we have: \[\{z\in \llbracket 1,k\rrbracket \mid N_{j_z}(\tilde{t}_{x})\geq n\}=\llbracket 1,x\rrbracket \text{ and }\{z\in \llbracket 1,k\rrbracket \mid N_{j_z}(\tilde{t}_{x})\leq n-1\}=\llbracket x+1,k\rrbracket.\] In other words, $\bt(j_1,n)<\bt(j_2,n)<\ldots <\bt(j_y,n)$.  As $a_{j_z,j_{z+1}}\neq 0$ for every $z\in \llbracket 1,k-1\rrbracket$, we have $j_1\lhd_n j_2\lhd_n \ldots \lhd_n j_y$. We thus set $k_n=y=|\{i\in \llbracket 1,k\rrbracket \mid j_i\in I_n\}|$ and we have 1) and 2). Then 3) is clear since $I_n\subset I_{n-1}$.

We prove \eqref{eq: k_n decreasing} by iteration on $i$, as the assumption corresponds to the case $i=1$. Let $i\in \llbracket 1,k-1\rrbracket$ and suppose that \eqref{eq: k_n decreasing} is satisfied for $(i-1,n+1)$, that is to say $\bt(j_{i-1},n+1)\geq \bt(j_i,n)$. Let $t=\bt(j_i,n)$. If it is infinite, then \eqref{eq: k_n decreasing} is clear, so we suppose it is finite. Then $n\leq N_{j_{i-1}}(t)$, since $j_{i-1}\lhd_n j_i$ by 2). We also have $N_{j_{i-1}}(t)<n+1$ by assumption, so $N_{j_{i-1}}(t)=n$. Therefore $-1=\alpha_{j_i}(\gamma^\vee_t)=2N_{j_i}(t)-N_{j_{i-1}}(t)-N_{j_{i+1}}(t)= 2(n-1)-n - N_{j_{i+1}}(t)$, so $N_{j_{i+1}}(t)=n-1$. Thus $t \geq \bt(j_{i+1},n-1)$, which ends the recursion. For $(i,n)=(k,2)$, since $\bt(j_k,2)$ is finite (because $j_k \in I_2$) the same computation proves that $N_{j_{k+1}}(\bt(j_k,2))=1$, hence $|\supp(j_k)|=3$. Thus we have proved 4) and 5). In particular, if $\bt(j_{i+1},n-1)=\infty$ then $\bt(j_i,n)= \infty$. If $k_n\neq 0$, applying this to $(i,n)=(k_n,n+1)$ we deduce that $j_{k_n}\notin I_{n+1}$, so $k_{n+1}<k_n$. Therefore the sequence $(k_n)$ is decreasing until it reaches zero, $k_{k+2}=0$ and $j_1\notin I_{k+2}$, which proves 5).

\end{proof}
\begin{Remark}\label{r_connected_components_In}
Let $\cC$ be a connected component of $I_2$. Write $\cC=\{j_1\}\sqcup \bigsqcup_{i=1}^m\{j_2^{(i)},\ldots,j_{k(i)}^{(i)}\}$, with the same notation as in Proposition~\ref{p_connected_components_I2}. Then Lemma~\ref{l_description_chains_In} implies that if $I_n\cap \cC\neq \emptyset$, then  there exists $(k_m(i))_{i\in \llbracket 1,m\rrbracket}\in \prod_{i=1}^m \llbracket 1,k(i)\rrbracket$ such that $\cC\cap I_n=\{j_1\}\sqcup \bigsqcup_{i=1}^m\{j_2^{(i)},\ldots,j_{k_m(i)}^{(i)}\}$ (where $\{j_2^{(i)},\ldots,j_{k_m(i)}^{(i)}\}=\emptyset$, if $k_m(i)=1$). Moreover, the order $\leq_n$ is simply the restriction of $\leq_2$ to $I_n$, and it is independent on the choice of a reduced decomposition for  $\beta$, by Remark~\ref{r_independence_leq_2}. 
\end{Remark}

\subsection{Optimal bound on $\max N_{j_1}$}
In this section, we fix $j_1\in I_2$ minimal for $\leq_2$, and we fix a maximal chain $j_1\lhd_2 j_2 \lhd_2 \dots \lhd_2 j_k$ as in Lemma~\ref{l_description_chains_I2}. In particular, $k$ always denote the length of a maximal $\leq_2$-chain starting at $j_1$.

By Lemma \ref{l_description_chains_I2} 1), we have $\deg(j_1)=3$, therefore $|\supp(j_1)|\in \{2,3,4\}$. We treat each case separately, to obtain a bound on $\max N_{j_1}$ which depends on $k$. Along the way, we obtain more precise information on the $\leq_n$-chains starting at $j_1$. Beforehand, let us give a technical lemma on the difference $N_{j_k}-N_{j_{k-1}}$.
\begin{Lemma}\label{l_relation_Nk-1_Nk}
Let $t\in \llbracket L_1+1,L\rrbracket$ be such that $N_{j_k}(t)\geq 2$. Then: \begin{enumerate}
\item If $\supp(j_k)=\{j_{k-1},j_k,j_{k+1}\}$ (i.e if $k'=k+1$), then  $2N_{j_k}(t)\leq N_{j_{k-1}}(t)+2$. If moreover $N_{j_k}(t)<\max N_{j_k}$, then $N_{j_{k-1}}(t)\leq 2 N_{j_k}(t)$. 

\item If $\supp(j_k)=\{j_{k-1},j_k\}$ (i.e if $k'=k$), then   $ 2N_{j_k}(t)\leq N_{j_{k-1}}(t)+1$. If moreover $N_{j_k}(t)<\max N_{j_k}$, then $2N_{j_k}(t)\geq N_{j_{k-1}}(t)-1$. In particular, if  $N_{j_{k-1}}(t)$ is even, $N_{j_k}(t)=\frac{1}{2}N_{j_{k-1}}(t)$ and if $N_{j_{k-1}}(t)$ is odd, $N_{j_k}(t)\in \{\frac{N_{j_{k-1}}(t)-1}{2},\frac{N_{j_{k-1}}(t)+1}{2}\}$.  
\end{enumerate}

\end{Lemma}

\begin{proof}

Let $t'=\bt(j_k,N_{j_k}(t))$ and let $t''=\bt(j_k,N_{j_k}(t)+1)$, if $N_{j_k}(t)<\max N_{j_k}$. We have $t'<t\leq t''$.
\begin{enumerate}
\item  If $\supp(j_k)=\{j_{k-1},j_k,j_{k+1}\}$, by Lemma~\ref{l_description_chains_I2} we have $j_{k+1}\in  I_1\setminus I_2$ and thus $N_{j_{k+1}}(t')=1$. We also have $2N_{j_k}(t')=N_{j_{k-1}}(t')+N_{j_{k+1}}(t')-1=N_{j_{k-1}}(t')$ and thus $2N_{j_k}(t)-2=N_{j_{k-1}}(t')\leq N_{j_{k-1}}(t)$, since $N_{j_{k-1}}$ is non-decreasing. We also have $2N_{j_k}(t'')=2N_{j_k}(t)=N_{j_{k-1}}(t'')\geq N_{j_{k-1}}(t)$, which proves 1).

\item If $\supp(j_k)=\{j_{k-1},j_k\}$, we have $2N_{j_k}(t')=N_{j_{k-1}}(t')-1$ and thus $2(N_{j_k}(t)-1)=2N_{j_k}(t')=N_{j_{k-1}}(t')-1\leq N_{j_{k-1}}(t)-1$. Therefore $2N_{j_k}(t)\leq N_{j_{k-1}}(t)+1$.  

We also have  $2N_{j_{k}}(t'')=N_{j_{k-1}}(t'')-1$ and thus $2N_{j_k}(t)\geq N_{j_{k-1}}(t)-1$. Consequently we have \[N_{j_{k-1}}(t)-1\leq 2N_{j_k}(t)\leq N_{j_{k-1}}(t)+1,\] which proves 2). 
\end{enumerate}
\end{proof}
We can now obtain a bound on $N_{j_1}$.
\subsubsection{Case where $|\supp(j_1)|\leq 3$}\label{subsection: supp < 3}
\begin{Lemma}\label{l_bounding_supp=2}
Let $j_1\lhd_2 j_2 \lhd_2 \dots \lhd_2 j_k\in I_2$ be as in Lemma~\ref{l_description_chains_I2}. We assume that $\supp(j_1)=\{j_1,j_2\}$. Then $\max N_{j_1}\leq 2$. 
\end{Lemma}

\begin{proof}
By Lemma~\ref{l_description_chains_I2}, $\deg(j_1)=-a_{j_2,j_1}=3$. Assume $j_1\in I_3$ and set $t=\bt(j_1,3)$. Then \[\alpha_{j_1}(\gamma^\vee_t)=\alpha_{j_1}(2\alpha_{j_1}^\vee+N_{j_2}(t)\alpha_{j_2}^\vee)=4-3N_{j_2}(t)=-1:\] this is impossible. Lemma follows.
\end{proof}

\begin{Lemma}\label{Lemma: class 3}
Let $j_1\in I_2$ be minimal for $\leq_2$ such that $|\supp(j_1)|=3$. Then $\max N_{j_1}\leq k+1$. 

More precisely, write $\supp(j_1)=\{j_1,j_2,j_2'\}$ with $a_{j_2',j_1}=-1$ and $a_{j_2,j_1}=-2$. Then:
\begin{enumerate}
    \item if $\bt(j_2',2)\leq \bt(j_1,3)$ then $\bt(j_1,3)=\infty$,
    \item if $\bt(j_1,3)<\bt(j_2',2)\neq \infty$ then $\supp(j_2')=\{j_1,j_2'\}$ and $\bt(j_1,4)=\bt(j_2',3)=\infty$,
    \item else $j_2'\notin I_2$ and, for all $n\geq 2$, $\bt(j_2,n-1)\leq \bt(j_1,n) \leq \bt(j_2,n)$. The chain $j_1\lhd_2 j_2 \lhd_2 \dots \lhd_2 j_k$ then satisfies properties 4-6 of Lemma~\ref{l_description_chains_In} and in particular $\max N_{j_1} \leq k+1$.
\end{enumerate}
\end{Lemma}
\begin{proof}
Note that, by minimality of $j_1$ for $\leq_2$, $\bt(j_1,2)\leq \bt(j_2',2)$ and $\bt(j_1,2)\leq \bt(j_2,2)$. Moreover by Lemma~\ref{l_description_chains_In}, the same holds replacing $2$ by $n\geq2$.

    Suppose first that $\bt(j_2',2) \leq \bt(j_1,3)$. If $\bt(j_2',2)=\infty$ then $\bt(j_1,3)=\infty$ is clear. Else, for $t\geq \bt(j_2',2)+1$ such that $N_{j_1}(t)=2$, $\alpha_{j_1}(\gamma^\vee_t)$ is even and therefore $\bt(j_1,3)=\infty$.

   Suppose now that $\bt(j_1,3)<\bt(j_2',2)\neq \infty$. Then at time $t=\bt(j_2',2)$, we have $-1=\alpha_{j_2'}(\gamma_t^\vee)\geq 2-N_{j_1}(t)\geq 2-3$. Therefore $N_{j_1}(t)=3$ and there is no other element in $\supp(j_2')$. In particular $\bt(j_2',2)<\bt(j_1,4)\leq \bt(j_2',3)$. Let Let $t\in \llbracket \bt(j_2',2)+1, \bt(j_1,4)\rrbracket$. Then $N_{j_2'}(t)=2$ and $\alpha_{j_1}(\gamma^\vee_t)=2N_{j_1}(t)-2N_{j_2}(t)-2$ is even, so $\bt(j_1,4)=\infty$. Therefore $\max N_{j_1}\leq 3$. Moreover $k\geq 2$, since $|\supp(j_1)|>1$, which proves that $\max N_{j_1}\leq k+1.$

   Finally, suppose that $\bt(j_2',2)=\infty$, that is to say $j_2'\notin I_2$. Let $n\geq 2$ and suppose that $t:=\bt(j_1,n)$ is finite. Then $-1=\alpha_{j_1}(\gamma^\vee_t)=2(n-1)-2N_{j_2}(t)-1$, so $N_{j_2}(t)=n-1$. Hence $\bt(j_2,n-1)\leq \bt(j_1,n)$. The second inequality is given by Lemma~\ref{l_description_chains_In} 2). By Lemma~\ref{l_description_chains_In} 6), we deduce $\max N_{j_1} \leq k+1$.
\end{proof}

\subsubsection{Case  where $|\supp(j_1)|=4$}\label{subsection: supp = 3}

\begin{Proposition}\label{Proposition: supp=4 kmax}
    Let $j_1 \in I_2$ be minimal for $\leq_2$ such that $\supp(j_1)=4$. Let $k$ denote the maximal length of a $\leq_2$ chain starting at $j_1$. Then $\max N_{j_1} \leq \max(6,k+1)$.
\end{Proposition}
\begin{proof}
If $j_1 \notin I_3$ then $\max N_{j_1}\leq 2$ and the result is clear, so we can assume $j_1 \in I_3$.
    If $|\supp(j_1)|=4$ and $\supp(j_1)\not\subset I_2$, then by Lemma~\ref{Lemma: class 4A} $\max N_{j_1} \leq k+1$. Else, since we assume $j_1\in I_3$, either there is $j\in \supp(j_1)$ such that $\supp(j)\not\subset I_2$ and, by Lemma~\ref{Lemma: class 4D}, $\max N_{j_1} \leq k+1$; either $\supp(j)\subset I_2$ for all $j\in \supp(j_1)$ and, by Lemma~\ref{Lemma: class 4E}, we have $\max N_{j_1}\leq 6$.
\end{proof}
\begin{Lemma}\label{Lemma: class 4A}
    Let $j_1\in I_2$ be minimal for $\leq_2$ such that $|\supp(j_1)|=4$ and suppose $\supp(j_1)\not\subset I_2$. Then, any maximal chain $j_1 \lhd_2 j_2 \dots \lhd_2 j_k$ with $k\geq 2$ satisfies properties 4-5-6 of Lemma \ref{l_description_chains_In}, and in particular $\max N_{j_1}\leq k+1$.

\end{Lemma}

\begin{proof}
    Let us write $\supp(j_1)=\{j_1,j_2,j_2',j_2''\}$ with $j_2''\notin I_2$. 
    By Lemma~\ref{l_description_chains_In}, to prove the result it suffices to prove that $\bt(j_1,n)\geq \bt(j_2,n-1)$ for all $n\geq 2$.

    If $\bt(j_1,n)=\infty$ then the inequality is clear, so suppose $t=\bt(j_1,n)\neq \infty$. Since $N_{j_2''}(t)=1$ we have $$-1=\alpha_{j_1}(\gamma_t^\vee)=2N_{j_1}(t)-N_{j_2}(t)-N_{j_2'}(t)-N_{j_2''}(t)=2(n-1)-1-N_{j_2}(t)-N_{j_2'}(t),$$ so $N_{j_2}(t)+N_{j_2'}(t)=2(n-1)$. Moreover,  $\max(N_{j_2}(t),N_{j_2'}(t))\leq n-1$, since, by Lemma~\ref{l_description_chains_In} 2), $j_1$ is also minimal for $\leq_n$. Therefore $N_{j_2}(t)+N_{j_2'}(t)=2(n-1)$ implies $N_{j_2}(t)=N_{j_2'}(t)=n-1$, which proves $\bt(j_1,n)=t>\bt(j_2,n-1)$. 

\end{proof}

\begin{Lemma}\label{Lemma: Class 4S}
    Let $j_1\in I_2$ be minimal for $\leq_2$ such that $|\supp(j_1)|=4$, $\supp(j_1) \subset I_2$ and $j_1\in I_3$.
 Then there is a unique $j\in \supp(j_1)$ such that $|\supp(j)|=2$, and for any $j_2\in \supp(j_1)\setminus \{j_1,j\}$, we have $\supp(j_2)=3$ and 
        \begin{equation}\label{eq: Class 4S statement}\bt(j_2,2)<\bt(j_1,3)<\bt(j,2)\leq \bt(j_1,4).\end{equation}
\end{Lemma}
\begin{proof}
   
    Write $\supp(j_1)=\{j_2^{(1)},j_2^{(2)},j_2^{(3)}\}$. Fix $i\in \llbracket 1,3\rrbracket$. Let $t=\bt(j_1,3)\neq \infty$. Since $j_1$ is minimal for $\leq_2$, it is also minimal for $\leq_3$. Then $N_{j_2^{(1)}}(t)+N_{j_2^{(2)}}(t)+N_{j_2^{(3)}}(t)=5$ and $N_{j_2^{(i)}}(t)\leq 2$ for all $i$, so upon relabelling, we can suppose $(N_{j_2^{(1)}}(t),N_{j_2^{(2)}}(t),N_{j_2^{(3)}}(t))=(2,2,1)$.
    
    For $i\in \{1,2\}$, $t_i:=\bt(j_2^{(i)},2)$ satisfies $t_i<t$ and $-1=\alpha_{j_2^{(i)}}(\gamma_{t_i}^\vee)<0=2N_{j_2^{(i)}}(t_i)-N_{j_1}(t_i)$. Therefore $\supp(j_2^{(i)})$ admits exactly one other element, hence is of cardinality $3$. Let now $t=\bt(j_2^{(3)})<\bt(j_1,3)$. Then $\alpha_{j_2^{(3)}}(\gamma_t^\vee)=-1$ and $2N_{j_2^{(3)}(t)}-N_{j_1}(t)\leq 2-3=-1$ so $N_{j_1}(t)=3$ and $\supp(j_2^{(3)})=\{j_1,j_2^{(3)}\}$, which concludes the proof.
\end{proof}

\begin{Lemma}\label{Lemma: class 4D}
    Let $j_1\in I_2$ be minimal for $\leq_2$ such that $|\supp(j_1)|=4$, $\supp(j_1) \subset I_2$ and $j_1 \in I_3$. Let $j_2''$ be the unique element of $\supp(j_1)$ with $|\supp(j_2'')|=2$, given by Lemma~\ref{Lemma: Class 4S}.
    
    Suppose that there exists $j_2' \in \supp(j_1)$ with $\supp(j_2')\not\subset I_2$. Then the following properties are satisfied:
    \begin{enumerate}
        \item For any $n\geq 1$, 
        \begin{equation}\label{eq: type 4D}
           \bt(j_1,2n)\leq\bt(j_2',n+1)\leq\bt(j_1,2n+1)\leq\bt(j_2'',n+1)\leq\bt(j_1,2n+2).   
        \end{equation}
        \item The unique maximal chain $j_1 \lhd_2 j_2 \dots \lhd_2 j_k$ such that $j_2\notin\{j_2',j_2''\}$ satisfies properties 4-5-6 of Lemma \ref{l_description_chains_In}, and in particular $\max N_{j_1}\leq k+1$.
        
    \end{enumerate}
\end{Lemma}
\begin{proof}
    By assumption, $\supp(j_2'')=\{j_1,j_2''\}$ and $\supp(j_2')\not\subset I_2$, hence both $j_2'$, $j_2''$ are maximal for $\leq_2$. Let $j_2$ denote the remaining element of $\supp(j_1)$. 
    
    Fix $n\geq2$. Since $j_2'$ is maximal for $\leq_2$ and has support of cardinality $3$, Lemma~\ref{l_relation_Nk-1_Nk} applied to $t:=\bt(j_2',n+1)+1$ (if it is finite, otherwise it is clear) gives $\bt(j_1,2n)\leq\bt(j_2',n+1)$. Similarly, point 2 of the same Lemma applied to $t:=\bt(j_2'',n+1)+1$ gives $\bt(j_1,2n+1)\leq\bt(j_2'',n+1)$.

    Let $t=\bt(j_1,2n+1)$, which we suppose finite. Then $N_{j_1}(t)=2n$, $N_{j_2''}(t)\leq n$ and, since $j_1$ is minimal for $\lhd_{2n+1}$, $N_{j_2}(t)\leq 2n$. Therefore $-1=\alpha_{j_1}(\gamma_t^\vee)\geq 4n-n-2n-N_{j_2'}(t)=n-N_{j_2'}(t)$, so $N_{j_2'}(t)\geq n+1$ and $t>\bt(j_2',n+1)$.

    Let $t=\bt(j_1,2n+2)$ which we suppose finite, then $N_{j_1}(t)=2n+1$, $N_{j_2'}(t)\leq n+1$ (by Lemma~\ref{l_relation_Nk-1_Nk}) and $N_{j_2}(t)\leq 2n+1$. Therefore $-1=\alpha_{j_1}(\gamma_t^\vee)\geq 4n+2-(2n+1)-(n+1)-N_{j_2''}(t)=n-N_{j_2''}(t)$, so $N_{j_2''}(t)\geq n+1$ and $\bt(j_2'',n+1)\leq t$. This concludes the proof of Formula~\eqref{eq: type 4D}.

    Let us check that $\bt(j_1,n)\geq \bt(j_2,n-1)\geq\bt(j_1,n-1)$ for any $n\geq 3$. The inequality $\bt(j_2,n-1)\geq \bt(j_1,n-1)$ holds because $j_1 \lhd_2 j_2 \implies j_1 \lhd_{n-1} j_2$ by point 2) of Lemma \ref{l_description_chains_In} (if they belong to $I_{n-1}$). For the remaining inequality, suppose $t:=\bt(j_1,n)\neq \infty$:
    \begin{itemize}
        \item If $n$ is even, write $n=2p$. Then by Formula~\eqref{eq: type 4D}, $\max(N_{j_2'}(t),N_{j_2''}(t))\leq p$, so $-1=\alpha_{j_1}(\gamma_t^\vee)\geq 2(2p-1)-2p-N_{j_2}(t)$ . We deduce $N_{j_2}(t)\geq n-1$.
        \item Else, write $n=2p+1$. Then similarly $N_{j_2''}(t)\leq p$ and $N_{j_2'}(t)\leq p+1$, so $-1=\alpha_{j_1}(\gamma_t^\vee)\geq 2(2p)-2p-1-N_{j_2}(t)$, so $N_{j_2}(t)\geq 2p=n-1$.
    \end{itemize}
Either way, we have proved $\bt(j_1,n)\geq\bt(j_2,n-1)$. Then 2) follows from Lemma~\ref{l_description_chains_In}.
\end{proof}

\begin{Lemma}\label{Lemma: class 4E}
     Let $j_1\in I_2$ be minimal for $\leq_2$ with $|\supp(j_1)|=4$, $j_1\in I_3$ and, for all $j\in \supp(j_1)$, $\supp(j)\subset I_2$. Then $\max N_{j_1} \leq 6$. 
     
     More precisely, using Lemma~\ref{Lemma: Class 4S}, we write $\supp(j_1)=\{j_1,j_2,j_2',j_2''\}$ such that $|\supp(j_2'')|=2$ and $\bt(j_2,3)\leq \bt(j_2',3)$, and write $\supp(j_2')=\{j_1,j_2',j_3'\}$. Then:
     \begin{enumerate}
         \item If $\bt(j_2',3)\leq \bt(j_1,4)$, then $j_1\notin I_4$ and $j_2''\notin I_3$
         \item If $\bt(j_1,4)\neq \infty$ and $\bt(j_2',3)=\infty$, then $j_1\notin I_5$ and $j_2''\notin I_3$
         \item Else if $\bt(j_1,4)< \bt(j_2',3)<\infty$, then   $\bt(j_2,n-1)\leq \bt(j_1,n)$ for all $n\in \llbracket2, 5\rrbracket$, $\supp(j_3')=\{j_2',j_3'\}$ and $j_1 \notin I_7$, $j_2'\notin I_5$, $j_3'\notin I_3$.
     \end{enumerate}
\end{Lemma}
\begin{proof}
    Suppose first that $\bt(j_2',3)\leq \bt(j_1,4)$ and assume by contradiction that $t:=\bt(j_1,4)$ is finite. By Lemma~\ref{Lemma: Class 4S}, we have $N_{j_2''}(t)\geq 2$, and since $\bt(j_2,3)\leq \bt(j_2',3)\leq \bt(j_1,4)$, we have $N_{j_2}(t)\geq 3$, $N_{j_2'}(t)\geq 3$. Therefore $N_{j_2}(t)+N_{j_2'}(t)+N_{j_2''}(t)\geq 8$, so $-1=\alpha_{j_1}(\gamma_t^\vee)\leq -2$, a contradiction. Hence in this case, $\bt(j_1,4)=\infty$ and $j_1\notin I_4$.

    Suppose now that $\bt(j_1,4)\neq \infty$ and $\bt(j_2',3)=\infty$ and let $t\geq \bt(j_1,4)+1$. Then $N_{j_1}(t)=4$ so $N_{j_2}(t)\leq 4$ and $N_{j_2''}(t)\leq 2$ by Lemma~\ref{l_relation_Nk-1_Nk}. Moreover, by assumption $N_{j_2'}(t)\leq 2$ and thus $\alpha_{j_1}(\gamma_t^\vee)\geq 8-4-2-2=0$ can not be equal to $-1$, so $N_{j_1}$ is blocked at $4$: $j_1\notin I_5$. Moreover, in these two cases, $N_{j_1}\leq 4$ and thus by Lemma~\ref{l_relation_Nk-1_Nk}, $N_{j_2''}\leq 2$. 

    Finally, assume that $\bt(j_1,4)<\bt(j_2',3)\neq\infty$. Note that $\bt(j_2,3) < \bt(j_1,4)$ is necessary in order to have $\alpha_{j_1}(\gamma^\vee_{\bt(j_1,4)})=-1$. Let $t=\bt(j_2',3)$, by assumption $N_{j_1}(t)\geq 4$ and $N_{j_3'}(t)\geq 1$. So $\alpha_{j_2'}(\gamma^\vee_t)=-1$ enforces that $N_{j_1}(t)=4$ and $N_{j_3'}(t)= 1$. 
    The first consequence is that $\bt(j_2',3)<\bt(j_3',2)$. Since we suppose $\supp(j_2')\subset I_2$, computation at time $\bt(j_3',2)$ then proves that \[\supp(j_3')=\{j_2',j_3'\}.\] 
    
    We then have \[j_3'\in I_3 \implies j_2'\in I_5 \implies j_1\in I_7.\] Indeed, the first implication is a direct consequence of Lemma~\ref{l_relation_Nk-1_Nk} 2). Now assume that $t:=\bt(j_2',5)<\infty$. Then by Lemma~\ref{l_relation_Nk-1_Nk}, we have $N_{j_3'}(t)\leq 2$ and thus $\alpha_{j_2'}(\gamma_t^\vee)=-1=8-N_{j_1}(t)-N_{j_3'}(t)$ and so $N_{j_1}(t)\geq 7$ and \begin{equation}\label{e_ineq_j1_j2'}
        \bt(j_1,7)\leq \bt(j_2',5).
    \end{equation}

    We now show that, under the assumptions of the statement, $j_1\notin I_7$, which is enough to conclude.

    We computed $N_{j_1}(\bt(j_2',3))=4$ and $N_{j_3'}(\bt(j_2',3))=1$, so $\bt(j_2',3)<\bt(j_1,5)$. Let $t=\bt(j_1,5)$ and suppose it is finite. For $\bt(j_2',3)<t'\leq t$  we have $\alpha_{j_2'}(\gamma^\vee_{t'})\geq 6 -4-2\geq 0$, so $N_{j_2'}(t)=3$. Moreover since $N_{j_1}(t)=4$, we have $N_{j_2''}(t)=2$ and $N_{j_2}(t)\leq 4$. Therefore $\alpha_{j_1}(\gamma^\vee_t)=-1$ enforces that $N_{j_2}(t)=4$.  This proves $\bt(j_2,n-1)\leq \bt(j_1,n)$ for $n=5$. For $n=2$ it is clear, for  $n=3$ it is given by Lemma~\ref{Lemma: Class 4S} and we have already checked it for $n=4$, hence the first statement of point 3 is checked. Moreover at time $t'=t+1$ we have \begin{equation}\label{eq: time 5+1}(N_{j_1}(t'),N_{j_2}(t'),N_{j_2'}(t'),N_{j_2''}(t'))=(5,4,3,2).\end{equation} Now set $t=\bt(j_1,6)$. If it is finite, then at this time exactly two out of the three values $N_{j_2},N_{j_2'},N_{j_2''}$ have increased by one. Indeed $N_{j_2}+N_{j_2'}+N_{j_2''}$ need to increase by two. By Lemma~\ref{l_relation_Nk-1_Nk}, $N_{j_2}\leq N_{j_1}$ and $2N_{j_2''}\leq N_{j_1}+1$ so these two terms can not increase by two. Moreover, by the same lemma, as long as $N_{j_2'}\leq 4$, $N_{j_3'}\leq 2$. Thus if $t'=\bt(j_2',5)$, we have $\alpha_{j_2'}(\gamma_{t'}^\vee)=-1=8-N_{j_1}(t')-N_{j_3'}(t')\geq 6-N_{j_1}(t')$, so $t'\geq \bt(j_1,7)\geq t$ and $N_{j_2'}(t)\leq 4$: $N_{j_2'}$ can not increase by two either. Thus two out of the three values $N_{j_2},N_{j_2'},N_{j_2''}$ have increased by one, we can actually check that all the three cases are possible. We deal with each case separately.

    \begin{enumerate}
        \item Suppose that $(N_{j_2}(t),N_{j_2'}(t),N_{j_2''}(t))=(5,3,3)$. Then at any time $t$' such that $\bt(j_2',4)\geq t'>t$ we have $\alpha_{j_2'}(\gamma^\vee_{t'})\leq 6-6-2<-1$ by Lemma~\ref{l_relation_Nk-1_Nk}), so $\bt(j_2',4)=\infty$. Therefore if $t'=\bt(j_1,7)$ is finite, necessarily $(N_{j_1}(t'),N_{j_2''}(t'))=(6,4)$, which contradicts Lemma~\ref{l_relation_Nk-1_Nk} 2). Hence $\bt(j_1,7)$ is infinite and $j_1\notin I_7$, which implies $j_2'\notin I_5$, $j_3'\notin I_3$.
        \item Suppose that $(N_{j_2}(t),N_{j_2'}(t),N_{j_2''}(t))=(5,4,2)$. The situation is similar: in this case $N_{j_2''}$ is stuck at $2$, so $\bt(j_2',5)\leq \bt(j_1,7)$. However by Lemma~\ref{l_relation_Nk-1_Nk} 2), $\bt(j_2',5) \leq \bt(j_3',3)$. We deduce that, for any $\bt(j_1,6)< t'\leq \bt(j_2',5)$, $\alpha_{j_2'}(\gamma^\vee_{t'})=0$, so $\bt(j_2',5)=\bt(j_1,7)=\bt(j_3',3)=\infty$.
        \item Suppose that $(N_{j_2}(t),N_{j_2'}(t),N_{j_2''}(t))=(4,4,3)$. Then by \eqref{e_ineq_j1_j2'} and Lemma~\ref{l_relation_Nk-1_Nk}, we have $$\bt(j_2,6)\leq \bt(j_1,7)\leq \min(\bt(j_2',5),\bt(j_2'',4)).$$ Hence if $t:=\bt(j_1,7)<\infty$, we have $N_{j_2}(t)+N_{j_2'}(t)+N_{j_2''}(t)=13$, with  $N_{j_2'}(t)+N_{j_2''}(t)\leq 7$. Consequently $\bt(j_2,6)<\bt(j_1,7)<\infty$. But this contradicts Lemma~\ref{l_impossible_scheme} below. Therefore $j_1\notin I_7$, $j_2'\notin I_5$ and $j_3'\notin I_3$.
    \end{enumerate}
\end{proof}

\begin{Lemma}\label{l_impossible_scheme}
Let $j_1,j_2$ be the first terms of a maximal chain for $\leq_2$. Then there exists no couple $(x,y)\in \llbracket 1,L\rrbracket^2$ such that $(N_{j_1}(x),N_{j_2}(x))=(6,4)$ and  $(N_{j_1}(y),N_{j_2}(y))=(6,6)$. 
\end{Lemma}

\begin{proof}
We assume that there exists  $(x,y)\in \llbracket 1,L\rrbracket^2$ such that $(N_{j_1}(x),N_{j_2}(x))=(6,4)$ and  $(N_{j_1}(y),N_{j_2}(y))=(6,6)$.  Let $z\in \llbracket 1,L\rrbracket$. Assume $(N_{j_1}(z),N_{j_2}(z))=(6,4)$. If we had $\supp(j_2)=\{j_1,j_2\}$, then $j_2$ would be maximal for $\leq_2$ and Lemma~\ref{l_relation_Nk-1_Nk} (2)  would lead to a contradiction. Therefore we have $\supp(j_2)=\{j_1,j_2,j_3\}$, for some $j_3\in I_1$, and by Lemma~\ref{l_description_chains_I2}, we have $a_{j_1,j_2}=a_{j_3,j_1}=-1$. 

Let $t=\bt(j_2,6)$. We have $N_{j_1}(t)=6$ by assumption. We have $10-6-N_{j_3}(t)=-1$ and thus \begin{equation}\label{e_ineq_t_t}
   \bt(j_3,4)< \bt(j_3,5)<\bt(j_2,6).
\end{equation}

Let $t=\bt(j_2,5)$. We have $t<\bt(j_1,7)$ and thus  we have $N_{j_3}(t)=3$ (in order to have $8-6-3=-1$). Therefore $(N_{j_1}(t+1),N_{j_2}(t+1),N_{j_3}(t+1))=(6,5,3)$.

Let now $t\in \llbracket 1,L\rrbracket$ be such that  $(N_{j_1}(t),N_{j_2}(t),N_{j_3}(t))=(6,5,3)$. As $\bt(j_3,4)<\infty$ and  $6-5\neq -1$, we have $\supp(j_3)\neq \{j_2,j_3\}$. So by Lemma~\ref{l_description_chains_I2}, we can write $\supp(j_3)=\{j_2,j_3,j_4\}$, for some $j_4\in I_1$, and we have $a_{j_2,j_3}=a_{j_4,j_2}=-1$. If $t=\bt(j_3,4)$, we have  $N_{j_4}(t)=2$ (since $6-5-2=-1$). So we have \begin{equation}\label{e_configuration1}
(N_{j_1}(t),N_{j_2}(t),N_{j_3}(t),N_{j_4}(t))=(6,5,3,2)
\end{equation}
and \begin{equation}\label{e_configuration2}
(N_{j_1}(t+1),N_{j_2}(t+1),N_{j_3}(t+1),N_{j_4}(t+1))=(6,5,4,2).
\end{equation} 

Let $t=\bt(j_3,5)$. By \eqref{e_ineq_t_t}, we have $\gamma_t^\vee=6\alpha_{j_1}^\vee+5\alpha_{j_2}^\vee+4\alpha_{j_3}^\vee+a\alpha_{j_4}^\vee+\delta^\vee$, for some $a\in \llbracket 1,4\rrbracket$ and $\delta^\vee\in \bigoplus_{i\in I\setminus \{j_1,j_2,j_3,j_4\}}\N\alpha_i^\vee$. We have $\alpha_{j_3}(\gamma_t^\vee)=8-5+a\alpha_{j_3}(\alpha_{j_4}^\vee)=-1$ and thus $a \alpha_{j_3}(\alpha_{j_4}^\vee)=aa_{j_4,j_3}=-4$ and thus $a=4$. Consequently \[\bt(j_4,4)<\bt(j_3,5)<\infty.\] 

Let $\tilde{t}=\bt(j_4,4)$. We have $2N_{j_4}(\tilde{t})=6> N_{j_3}(\tilde{t})+1$ since $N_{j_3}(\tilde{t})\leq 4$.  By Lemma~\ref{l_relation_Nk-1_Nk}, this implies $\supp(j_4)\neq \{j_3,j_4\}.$ Let $j_5$ be such that $\supp(j_4)=\{j_3,j_4,j_5\}$.
 Then $6-N_{j_3}(\tilde{t})-N_{j_5}(\tilde{t})=-1=6-4-N_{j_5}(\tilde {t})$ and thus \begin{equation}\label{e_majoration_t5}
    \bt(j_5,3)<\bt(j_4,4)<\infty.
\end{equation}

 Let  $t=\bt(j_4,3)$. By \eqref{e_configuration2}, we have $4-4-N_{j_5}(t)=-1$ and thus $N_{j_5}(t)=1$. We thus have $(N_{j_1}(t),N_{j_2}(t),N_{j_3}(t),N_{j_4}(t),N_{j_5}(t))=(6,5,4,2,1)$ and then $(N_{j_1}(t+1),N_{j_2}(t+1),N_{j_3}(t+1),N_{j_4}(t+1),N_{j_5}(t+1))=(6,5,4,3,1)$.
 
By \eqref{e_majoration_t5}, $t:=\bt(j_5,2)<\infty$. Then $\gamma_t^\vee= z \alpha_{j_4}^\vee+\alpha_{j_5}^\vee+\delta^\vee,$ for some $z\in \Z_{\geq 3}$ and $\delta^\vee\in Q^\vee_\N=\bigoplus_{i\in I\setminus \{j_4,j_5\}}\N \alpha_i^\vee$. By Lemma~\ref{l_description_chains_I2} $a_{j_4,j_5}=-1$, thus $2-z-\alpha_{j_5}(\delta^\vee)=-1$  and hence $z=3$ and $\alpha_{j_5}(\delta^\vee)=0$. This proves that $\supp(j_5)=\{j_4,j_5\}$. 
 Then if $t=\bt(j_5,2)$, we have \[(N_{j_1}(t),N_{j_2}(t),N_{j_3}(t),N_{j_4}(t),N_{j_5}(t))=(6,5,4,3,1)\] and  \[(N_{j_1}(t+1),N_{j_2}(t+1),N_{j_3}(t+1),N_{j_4}(t+1),N_{j_5}(t+1))=(6,5,4,3,2).\] But then $2N_{j_i}(t+1)-N_{j_{i+1}}(t+1)-N_{j_{i-1}}(t+1)=0$ for every $i\in \llbracket 2,4\rrbracket$, and $2N_{j_5}(t+1)-N_{j_4}(t+1)=1$. Therefore the situation can no longer evolve, which proves the lemma. 
\end{proof}

\subsubsection{Conclusion: Finiteness of the set of quantum roots}

With the preceding results, we are already able to give an upper bound on the number of quantum roots in a given Kac-Moody root system. In particular, we prove that the set $\cQ(\Phi_+)$ is finite for any Kac-Moody root system.

\begin{Proposition}\label{p_majoration_coefficients_AS_roots}
 Let $k_{\max}$ be the length of a maximal chain for $\leq_2$ and $n=|I|$ (we have $k_{\max}\leq n$).  Let $\beta$ be a quantum root. 
Write $\beta^\vee=\sum_{i\in I} x_i \alpha_i^\vee$. Then we have $x_i\leq\max(k_{\max}+1,6)$ for every $i\in I$. In particular, $\htt(\beta^\vee)\leq n(\max(6,k_{\max}+1))$. 
\end{Proposition}

\begin{proof}
    Let $i\in I$. If $i\notin I_2$, we have $x_i\leq 1$. Assume $i\in I_2$. Let $j\in I_2$ be such that $j\leq_2 i$ and such that $j$ is minimal for $\leq_2$.
Then by \eqref{e_inequalities_N}, we have $x_i\leq x_j$. Moreover, by Lemma \ref{Lemma: class 3}, Lemma \ref{l_bounding_supp=2} and Proposition \ref{Proposition: supp=4 kmax}, $x_j\leq \max(6,k_{\max}+1)$. Therefore $\htt(\beta^\vee)=\sum_{i'\in I_1} x_{i'}\leq |I_1| \max(k_{\max}+1,6)\leq n\max(k_{\max}+1,6).$
\end{proof}

\begin{Theorem}\label{t_finiteness_almost_simple_roots}
Let $n\in \Z_{\geq 2}$. Then   for every generalized Cartan matrix $A$ of size $n$, if $\Phi$ is the real root system associated to $A$, then the set of quantum roots is finite, with cardinality at most $n^{n+5}$. 
\end{Theorem}

\begin{proof}
Let $\cQ(\Phi_+)$ be the set of quantum roots.
By Proposition~\ref{p_majoration_coefficients_AS_roots}, for every $\beta\in \cQ(\Phi_+)$, we can write $\beta^\vee=\sum_{i\in I} x_i \alpha_i$, with $x_i\in \llbracket 0,n+5\rrbracket$ for every $i\in I$. Therefore $|\cQ(\Phi_+)|=|\{\beta^\vee\mid \beta\in \cQ(\Phi_+)\}|\leq n^{n+5}$.
\end{proof}

\begin{Remark}
\begin{enumerate}
    \item   Note that the bound given in Theorem~\ref{t_finiteness_almost_simple_roots} is very rough and is never sharp. A way to improve it would be to first count the number of possible $1$-star convex subtrees of $I$ (whose cardinality is at most $2^{|I|}$). Then for a given such subtree $I_1$ of $I$, the constraints of Lemma~\ref{l_description_chains_I2}, Lemma~\ref{l_description_chains_In} and Proposition~\ref{p_majoration_coefficients_AS_roots} should enable to give a much better majoration of the number of quantum $\beta$ satisfying $I_1(\beta)=I_1$.  Moreover if a generalized Cartan matrix has few $-1$ coefficients, then it will have few quantum roots, see Lemma~\ref{l_no_-1_coefficient} for example.

    \item Note that in the reductive case, in ADE type, the set of quantum root is the entire set $\Phi_+$ (see Proposition~\ref{p_quantum_ADE}). On the contrary, when $\Phi_+$ is infinite, ``almost all'' the roots are not quantum, by the theorem above.
\end{enumerate}  
\end{Remark}

\subsection{Classification of quantum roots through Dynkin sequences}\label{subsection: classification}
In Subsections~\ref{subsection: supp < 3} and~\ref{subsection: supp = 3} we have obtained finer structure results on quantum roots than what we used to prove finiteness of $\cQ(\Phi_+)$. In this section, we reformulate these results in terms of Dynkin sequences. We then obtain a complete classification of quantum roots through Dynkin sequences (see Theorem~\ref{Theorem: Classification quantum roots}). This section is independent of Section~\ref{s_almost_simple_roots_covers}.
\subsubsection{Standard Dynkin diagrams}

There is a name for certain classes of Dynkin diagrams, which we now give. Most of our notation follow the standard classification of Dynkin diagrams (see e.g \cite{bourbaki1981elements}). In this section, $\Gamma$ is a Dynkin diagram with set of edges $E$ and weight function $w$. Recall that, if $A=(a_{i,j})_{i,j\in I}$ is a generalized Cartan matrix, then the associated Dynkin diagram is $\Gamma=(I,E,w)$ with $E=\{(i,j)\in I^2 \mid a_{i,j}< 0\}$ and $w(i,j)=-a_{i,j}$ for all $(i,j)\in E$. We identify a Dynkin diagram $\Gamma$ with its set of vertices.

Amongst Dynkin segments, we have the following classes:
\begin{itemize}
    \item[(A)] For $n\geq 1$, we say that $\Gamma$ is a Dynkin diagram of type $A_n$ if it is a segment with $n$ vertices such that $w(i,j)=1$ for all $(i,j)\in E$.
    \item[(C)] For $n\geq 2$, we say that $\Gamma$ is a Dynkin diagram of type $C_n$ if it is a segment with a leaf $j_0$ such that $\Gamma \setminus \{j_0\}$ is a diagram of type $A_{n-1}$ and, for $j_1$ the unique vertex adjacent to $j_0$, $(w(j_0,j_1),w(j_1,j_0))=(1,2)$. We call $j_0$ the \textbf{special vertex} of $\Gamma$.
    \item[(F)] For $n\geq 4$, we say that $\Gamma$ is a Dynkin diagram of type $F_n$ if it is a Dynkin segment with a leaf $j_{-1}$ such that $\Gamma \setminus \{j_{-1}\}$ is a Dynkin diagram of type $C_{n-1}$ and its special vertex $j_0$ is the unique neighbour of $j_{-1}$.
    \item[(G)] For $n\geq 2$, we say that $\Gamma$ is a Dynkin diagram of type $G_n$ if it is a segment with a leaf $j_0$ such that $\Gamma \setminus \{j_0\}$ is a diagram of type $A_{n-1}$ and, for $j_1$ the unique vertex adjacent to $j_0$, $(w(j_0,j_1),w(j_1,j_0))=(1,3)$.     
\end{itemize}
We also have classes of Dynkin trees which are not Dynkin segments, but Dynkin trees with a unique branching point:
\begin{itemize}
    \item[(D)] For $n\geq 4$, we say that $\Gamma$ is a Dynkin diagram of type $D_n$ if it has two distinct leaves $j_2'$, $j_2''$ sharing their unique neighbour $j_1$, such that $\Gamma \setminus \{j_2',j_2''\}$ is a Dynkin diagram of type $A_{n-2}$ with leaf $j_1$. 
    \item[(E)] For $n\geq 6$, we say that $\Gamma$ is a Dynkin diagram of type $E_n$ if it has a unique branching point $j_1$ with three neighbours, it is not of type $D_n$ but there exists $j_3'\in \Gamma$ such that $\Gamma \setminus \{j_3'\}$ is a Dynkin diagram of type $D_{n-1}$.
\end{itemize}

Note that the Dynkin diagrams associated to finite root systems are exactly the Dynkin diagrams of type $(A_n)_{n\geq 1}$, $(B_n)_{n\geq 2}$ (which is dual to type $C$), $(C_n)_{n\geq 2}$, $(D_n)_{n\geq 4}$ and the exceptional types $E_6$, $E_7$, $E_8$, $F_4$ and $G_2$.

\begin{figure}[H]
	\centering
	\includegraphics[width=\linewidth]{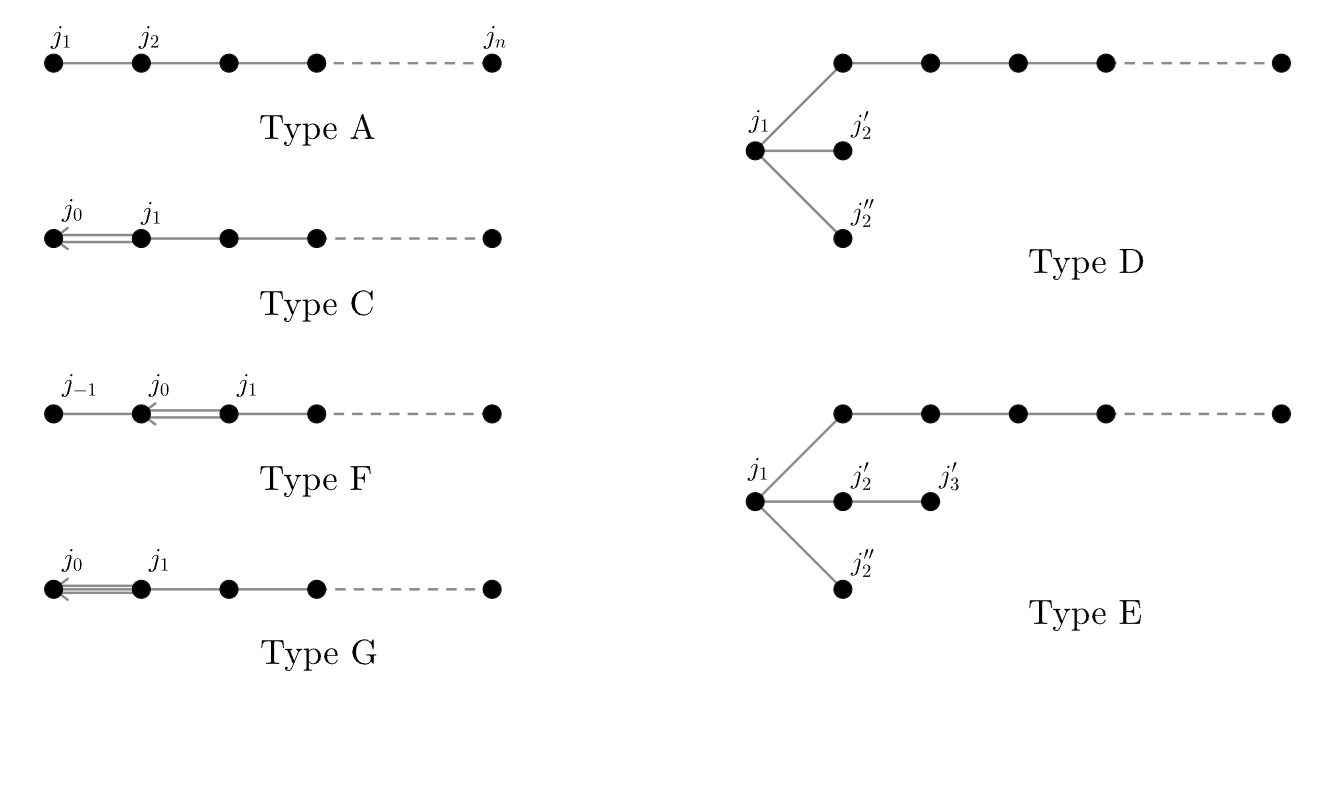}
	\caption{The different types of Dynkin diagrams}
	\label{fig: Type of Dynkin diagrams}
\end{figure}

\subsubsection{Dynkin sequences associated to quantum roots}\label{subsection: Dynkin sequences quantum roots}
In this section, we fix a quantum root $\beta$. For $n\geq 2$ and $j\in I_2(\beta)$ such that $\deg(j)=3$, let $I_n(\beta,j)$ denote the connected component of $I_n(\beta)$ containing $j$ (if $j\notin I_n(\beta)$ then we set $I_n(\beta,j)=\emptyset$). We have the following classification for $(I_n(\beta,j))_{n\geq 2}$.
\begin{Proposition}\label{Classification: supp=2}
    Let $j\in I_2(\beta)$ be such that $\deg(j)=3$ and $|\supp(j)|=2$. Then 
    \begin{itemize}
        \item[(2G)] $I_2(\beta,j)$ is a Dynkin diagram of type $G$, $j$ is the only leaf of $I_2(\beta,j)$ which is a leaf in $I_1(\beta)$. Moreover $I_3(\beta,j)$ is empty. 
        
    \end{itemize}
\end{Proposition}

 \begin{figure}[H]
 \centering
 \includegraphics[width=\linewidth]{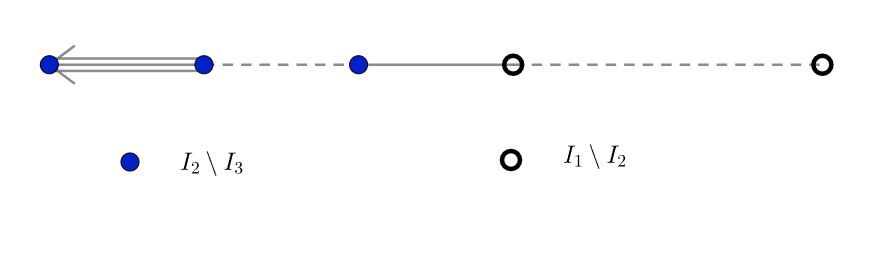}
 \caption{Dynkin diagram of type (2G)}
\end{figure}

\begin{proof}
    By Lemma~\ref{l_bounding_supp=2}, $j\notin I_3$, so $I_3(\beta,j)$ is empty. Moreover, since $|\supp(j)|=2$, by Lemma \ref{l_description_chains_I2} $I_2(\beta,j)$ is a Dynkin segment of type $G$, with $j$ as special vertex. If $I_2(\beta,j)\neq \{j\}$, by Lemma \ref{l_description_chains_In} 4), the other leaf of $I_2(\beta,j)$ has support of size three, hence is not a leaf in $I_1(\beta)$.
\end{proof}

\begin{Proposition}\label{Classification: supp=3}
    Let $j\in I_2(\beta)$ be such that $\deg(j)=3$ and $|\supp(j)|=3$. Then, one of the following is satisfied:
    \begin{itemize}
        \item[(3S)] $I_3(\beta,j)=\emptyset$ and $I_2(\beta,j)\setminus \{j\}$ is the disjoint union of two non-empty Dynkin diagrams of type $A$. Moreover the leaves of $I_2(\beta,j)$ are not leaves in $I_1(\beta)$.
        \item[(3C)] $(I_n(\beta,j))_{n\geq 2}$ is a decreasing sequence of Dynkin diagrams of type $C$ with leaf $j$. Leaves of $I_2(\beta,j)$  are not leaves of $I_1(\beta)$.
        \item[(3F)] Let $j_2'\in \supp(j)$ be such that $a_{j_2',j}=-1$. Then $j_2'$ is a leaf of both $I_1(\beta)$ and $I_2(\beta,j)$, $I_2(\beta,j)\setminus \{j'_2\}$ is of type $C$, $I_3(\beta,j)$ is of type $C$ and $I_4(\beta,j)=\emptyset$. The leaves of $I_3(\beta,j)$ (resp. $I_2(\beta,j)\setminus \{j_2'\}$) are not leaves in $I_2(\beta,j)$ (resp. $I_1(\beta)$).
    \end{itemize}
\end{Proposition}

\begin{figure}[H]
    \centering
    \includegraphics[width=\linewidth]{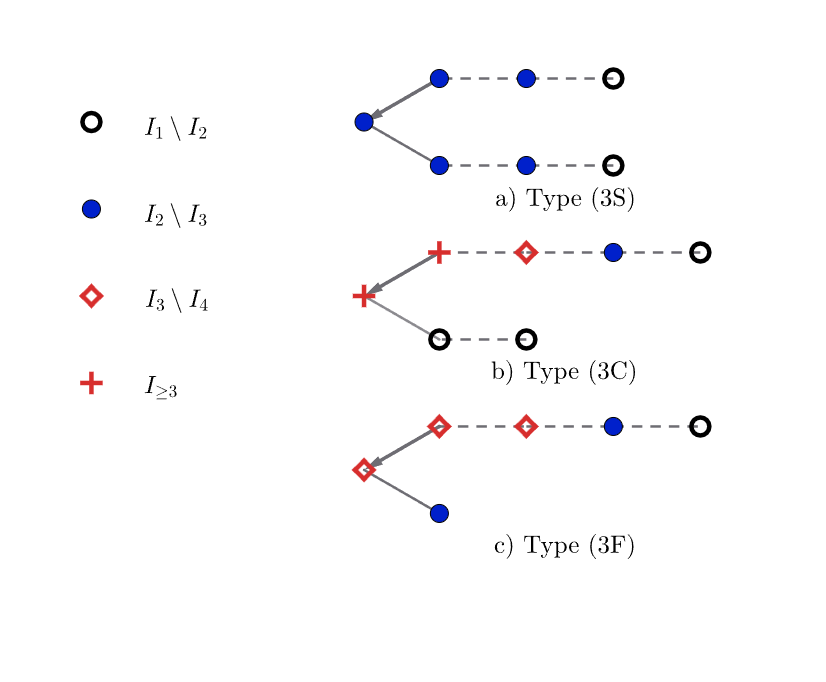}
    \caption{Dynkin sequences of type (3)}
    \label{fig: Dynkin diagram Type 3}
\end{figure}

\begin{proof}
    This is a reformulation of Lemma \ref{Lemma: class 3}. More precisely, write $\supp(j)=\{j,j_2,j_2'\}$ with $a_{j_2,j_1}=-2$ and $a_{j_2',j_1}=-1$. By Proposition \ref{p_connected_components_I2}, $I_2(\beta,j)\setminus \{j\}$ is the disjoint union of two diagrams of type $A$. The case 1) of Lemma  \ref{Lemma: class 3} therefore corresponds to type (3S), the case 2) corresponds to the type (3F), and the case 3) corresponds to type (3C). Moreover if $j_2\in I_2(\beta,j)$, in every case we can apply Lemma \ref{l_description_chains_In} 4-5) to the chain starting by $j_1 \lhd_2 j_2$ to deduce that, for every $n\geq 2$ leaves of $I_n(\beta,j)$ are not leaves of $I_{n-1}(\beta,j)$ (except $j_2'$ for $n=2$ in type (3F), and $j$ for $n\geq 3$ in type (3C)).
\end{proof}
\begin{Proposition}\label{Classification supp=4}
    Let $j_1\in I_2(\beta)$ be such that $\deg(j_1)=3$ and $|\supp(j_1)|=4$. Then the Dynkin sequence $(I_n(\beta,j_1))_{n\geq 2 }$ is of one of the following forms.
    \begin{itemize}
        \item[(4S)] The diagram $I_2(\beta,j_1)\setminus \{j_1\}$ is the disjoint union of three non-empty diagrams of type $A$. Moreover the leaves of $I_2(\beta,j_1)$ are not leaves in $I_1(\beta)$, and $I_3(\beta,j_1)=\emptyset$.

        \item[(4A)] The Dynkin sequence $(I_n(\beta,j_1))_{n\geq 2}$ is a decreasing sequence of diagrams of type $A$ such that, for all $n \geq 2$, the leaves of $I_n(\beta,j_1)$ are not leaves of $I_{n-1}(\beta,j_1)$.

        \item[(4D)]\label{Class 4D} Let $K=N_{j_1}(\beta)=\max \{n\in \Z_{\geq 0}\mid j_1\in I_K(\beta)\}$. Then $K\geq 5$  and there exist $j_2',j_2''\in \supp(j_1)$   such that:
        \begin{itemize}
            \item $j_2''$ is a leaf in $I_1(\beta)$ and $j_2'$ is a leaf of $I_2(\beta,j_1)$ but not of $I_1(\beta)$

\item $(I_n(\beta,j_1))_{n \in \llbracket 1,K\rrbracket}$ is a decreasing sequence of diagrams. If $K$ is even, then $I_{n}(\beta,j_1)$ is of type $D$ for $n\in \llbracket 1,K/2\rrbracket$  and $I_{n}(\beta,j_1)$ is of type $A$ for $n\in \llbracket K/2+1,K\rrbracket$. If $K$ is odd, there exists $\eta\in \{0,1\}$ such that $I_{n}(\beta,j_1)$ is of type $D$ for $n\in \llbracket 1,(K-1)/2+\eta\rrbracket$ and $I_{n}(\beta,j_1)$ is of type $A$ for $n\in \llbracket (K-1)/2+\eta+1,K\rrbracket$.

\item $j_2'$ and $j_1''$ are leaves of the diagrams containing them. For $n\in \llbracket 1,K-1\rrbracket$, if $\tilde{j}$ is a leaf of $I_{n}(\beta,j_1)$ and $\tilde{j}\notin\{j_1,j_2',j_2''\}$, then $\tilde{j}\notin I_{n+1}(\beta,j_1)$. 

\item $I_{K-1}(\beta,j_1)\neq \{j_1\}.$

        \end{itemize}

        \item[(4EA)] $I_2(\beta,j_1)$ is a Dynkin diagram of type $E$, and there exist $j_2',j_2''\in \supp(j_1)$ such that:
        \begin{itemize}
            \item $j_2''$ is a leaf in $I_1(\beta)$ and does not lie in $I_3(\beta,j_1)$ 
            \item $\supp(j_2')=\{j_1,j_2',j_3'\}$ for some leaf $j_3'$ of both $I_1(\beta)$, $I_2(\beta)$ and $j_3'\notin I_3(\beta,j_1)$
            \item $(I_n(\beta,j_1))_{n\in \{3,4,5\}}$ is a decreasing sequence of diagrams of type $A$, $j_2'$ is a leaf of $I_3(\beta,j_1)$ and $j_2'\notin I_5(\beta,j_1)$
            \item $I_7(\beta,j_1)=\emptyset$ and, if $I_5(\beta,j_1)=\emptyset$ then $j_2'\notin I_4(\beta,j_1)$
            \item For $n\geq 1$, if $j$ is a leaf of $I_n(\beta,j_1)$ and $j\notin \{j_1,j_2',j_3',j_2''\}$ then $j\notin I_{n+1}(\beta)$.
            \end{itemize}
        \item[(4ED)]  $I_2(\beta,j_1)$ is a Dynkin diagram of type $E$, and there exist $j_2',j_2''\in \supp(j_1)$ such that:
        \begin{itemize}
            
            \item $\supp(j_2')=\{j_1,j_2',j_3'\}$ for some leaf $j_3'$ of $I_1(\beta)$ and $j_3'\notin I_3(\beta,j_1)$
            \item $I_3(\beta,j_1)$ is a diagram of type $D$, two of its leaves are $j_2',j_2''$ and $j_2''\notin I_4(\beta,j_1)$
            \item $I_4(\beta,j_1)$, $I_5(\beta,j_1)$ are non empty diagrams of type $A$, $j_1$ is a leaf of $I_5(\beta,j_1)$
            \item $I_7(\beta,j_1)=\emptyset$
            \item Except in the case (4ED2) below, for $n\geq 1$ if $j$ is a leaf of $I_n(\beta,j_1)$ and $j\notin \{j_1,j_2',j_3',j_2''\}$ then $j\notin I_{n+1}(\beta,j_1)$.
        \end{itemize}
        Moreover, exactly one of the following holds:
        \begin{itemize}
            \item[(4ED1)] $j_2' \in I_4(\beta,j_1)$, $I_6(\beta,j_1)=\emptyset$.
            \item[(4ED2)] $j_2' \in I_4(\beta,j_1)$, $I_6(\beta,j_1)=\{j_1\}$ and there exists $n_0\in \llbracket 2,6\rrbracket$ such that $|I_{n-1}(\beta,j_1)\setminus \big(I_{n}(\beta,j_1)\cup \{j_2',j_2'',j_3'\}\big)|=1$ for $n\in \llbracket n_0+1,5\rrbracket$, and $I_{n_0-1}(\beta,j_1) \setminus \{j_2',j_2'',j_3'\}=I_{n_0}(\beta,j_1)\setminus \{j_2',j_2'',j_3'\}$. Moreover for $n<n_0$, the leaf of $I_n(\beta,j_1)$ not in $\{j_1,j_2',j_2'',j_3'\}$ is not a leaf of $I_{n-1}(\beta,j_1)$. In particular if $n_0=2$ then $\beta$ is the highest root of an $E_8$-subsystem of $\Phi$. 
            \item[(4ED3)] $j_2'\notin I_4(\beta,j_1)$ and $I_5(\beta,j_1)\neq I_6(\beta,j_1)$.
        \end{itemize}
        \item[(4SA)] There is a unique $j_2''\in \supp(j_1)$ which is a leaf of $I_2(\beta,j_1)$ and it is the unique leaf of $I_2(\beta,j_1)$ which is also a leaf of $I_1(\beta,j_1)$; $I_3(\beta,j_1)$ is a non-empty diagram of type $A$. Moreover exactly one of the following holds:
        \begin{itemize}
            \item[(4SA1)] The vertex $j_1$ is not a leaf of $I_3(\beta,j_1)$. Then $I_4(\beta,j_1)=\emptyset$, and leaves of $I_3(\beta,j_1)$ are not leaves of $I_2(\beta,j_1)$.
            \item[(4SA2)] The vertex $j_1$ is a leaf of $I_3(\beta,j_1)$. Then $I_4(\beta,j_1)$ is a diagram of type $A$, either empty either having $j_1$ as a leaf, and $I_5(\beta,j_1)=\emptyset$. Moreover $j_1$ is the unique leaf of $I_4(\beta,j_1)$ which is also a leaf of $I_3(\beta,j_1)$, and leaves of $I_3(\beta,j_1)$ are not leaves of $I_2(\beta,j_1)$. 
        \end{itemize} 
 \end{itemize}
\end{Proposition}

\begin{figure}[H]
\begin{center}
\includegraphics[width=\linewidth]{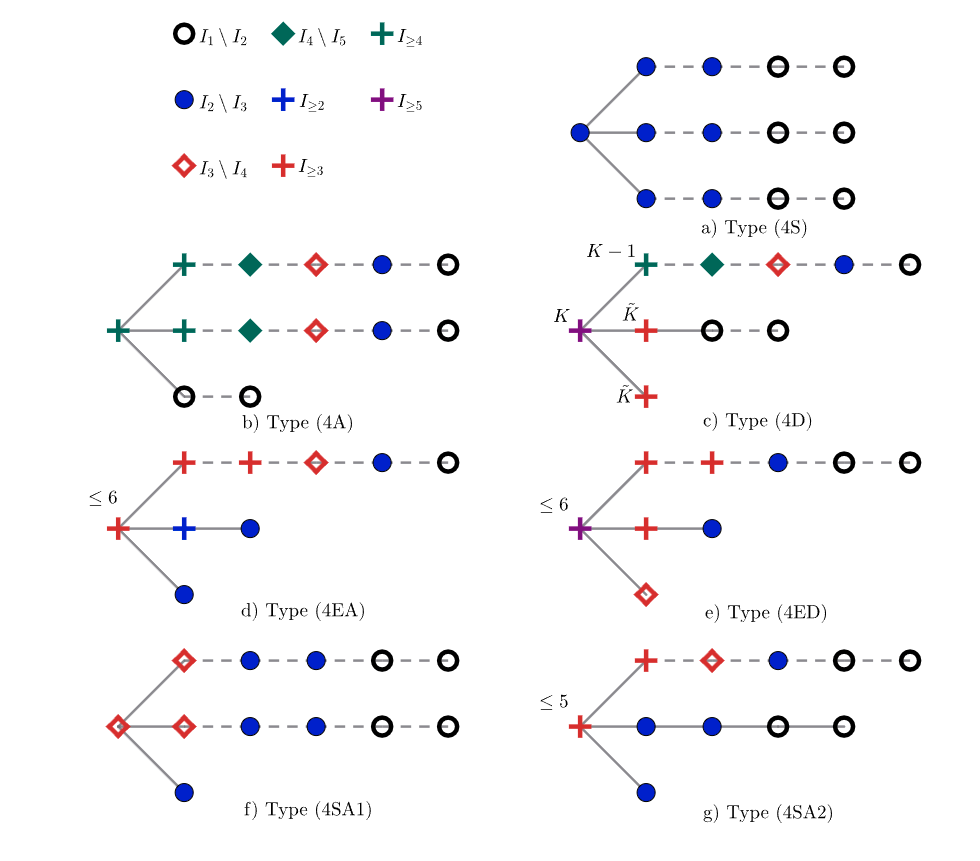}
\caption{Dynkin sequences of type (4)}
\label{fig: Dynkin sequences type 4}
\end{center}
\end{figure}

\begin{proof}
    Fix $j_1\in I_2(\beta)$  such that $|\supp(j_1)|=4$ and $\deg(j_1)=3$. Then the connected components $(I_n(\beta,j_1))_{n\geq 2}$ of $j_1$ in $I_n(\beta)$ form a sequence of Dynkin diagrams ordered for the inclusion. By Proposition \ref{p_connected_components_I2}, $I_2(\beta,j_1)\setminus \{j_1\}$ (and therefore by Lemma \ref{l_description_chains_In}, $I_n(\beta,j_1)\setminus\{j_1\}$ for all $n\geq 2$) is a disjoint union of $3$ diagrams of type $A$ (possibly empty), each one starting at a distinct element of $\supp(j_1)\setminus\{j_1\}$. Moreover, Lemmas \ref{Lemma: class 4A} to \ref{l_impossible_scheme} give more information on their forms:
    \begin{enumerate}
        \item If $\supp(j_1)\not\subset I_2$, then by Lemma \ref{Lemma: class 4A}, the sequence of Dynkin diagrams $(I_n(\beta,j_1))_{n\geq 2}$ is of the form (4A).
        \item Else, if $|\supp(j_1)|= 3$  for all $j\in \supp(j_1)\subset I_2$, by Lemma \ref{Lemma: Class 4S} it is of the form (4S). Moreover, by Lemma~\ref{l_description_chains_In}  4-6), the leaves of $I_2(\beta,j_1)$ are not leaves of $I_1(\beta)$.

        \item Else, if there is $j\in \supp(j_1)$ such that $\supp(j_1)\not\subset I_2$, by Lemma \ref{Lemma: class 4D}, it is of the form (4D). 
        Indeed, let $K=N_{j_1}(\beta)$. First assume that $K$ is even. Let  $\tilde{K}=K/2$. Then by \eqref{eq: type 4D}, we have \[\bt(j_2',\tilde{K})\leq \bt(j_2'',\tilde{K})\leq \bt(j_1,2\tilde{K})\leq \bt(j_2',\tilde{K}+1)\leq \bt(j_1,2\tilde{K}+1)=\infty=\bt(j_2'',\tilde{K}+1),\] with $\bt(j_1,2\tilde{K})<\infty$.

We have $I_{\tilde{K}}(\beta,j_1)\supset \{j_2',j_2''\}$ and $j_{2}''\notin I_{\tilde{K}+1}(\beta,j_1)$. Therefore $I_{\tilde{K}}(\beta,j_1)$ is of type $D$ and $I_{\tilde{K}+1}(\beta,j_1)$ is of type $A$. By Lemma~\ref{Lemma: class 4D}, $j_2\in I_{2\tilde{K}-1}(\beta,j_1)\neq \{j_1\}$.

        Assume now that $K$ is odd. Set $\tilde{K}=(K-1)/2$. Then by \eqref{eq: type 4D}, we have
        \[\bt(j_2'',\tilde{K})\leq \bt(j_2',\tilde{K}+1)\leq \bt(j_1,2\tilde{K}+1)\leq \bt(j_2'',\tilde{K}+1)\leq \bt(j_1,2\tilde{K}+2)=\infty,\] 
        \[\bt(j_1,2\tilde{K}+1)<\infty \text{ and }\infty=\bt(j_1,2\tilde{K}+2)=\bt(j_1',\tilde{K}+2)=\bt(j_1'',\tilde{K}+2).\] Therefore $j_2'\in I_{\tilde{K}+1}(\beta,j_1)$, $j_2',j_2''\notin I_{\tilde{K}+2}(\beta,j_1)$ and thus $I_{\tilde{K}+1}(\beta,j_1)$ is of type $D$ or $A$ and $I_{\tilde{K}+2}(\beta,j_1)$ is of type $A$. By Lemma~\ref{Lemma: class 4A}, $j_2\in I_{2\tilde{K}}(\beta,j_1)\neq \{j_1\}.$

        \item Else, we can apply Lemma \ref{Lemma: class 4E}. If $j_1\notin I_4(\beta,j_1)$ we are in the case 1) of Lemma \ref{Lemma: class 4E} and we obtain the form (4SA1), else if $j_2'\notin I_3(\beta,j_1)$ we are in the cases 1) or 2) of Lemma \ref{Lemma: class 4E} and we obtain the form (4SA2).

        Finally case 3) of Lemma \ref{Lemma: class 4E} corresponds to the forms (4EA) or (4ED), according to whether $j_2''$ lies in $I_3$. We obtain the more specific forms (4ED1) to (4ED3) through a more careful study in these cases, as it is done in the proofs of Lemma \ref{Lemma: class 4E} and Lemma \ref{l_impossible_scheme}, we now develop these cases.  
        
       Recall the notation of Lemma~\ref{Lemma: class 4E} for $j_3'$. Assume by contradiction that $j_1\notin I_5(\beta,j_1)$ and $j_2'\in I_4(\beta,j_1)$. Then as we are in the case 3) of Lemma~\ref{Lemma: class 4E}, $j_3'\notin I_3(\beta,j_1)$. But then if $t=\bt(j_2',4)$, we have $\alpha_{j_2'}(\gamma_t^\vee)=6-N_{j_1}(\gamma_t^\vee)-N_{j_3'}(\gamma_t^\vee)=-1$, with $N_{j_3'}(\gamma_t^\vee)\leq 2$ and thus $N_{j_1}(\gamma_t^\vee)\geq 5$: $j_1\in I_5$, a contradiction. Therefore if $j_1\notin I_5$, we have $j_2\notin I_4(\beta,j_1)$. Moreover by Lemma~\ref{l_relation_Nk-1_Nk}, we have $j_2''\notin I_3(\beta,j_1)$. Thus if $j_1\notin I_5$, we have $j_2'\notin I_4(\beta,j_1)$, $j_3'\notin I_3(\beta,j_1)$, $j_2''\notin I_3(\beta,j_1)$ so the sequence is of the form (4EA). Else, $j_1\in I_5(\beta,j_1)$ and by Formula~\eqref{eq: time 5+1}, there is a time $t$ such that \[(N_{j_1}(t),N_{j_2}(t),N_{j_2'}(t),N_{j_2''}(t))=(5,4,3,2).\]  We can also suppose that $N_{j_3'}(t)=2$. Indeed, $j_3'\in I_2$ so $\bt(j_3',2)$ is finite. Since $\supp(j_3')=\{j_2',j_3'\}$, by Lemma~\ref{l_relation_Nk-1_Nk} $\bt(j_2',3)<\bt(j_3',2)<\bt(j_2',4)$, $r_{j_3'}$ commutes with every other reflection applied between $\bt(j_2',3)$ and $\bt(j_2',4)$, so we can suppose that $\bt(j_3',2)=\bt(j_2',3)+1$.
        \begin{itemize}
            \item If $N_{j_2''}$ no longer increases, the sequence is of the form (4EA).
            
            \item Else, by Lemma~\ref{l_relation_Nk-1_Nk}, $N_{j_2''}$ increases before $N_{j_1}$, and if $N_{j_2'}$ does not increase the sequence is of the form (4ED3).  Indeed, by Lemma~\ref{Lemma: class 4E} 3), $\bt(j_2,n-1)\leq \bt(j_1,n)$ for $n\in \llbracket2,5\llbracket$ and it is also satisfied for $n\geq 7$, since $\bt(j_1,7)=\infty$. If $t:=\bt(j_1,6)$ is finite, since $N_{j_2'}(t)+N_{j_2''}(t)\leq 6$ necessarily $N_{j_2}(t)\geq 5$ hence $\bt(j_2,5)\leq \bt(j_1,6)$. Therefore by Lemma~\ref{l_description_chains_In} 6), the corresponding sequence $(k_n)_{n\geq 2}$ is decreasing, in particular $I_5(\beta,j_1)\neq I_6(\beta,j_1)$.

            \item Else if $N_{j_2'}$ increases and $j_1\notin I_6$, we obtain the form (4ED1). 
            
            \item The final case is if $N_{j_2'}$ increases and $j_1\in I_6(\beta,j_1)$, then there exists a time $t$ such that $(N_{j_1}(t),N_{j_2}(t),N_{j_2'}(t),N_{j_2''}(t))=(6,4,4,3)$. Following the analysis done in Lemma \ref{l_impossible_scheme}, we obtain the form (4ED2): if $N_{j_2}$ stays at $4$, it is of the form (4ED2) with $n_0=6$ (since $j_3'\notin I_3$ and $\max N_{j_1}=6$, we have $j_2'\notin I_5$).
            
            \item Else (using the same reasoning as in the proof of    Lemma~\ref{l_impossible_scheme} to obtain Formula~\eqref{e_configuration1}) 
there is a time $t$ such that \[(N_{j_1}(t),N_{j_2}(t),N_{j_3}(t),N_{j_4}(t))\in \{(6,5,3,2),(6,5,3,3)\}.\] If $j_3\notin I_4(\beta,j_1)$ we obtain the form (4ED2) with $n_0=5$, else there is a time for which Equality~\eqref{e_configuration2} holds, and the argument can be repeated. This necessarily ends for $n_0=2$, in which case $j_5\in I_2(\beta,j_1)\setminus I_3(\beta,j_1)$ has support of size $2$, and the root is the highest root of $E_8$.
        \end{itemize}
    \end{enumerate}
\end{proof}

\subsubsection{General classification of quantum roots}
\begin{Definition}\label{Definition: quantum dynkin sequence}
    Let $I$ be a Dynkin diagram, and let $(I_n)_{n\in \mathbb Z_{\geq 1}}$ be a sequence of Dynkin subdiagram. We say that $(I_n)_{n\in \mathbb Z_{\geq 1}}$ is \textbf{a quantum Dynkin sequence} if it is the Dynkin sequence of a quantum root. That is to say, $(I_n)_{n\in \Z_{\geq1}}$ is a quantum Dynkin sequence if and only if there exists $\beta\in \cQ(\Phi_+)$ such that $\beta^\vee=\sum_{n\geq 1} \sum_{i\in I_n} \alpha_i^\vee$.
\end{Definition}

In this section, we obtain a classification of quantum Dynkin sequences, and therefore a classification of quantum roots.

\begin{Definition}
Let $I$ be a Dynkin diagram and let $(I_n)_{n\in \Z_{\geq 1}}$ and $(I_n')_{n\in \Z_{\geq 1}}$ be two Dynkin sequences of $I$. We say that  $(I_n)$ and  $(I_n')$ are \textbf{mergeable} if $I_1=I_1'$ and if $I_2$, $I_2'$ are disconnected. That is to say, if $I_2\cap I_2'=\emptyset$ and if there is no edge $(i,i')$ such that $i\in I_2$, $i'\in I_2'$.
    
\end{Definition}

\begin{Lemma}\label{l_union_quantum_dynkin_diagrams}
    Suppose that $(I_n)_{n\in \Z_{\geq 1}}$ and $(I_n')_{n\in \Z_{\geq 1}}$ are two quantum Dynkin sequences. Then, if $(I_n)_{n\in \Z_{\geq 1}}$ and $(I_n')_{n\in \Z_{\geq 1}}$ are mergeable, $(I_n\cup I_n')_{n\in \Z_{\geq 1}}$ is also quantum.
\end{Lemma}

\begin{proof}
    Let $\beta\in \cQ(\Phi_+)$ (resp. $\beta'\in \cQ(\Phi_+)$ be the quantum root such that $(I_n)_{n\in \Z_{\geq 1}}=(I_n(\beta))_{n\in \Z_{\geq 1}}$ (resp. $(I_n')_{n\in \Z_{\geq 1}}=(I_n(\beta'))_{n\in \Z_{\geq 1}}$), which exists since by assumption the two sequences are quantum. Since the two sequences are also mergeable, we have $I_1=I_1'$ and this is a $1$-star convex tree by Proposition~\ref{p_characterization_I1}. 
    
    Let $\beta_0\in \cQ(\Phi_+)$ be the quantum root such that $I_1(\beta_0)=I_1$ and $I_2(\beta_0)=\emptyset$, which is given by Proposition~\ref{p_characterization_I1}. Note that, since $I_1=I_1'$ and $I_n$, $I_n'$ are disjoint for $n\geq 2$, we have $\sum_{n\geq 1}\sum_{i\in I_n \cup I_n'} \alpha_i^\vee=\beta^\vee+(\beta')^\vee-\beta_0^\vee$. It therefore suffices to prove that this is the coroot associated to a quantum root.
    
    By Lemma~\ref{l_I1_placed_beggining}, we can write $\beta=w(\beta_0)$ (resp. $\beta'=w'(\beta_0)$) where $w\in \langle r_i\mid i \in I_2\rangle \subset W^v$ (resp. $w'\in \langle r_i\mid i \in I_2'\rangle \subset W^v$) is minimal. Fix a reduced decomposition $r_{i_1}\dots r_{i_n}$ of $w'$. Note that for all $k\in \llbracket 1,n\rrbracket$, $r_{i_k}\dots r_{i_n}(\beta_0)$ is quantum with coroot $\beta_0^\vee +\sum_{j=k}^n \alpha_{i_j}^\vee$ (in particular $(\beta')^\vee=\beta_0^\vee +\sum_{j=1}^n \alpha_{i_j}^\vee$). Since $I_2$ and $I_2'$ are disconnected, we have $w^{-1}(\alpha_i)=\alpha_i$ for all $i\in \{i_1,\dots i_n\}\subset I_2'$. Therefore, by applying Lemma~\ref{Lemma: commutation reflections} to the triple $(\beta_0,i_n,w)$, $r_{i_n}w(\beta_0)$ is quantum with coroot $\beta^\vee+\alpha_{i_n}^\vee$. Since $r_{i_n}w(\beta_0)=wr_{i_n}(\beta_0)$ we can iterate the application of Lemma~\ref{Lemma: commutation reflections}, to the triple $(r_{i_n}(\beta_0),i_{n-1},w)$ and we obtain that $r_{i_{n-1}}r_{i_n} w(\beta_0)$ is quantum with coroot $\beta^\vee+\alpha_{i_{n-1}}^\vee+\alpha_{i_n}^\vee$. After $n$ successive applications of Lemma~\ref{Lemma: commutation reflections}, we obtain that $w'w(\beta_0)=r_{i_1}\dots r_{i_n} w(\beta_0)$ is quantum with coroot $\beta^\vee+ \sum_{j=1}^n \alpha_{i_j}^\vee=\beta^\vee +(\beta')^\vee-\beta_0^\vee$. Therefore $(I_n\cup I_n')_{n\in \Z_{\geq 1}}=\left(I_n(w'w(\beta_0)\right)_{n\in \Z_{\geq 1}}$ is a quantum Dynkin sequence.
\end{proof}

\begin{Lemma}\label{Lemma: commutation reflections}
    Suppose that $\beta \in \Phi_+$, $i \in I$ and $w\in W^v$ are such that $w^{-1}(\alpha_i)=\alpha_i$. Then if $\beta$, $r_i(\beta)$, $w(\beta)$ are quantum, $r_i w(\beta)$ is also quantum. Moreover $r_i w(\beta^\vee)=w(\beta^\vee)+r_i(\beta^\vee)-\beta^\vee$
\end{Lemma}
\begin{proof}
    Suppose that $\beta$ and $w(\beta)$ are quantum. By Proposition~\ref{Proposition: simplereflectionquantumroot}, $r_i w(\beta)$ is quantum if and only if $|\langle w (\beta^\vee), \alpha_i\rangle|=|\langle  \beta^\vee, w^{-1}(\alpha_i)\rangle|\leq 1$. However, $w^{-1}(\alpha_i)=\alpha_i$. Therefore this is satisfied if and only if $|\langle  \beta^\vee, \alpha_i\rangle|\leq 1$ that is to say, by Proposition~\ref{Proposition: simplereflectionquantumroot}, if and only if $r_i(\beta)$ is quantum. 
    
    Moreover since $r_i(\beta^\vee)-\beta^\vee=-\langle \beta^\vee,\alpha_i\rangle \alpha_i^\vee$, we check that \[r_iw(\beta^\vee)=w(\beta^\vee)-\langle w(\beta^\vee),\alpha_i\rangle \alpha_i^\vee= w(\beta^\vee)+r_i(\beta^\vee)-\beta^\vee.\]
\end{proof}

\begin{Theorem}\label{Theorem: Classification quantum roots}
    Let $I$ be a Dynkin diagram with associated root system $\Phi$, and let $(I_n)_{n\in \mathbb Z_{\geq 1}}$ be a sequence of Dynkin subdiagrams of $I$. Then there exists a quantum root $\beta\in \cQ(\Phi_+)$ such that $(I_n)_{n\geq 1}=(I_n(\beta))_{n\geq 1}$ (see Definition \eqref{eq: I_n definition}) if and only if the sequence $(I_n)_{n\in \mathbb Z_{\geq 1}}$ satisfies the following properties:
                \begin{enumerate}
        \item For any $n\in \mathbb Z_{\geq 1}$, $I_{n+1}$ is a subdiagram of $I_n$.
        \item The diagram $I_1$ is a $1$-star-convex Dynkin tree.
        \item For any $n\geq 2$, for any vertex $j\in I_n$ we have $\deg(j)\leq 3$ (see Formula~\eqref{e_definition_degree}). Moreover each connected component of $I_n$ admits a unique vertex of degree $3$. 
        
        \item For any vertex $j\in I_2$ such that $\deg(j)=3$, let $I_n(j)$ denote the corresponding connected component (with $I_n(j)=\emptyset$ if $j\notin I_n$), given by point 3. Then:
        \begin{enumerate}
            \item If $|\supp(j)|=2$, then $(I_n(j))_{n\geq 2}$ is a sequence of type (2G) (see Proposition \ref{Classification: supp=2}).
            \item If $|\supp(j)|=3$, then $(I_n(j))_{n\geq 2}$ is a sequence of type (3S), (3C) or (3F) (see Proposition \ref{Classification: supp=3}).
            \item If $|\supp(j)|=4$, then $(I_n(j))_{n\geq 2}$ is a sequence of type (4S), (4A), (4D), (4EA), (4ED) or (4SA) (see Proposition \ref{Classification supp=4}).
        \end{enumerate}
    \end{enumerate}

\end{Theorem}

\begin{proof}
Let $\beta\in \cQ(\Phi_+)$. Then by the analysis done in the preceding sections, $(I_n(\beta))_{n\geq 1}$ satisfies properties 1) to 4): this is  Proposition~\ref{p_characterization_I1} for point 2) , Proposition~\ref{p_connected_components_I2} for point 3) and Propositions~\ref{Classification: supp=2},~\ref{Classification: supp=3} and~\ref{Classification supp=4} for point 4).  Therefore it remains to associate a quantum root $\beta$ to each Dynkin sequence sequence satisfying the properties 1 to 4.

    Let $(I_n)_{n\geq 1}$ be a quantum Dynkin sequence. Let $J_3:=\{j\in I_2 \mid \deg(j)=3\}$. For any $j\in J_3$ and $n\geq 1$, as in Definition~\ref{Definition: quantum dynkin sequence} let $I_n(j)$ denote the connected component of $j$ in $I_n$, which is of one of the type given by Definition~\ref{Definition: quantum dynkin sequence} 4), and either $I_n(j)=\emptyset$, either $I_n(j)\cap J_3=\{j\}$. We have $I_1(j)=I_1$ for every $j\in J_3$.  The graphs $I_2(j)$ and $I_2(j')$ are disconnected for  $j,j'\in J_3$ such that $j\neq j'$, and thus $(\bigcup_{j\in E_1} I_n(j))_{n\in \Z_{\geq 1}}$ and $(\bigcup_{j\in E_2} I_n(j))_{n\in \Z_{\geq 1}}$ are mergeable  for every subsets $E_1,E_2$ of $J_3$ such that  $E_1\cap E_2=\emptyset$. Therefore using Lemma~\ref{l_union_quantum_dynkin_diagrams}, we can assume $|J_3|\leq 1$. We already proved the case where $J_3$ is empty in Proposition~\ref{p_characterization_I1}.
    
    We now assume that $J_3$ is a singleton $\{j\}$. We have $I_n(j)=I_n$ for all $n\in \Z_{\geq 1}$.

    By Proposition~\ref{p_characterization_I1} there exists $\beta_\emptyset \in \cQ(\Phi_+)$ such that $I_1(\beta_\emptyset)=I_1$ and $I_n(\beta_\emptyset)=\emptyset$ for $n\geq 2$. Note that, since $(I_n)_{n\geq 2}$ is amongst the types given in Propositions~\ref{Classification: supp=2},~\ref{Classification: supp=3} and~\ref{Classification supp=4}, it has the same properties as the connected component $\mathcal C$ in Proposition~\ref{p_connected_components_I2} and Remark~\ref{r_connected_components_In}. That is to say, with $m=|I_2\cap \supp(j)|-1$, $I_2\setminus \{j\}$ is the disjoint union of $m$ non-empty Dynkin segments of type $A$, one for each neighbour of $j$ in $I_1$. As in Proposition~\ref{p_connected_components_I2}, let us write $(I_2\cap\supp(j))\setminus \{j\} = \{j_2^{(i)}\mid i \in \llbracket1,m\rrbracket\}$ and, for all $i\in \llbracket 1,m \rrbracket$ write $j_2^{(i)},j_3^{(i)},\dots, j_{k(2,i)}^{(i)}$ the ordered elements of the segment starting at $j_2^{(i)}$. Therefore, $k(2,i)-1$ is the length of the segment of type $A$ starting at $j_2^{(i)}$ in $I_2\setminus \{j\}$. Upon relabelling, we may suppose $k(2,i)\geq k(2,i+1)$ for all $i \in \llbracket 1,m-1\rrbracket$. 
    Moreover by Remark~\ref{r_connected_components_In}, $[j_2^{(i)},j_{k(2,i)}^{(i)}]\cap I_n=[j_2^{(i)},j_{k(n,i)}^{(i)}]$ for some non-increasing sequence $(k(n,i))_{n\geq 2}$ (with the convention that $k(n,i)=1$ if $j_2^{(i)}\notin I_n$, and $[j_2^{(i)},j_1^{(i)}]=\emptyset$). 

    Finally, for each $i \in \llbracket1,m\rrbracket$, $n\geq 2$, let us define the following element of $W^v$: 
    \begin{equation}
        w^{(i)}_n=\begin{cases}r_{j_{k(n,i)}^{(i)}}\dots r_{j_2^{(i)}} \text{ if } k(n,i)\geq 2 \\
        1_{W^v} \text{ otherwise.}
        
        \end{cases}, \;
        w_n=(\prod_{i=1}^m w_n^{(i)})r_j
    \end{equation}
With these notations, we can now construct $\beta$ from $\beta_\emptyset$.  
\medskip

Suppose for now that $(I_n)_{n\geq 2}$ is of the form (2G), (3S), (3C), (4S) or (4A). In all these cases, $j^{(i)}_{k(n,i)}$ is not a leaf of $I_{n-1}$ so the sequences $(k(n,i))_{n\geq 2}$ are decreasing until they reach $1$, and $|\supp(j^{(i)}_{k(2,i)})|=3$.

By Proposition~\ref{Proposition: simplereflectionquantumroot}, as $\deg(j)=3$,  $r_j(\beta_\emptyset)$ is quantum, and $I_2(r_j(\beta_\emptyset))=\{j\}$. If $m=0$, $\supp(j)\cap I_2=\{j\}$, so we are done. Else, recursively applying Proposition~\ref{Proposition: simplereflectionquantumroot} $\ell(w_2)$ times, we similarly check that  $w_2(\beta_\emptyset)$ is quantum (we crucially use that $|\supp(j^{(i)}_{k(2,i)})|=3$ when we apply $r_{j_{k(2,i)}^{(i)}}$). Indeed, for any $i\in \llbracket1,m\rrbracket$ and $p\in \llbracket2,k(2,i)\rrbracket$, at time $t:=\bt(j_p^{(i)},2)$, we have $(N_{j_{p-1}^{(i)}}(t),N_{j_{p}^{(i)}}(t),N_{j_{p+1}^{(i)}}(t))=(2,1,1)$ therefore we can apply Proposition~\ref{Proposition: simplereflectionquantumroot}.  Moreover $I_2(w_2(\beta_\emptyset))=I_2$ and $I_3(w_2(\beta_\emptyset))=\emptyset$. We iterate the process: if $j\in I_n$ we apply $w_n$ to $w_{n-1}\dots w_2(\beta_ \emptyset)$, which is quantum, and we obtain a quantum root (we need to use that $k(n,i)<k(n-1,i)$ if $w_n^{(i)}\neq 1_{W^v}$). When we reach $n_0$ such that $I_{n_0+1}=\emptyset$ and $I_{n_0}\neq \emptyset$, we stop the process, and $\beta:=w_{n_0}\dots w_2(\beta_\emptyset)$ is a quantum root such that $I_n(\beta)=I_n$ for all $n\geq 2$.
\medskip

It remains to deal with the forms (3F), (4D), (4EA), (4ED), (4SA). Let us start by the form (4D), it is the only one amongst these for which $I_n$ may be non empty for arbitrary large $n$". Let $K=\max\{n\in \Z\mid I_n\neq \emptyset\}$. Let  $\tilde{K}=\lfloor (k-1)/2\rfloor$.
In this case $m=3$ but $k(2,2)=k(2,3)=2$, and upon relabelling we can suppose that $j_2^{(3)}$ is a leaf of $I_1$, but $j_2^{(2)}$ is not. 

By construction, the root $\beta^0:=\beta_\emptyset$ is quantum and has the property~\eqref{eq: case4dthminit}:
\begin{align*}
    \tag{$H_0$}\label{eq: case4dthminit}
    I_1(\beta^0)=I_1,\, I_2(\beta^0)=\emptyset.
\end{align*}

For $p\in \llbracket 1,\tilde{K}-1\rrbracket$, we say that a root $\beta^p$ has the property~\eqref{eq: case4dthm} if:
\begin{align*}
    \tag{$H_p$}\label{eq: case4dthm}
N_{\!j_k^{(1)}}(\beta^p)=N_{j}(\beta^p)=2p+1\text{ for } k\in \llbracket 1,k(2p+1,1)\rrbracket, N_{j_{k(2p+1,1)+1}^{(1)}}(\beta^p)=2p\\
\text{ and } N_{\!j_2^{(i)}}(\beta^p)=p+1 \text{ for } i\in \{2,3\}.
\end{align*} 
 Suppose that $p\in \llbracket0,\tilde{K}-1\rrbracket$ is such that $\beta^p$ is constructed, quantum and satisfies \eqref{eq: case4dthm}.

 Then using the fact that $(k(n,1))$ is decreasing and applying Proposition~\ref{Proposition: simplereflectionquantumroot} we check that $\beta^{p+1}:=r_{j_2^{(3)}}w^{(1)}_{2p+3}r_jr_{j_2^{(2)}}w^{(1)}_{2p+2}r_j(\beta^{p})$ is a quantum root which satisfies $(H_{p+1})$. We obtain $\beta^{\tilde{K}}$ satisfying $(H_{\tilde{K}})$. Then if $K$ is odd we are done. Otherwise, we set $\beta=r_j(\beta^{\tilde{K}})$, which concludes the case (4D).
 
 In the remaining cases, $I_n(j)=\emptyset$ for all $n\geq 7$, so they can be checked by hand in a similar fashion, following the arguments in the proofs of Lemma~\ref{Lemma: class 3} (for the case (3F)) or Lemma~\ref{Lemma: class 4E} (for the cases  (4EA), (4ED), (4SA)). To have a complete proof, we now explicit these verifications.
 
In the case (3F), $m=2$ and upon relabelling, suppose that $(a_{j_2^{(1)},j},a_{j_2^{(2)},j})=(-2,-1)$. Then we check that $\beta:=r_{j^{(2)}_2}w^{(1)}_3r_jw^{(1)}_2r_j(\beta_\emptyset)$ is a quantum root such that $I_n(\beta)=I_n$ for all $n\geq 2$.

For the case (4SA), $m=3$ and denote $j_2^{(3)}$ the unique leaf of $I_2$ in $\supp(j)$, by definition it is also a leaf of $I_1$. In the case (4SA2), since $j$ is a leaf of $I_3$, there is another $j_2^{(2)}\in \supp(j)$ which does not belong to $I_3$. We set $w=r_{j^{(3)}_2}w^{(2)}_3 w^{(1)}_3r_jw^{(2)}_2w^{(1)}_2r_j$. Then, in the case (4SA1), $\beta:=w(\beta_\emptyset)$ is the required quantum root. In the case (4SA2), we instead set $\beta:=w^{(1)}_4r_jw(\beta_\emptyset)$ (note that $w^{(2)}_3=1_{W^v}$ hence $\langle w(\beta_\emptyset^\vee), \alpha_j\rangle =-1$ as required to apply Proposition~\ref{Proposition: simplereflectionquantumroot}).

It remains to deal with the cases (4EA) and (4ED). Again $m=3$ and by definition we can suppose  that $j_2^{(3)},j_3^{(2)}$ are leaves of $I_1$ (they correspond to $j_2''$ and $j_3'$ respectively, with the notation of Proposition~\ref{Classification supp=4}). Suppose that $(I_n)_{n\geq 2}$ is not of the form (4ED2).

Let $$\beta^1:=r_{j^{(2)}_3}r_{j^{(2)}_2}w^{(1)}_4r_jr_{j_2^{(3)}}w^{(1)}_3r_jr_{j_2^{(2)}}w^{(1)}_2r_j(\beta_\emptyset),$$ 
then by Proposition~\ref{Proposition: simplereflectionquantumroot}, $\beta^1$ is a quantum root and satisfies $(N_j,N_{j^{(1)}_2},N_{j^{(2)}_2},N_{j^{(3)}_2},N_{j^{(2)}_3})=(4,4,3,2,2)$. If $(I_n)_{n\geq 2}$ is of the form (4EA) and $j\notin I_5$, then we set $\beta:=\beta^1$. If $(I_n)_{n\geq 2}$ is of the form (4EA) but $j\in I_5$, we instead set $\beta:=w^{(1)}_5r_j (\beta^1)$. It remains to treat the cases (4ED).

Let $\beta^2:=w_4^{(2)} r_{j^{(3)}_2}w^{(1)}_5r_j (\beta^1)$. If $j\notin I_5$ (hence in the case (4ED1) or a particular case of (4ED3)), we set $\beta:= \beta^2$. Else $j\in I_5$, we set $\beta:= r_j(\beta^2)$ (note that this is indeed a quantum root by Proposition~\ref{Proposition: simplereflectionquantumroot} because either $w^{(2)}_4=1_{W^v}$ either or $w^{(1)}_5=1_{W^v}$). 

The only remaining case is the form (4ED2). It corresponds to the analysis done in the proof of Lemma~\ref{l_impossible_scheme}, and we construct $\beta$ following this proof: Recall the notation for $n_0\in \llbracket2,6\rrbracket$, such that the leaf $j_{k(n_0,1)}^{(1)}$ of $I_{n_0}$ is also a leaf of $I_{n_0-1}$, note that $k(n_0,1)=7-n_0$. Let us start by $\beta_0$ such that $(I_n(\beta_0))_{n\geq 0}$ is of the form (4ED1), with $I_k(\beta_0)=I_k$ for $k\leq n_0-1$ and $I_k(\beta_0)\setminus\{j_2^{(2)},j_3^{(2)},j_2^{(3)}\}=[j,j_{6-k}^{(1)}]$ for $k\in \llbracket n_0,6\rrbracket$ (in particular $I_6(\beta_0)=\emptyset$ and $I_5(\beta_0)=\{j\}$). Then $\alpha_j(\beta_0^\vee)=-1$ so $r_j(\beta_0)$ is quantum with Dynkin sequence of the form (4ED2). If $n_0=6$, we can set $\beta:= r_j(\beta_0)$. Else we check that $\beta:=r_{j_{7-n_0}^{(1)}}\dots r_{j_2^{(1)}} r_j(\beta_0)$ is a quantum root with Dynkin sequence $(I_n)_{n\geq 1}$.

This concludes the proof of Theorem~\ref{Theorem: Classification quantum roots}.
\end{proof}

\subsection{Examples of sets of quantum roots}
As shown by Lemma~\ref{Lemma : form of covers} and Proposition~\ref{p_almost_simple_roots_covers}, quantum roots naturally appears in the study of covers for the affinized Bruhat order.
In this subsection, we determine the sets of quantum roots in some particular cases (type $A_n$, $A_n^{(1)}$, $D_n^{(1)}$, etc.). Besides giving examples, we want to illustrate that the methods developed in this section to prove Theorem~\ref{t_finiteness_almost_simple_roots} are constructive and enable to determine the set of quantum roots in particular examples.

Let $\beta\in \Phi_+$. Write $\beta=r_{i_L}\ldots r_{i_2}(\alpha_{i_1})$, with $L$ minimal, $i_1,\ldots, i_L\in I$. For $\ell\in \Z_{\geq 0}$, set  \[I_\ell(\beta)=\{j\in I\mid |\{t\in \llbracket 1,L\rrbracket \mid i_t=j\}|\geq  \ell\}.\]

\subsubsection{Finite type}

The following result was explained to us by Dinakar Muthiah.

\begin{Proposition}\label{p_quantum_ADE}
Assume that the root system $\Phi$ is  of type A, D or E. Then the set of quantum roots of $\Phi$ is $\Phi_+$.
\end{Proposition}

\begin{proof}
By \cite[Planche I to X]{bourbaki1981elements}, we can embed $\Phi$ in $\R^{n}$, for some $n\in \N$, equipped with the canonical scalar product $(\cdot|\cdot)$ and $\Phi$ is a subset of elements  of norm $\sqrt{2}$.  Let $\alpha\in \Phi$.  We have $\alpha^\vee=2\alpha/(\alpha|\alpha)=\alpha$ for every $\alpha\in \Phi$. Thus if $\alpha,\beta\in \Phi$, we have $\langle \alpha^\vee,\beta\rangle=(\alpha|\beta)\in \Z$ and $|(\alpha|\beta)|\leq |\alpha||\beta|\leq 2$, so $\langle \alpha^\vee,\beta\rangle\in \llbracket -2,2\rrbracket$. Now if $\alpha$ and $\beta$ are not proportional, we have $|(\alpha|\beta)|<|\alpha||\beta|$ and thus $\langle \alpha^\vee,\beta\rangle \in \llbracket -1,1\rrbracket$. For every $\alpha\in \Inv(s_\beta)$, we have $\langle \alpha^\vee,\beta\rangle>0$ (see  \cite[Corollary 1.10]{philippe2023grading}  for example) and thus $\langle \alpha^\vee,\beta\rangle=1$ for all $\alpha\in \Inv(s_\beta)$, which proves that $\beta$ is quantum.
\end{proof}

\begin{Remark}
    Proposition~\ref{p_quantum_ADE} can also be proved using Theorem~\ref{Theorem: Classification quantum roots}: Dynkin diagrams of type ADE are trees which are $1$-star-convex at any given vertex.  Fix a root system $\Phi$ of type ADE, and $\beta\in \Phi_+$. Then the Dynkin diagram $I_1(\beta)$ is of type ADE as well. If it is of type A, it has no vertex of degree 3 or larger, so $I_2(\beta)=\emptyset$. Hence $(I_n(\beta))_{n\geq 1}$ is quantum and $\beta\in \cQ(\Phi_+)$. If $I_1(\beta)$ is of type $D$ or $E$, it admits a unique vertex $j$ of degree 3 and no vertex of higher degree, so $I_2(\beta)$ is non-empty if and only if it contains $j$. We check that if this is the case, $(I_n(\beta))_{n\geq 1}$ needs to be of the form (4A), with $I_3(\beta)=\emptyset$, when $I_1(\beta)$ is of type D; or of the forms (4A), (4D),(4EA), (4ED) when $I_1(\beta)$ is of type E.
\end{Remark}

However, for every other indecomposable Cartan matrix (of finite type), there exists a positive root which is not quantum.  Indeed, there exist $i,j\in I$ such that $\langle \alpha_i^\vee,\alpha_j\rangle\leq-2$. Set $\beta=r_j(\alpha_i)$.
Then $\alpha_i^\vee\in \Inv(s_\beta)$ and $\langle \beta^\vee,\alpha_j\rangle=-\langle \alpha_i^\vee,\alpha_j\rangle\geq 2$ and thus $r_j(\alpha_i)$ is not quantum. The set of quantum roots can in fact be explicitly described using the notions of long and short roots \cite[Lemma 4.2]{lenart2012uniform}.

Note that when $\Phi$ is infinite, then $\Phi_+$ strictly contains the set of quantum roots, which is finite.

\subsubsection{Case where the generalized Cartan matrix has no coefficient $-1$}\label{ss_no_minus_one}

\begin{Lemma}\label{l_no_-1_coefficient}
Let $\beta\in \Phi_+$ be a quantum root.  Write $\beta=r_{i_L}\ldots r_{i_{2}}(\alpha_{i_1})$, with $i_1,\ldots,i_L\in I$ and $L\in \N$ minimal. Then we have $a_{i_1,i_2}=-1$ (where $A=(a_{i,j})_{i,j\in I}$ is the generalized Cartan matrix of the Kac-Moody datum). In particular, if $a_{i,j}\neq -1$ for every $i,j\in I$, the set of quantum roots is exactly the set of simple roots.

\end{Lemma}

\begin{proof}
    We have $s_{\beta}=r_{i_L}\ldots r_{i_{2}} r_{i_1} r_{i_{2}} \ldots r_{i_L}$. We have $r_{i_L}\ldots r_{i_{3}}.\alpha_{i_{2}}\in \Inv(s_{\beta})\setminus \{\beta\}$  and \[\langle r_{i_L}\ldots r_{i_2}(\alpha_{i_1}^\vee),r_{i_L}\ldots r_{i_3}(\alpha_{i_2})\rangle=-\langle \alpha_{i_1}^\vee,\alpha_{i_2}\rangle=-a_{i_1,i_2}=1. \]
\end{proof}

\subsubsection{Type $A_n^{(1)}$, $n\geq 2$}

\begin{Proposition}
We assume that the generalized Cartan matrix $A$ is  of  type $A_n^{(1)}$, with $n\in \Z_{\geq 2}$ (see \cite[Chapter 4]{kac1994infinite} for the classification).  We regard $I$ as a diagram as in Subsection~\ref{ss_gen_def_not} and we denote by $\mathrm{Seg}(I)$ the set of non-empty subsets  of $I$ which are a segment when regarded as a subdiagram of $I$. Then $\cQ(\Phi_+)=\{\sum_{i\in J} \alpha_i\mid J\in \mathrm{Seg}(I)\}$. \end{Proposition}

\begin{proof}
Let $\beta\in \cQ(\Phi_+)$. We have $\deg(j)=2$  for every $j\in I$ (we defined the degree in  \eqref{e_definition_degree}).  By Lemma~\ref{l_description_chains_I2}, we deduce that $I_2(\beta)=\emptyset$. Therefore  by Proposition~\ref{p_characterization_I1},  $\{\beta^{\vee}\mid \beta\in \cQ(\Phi_+)\}=\{\sum_{j\in J} \alpha_j^\vee\mid J\in \mathrm{Seg}(I)\}$.  As the generalized Cartan matrix only contains $0$ or $-1$ coefficients out of the diagonal, we have $\beta^\vee=\sum_{i\in I} n_i \alpha_i^\vee$, for all $\beta=\sum_{i\in I} n_i\alpha_i\in \Phi$, and the proposition follows.
\end{proof}

Note that the case $A_1^{(1)}$ is a particular case of subsubsection~\ref{ss_no_minus_one}.

\subsubsection{Type $D_n^{(1)}$, $n\geq 4$}

Let $n\in \Z_{\geq 5}$. Let $I=I_{D_n^{(1)}}$ be the  Dynkin diagram of type $D_n^{(1)}$  (see Figure~\ref{f_Dynkin_diagram_Dn_aff}). 
 \begin{figure}[h]
 \centering
 \includegraphics[scale=0.5]{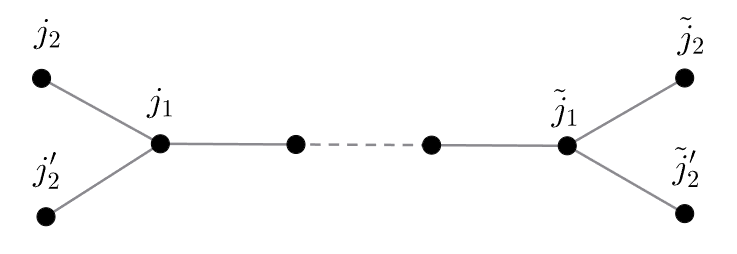}
 \caption{Dynkin diagram of type $D_n^{(1)}$}\label{f_Dynkin_diagram_Dn_aff}
\end{figure}

It admits two vertices  of adjacency $3$, which we denote $j_1, \tilde{j}_1$. Let $j_2,j_2'$ (resp. $\tilde j_2,\tilde j_2'$) denote the two leaves of $I$ in $\supp(j_1)$ (resp. $\supp(\tilde{j}_1)$). The Dynkin diagram $I$ contains four distinct Dynkin diagrams of type $D_n$, obtained by removing one of the vertices amongst $j_2,j_2',\tilde{j}_2,\tilde{j}_2'$. Let $\Phi\subset \oplus_{i\in I} \Z\alpha_i$ be the root system associated to $I$.

From Theorem~\ref{Theorem: Classification quantum roots} we deduce that $\beta\in \oplus_{i\in I} \mathbb N \alpha_i$ is a quantum root if and only if one of the following condition is satisfied:
\begin{enumerate}
    \item The element $\beta$ is a root in one of the four standard sub root systems of type $D_n$ (in particular $I_1(\beta)\neq I$ and $I_3(\beta)=\emptyset$).
    \item We have that $I_1(\beta)=I$ and $I_2(\beta)\subset I\setminus\{j_2,j_2',\tilde j_2,\tilde j_2'\}$ admits at most two connected components, each of them containing exactly one element among $\{j_1,\tilde j_1\}$. Moreover $I_3(\beta)=\emptyset$.
\end{enumerate}
Note that both of these cases correspond to Dynkin sequences of type (4A).

If $n=4$, then  $I_{D_4^{(1)}}$ consists of a unique vertex of adjacency $4$ and four leaves. Similarly a root $\beta$ is quantum if and only if either it is a root supported on a subdiagram of type $D_4$, or if $I_1(\beta)=I_{D_4^{(1)}}$ and $I_2(\beta)=\emptyset$.

\section{From quantum roots to covers for the affinized Bruhat order}\label{s_almost_simple_roots_covers}
In this section, we use the finiteness of the set of quantum roots to obtain finiteness results on the affinized Bruhat order on the Weyl semi-group associated to a Kac-Moody datum. 
We start this section by the introduction of this object and of the affinized Bruhat order. 
\subsection{Definition and notations: Affine Weyl semi-group}\label{ss_def_not_aff}

Let $\mathcal D=(A,X,Y,(\alpha_i)_{i\in I},(\alpha_i^\vee)_{i\in I})$ be a Kac-Moody root datum. Let $W^v$ be the vectorial Weyl group and $\Phi$ the root system attached to it.
\paragraph{Dominant coweights and Tits cone}

Elements of $Y$ are called coweights, a coweight $\lambda$ is \textbf{regular} if $\langle \lambda,\beta\rangle\neq 0$ for all $\beta\in \Phi$, or equivalently if its fixator in $W^v$ is trivial. We say that a coweight $\lambda\in Y$ is \textbf{dominant} if $\langle \lambda,\alpha_i\rangle\geq 0$ for all $i\in I$, or equivalently if $\langle \lambda,\beta\rangle \geq 0$ for all $\beta\in\Phi_+$. The set of dominant coweights is denoted as \index{y@$Y^{++}$, $Y^+$} $Y^{++}=\{\lambda \in Y\mid \langle \lambda,\alpha_i\rangle \geq 0\; \forall i\in I\}$ and the set of dominant regular coweights is denoted \index{Y@$Y^{++}_{reg}$}$Y^{++}_{reg}=\{\lambda \in Y^{++}\mid \langle \lambda,\alpha_i\rangle >0\; \forall i \in I\}$.

The orbit of $Y^{++}$ by $W^v$ is the \textbf{integral Tits cone}, $Y^+=\bigcup\limits_{w\in W^v} w.Y^{++}$; it is a sub-cone of $Y$ and it is a proper sub-cone if and only if $W^v$ is infinite, if and only if $A$ is not of finite type (see \cite[1.4.2]{kumar2002kac}).

The set $Y^{++}$ is a fundamental domain for the action of $W^v$ on $Y^+$, and for any coweight $\lambda \in Y^+$, we denote $\lambda^{++}$ to be the unique element of $Y^{++}$ in its $W^v$-orbit.

For $\lambda,\mu\in Y$, we set $[\lambda,\mu]_{Y}=\{\nu\in Y\mid \exists t\in \Q, \nu=(1-t)\lambda+t\mu\}$. We set $]\lambda,\mu[_{Y}=[\lambda,\mu]_Y\setminus \{\lambda,\mu\}$. 

\paragraph{Parabolic subgroups and minimal coset representatives}
Let $\lambda\in Y^{++}$ be a dominant coweight. Then the fixator of $\lambda$ in $W^v$ is a standard parabolic subgroup, which we denote by \index{w@$W_\lambda$, $W_J$, $W^\lambda$} $W_\lambda$. Explicitly, if $J=\{i\in I \mid \langle \lambda,\alpha_i\rangle=0\}$, then $W_\lambda =W_J:=\langle r_j\mid j\in J\rangle$. 

Any class $[w]\in W^v/W_\lambda$ admits a unique element of minimal length, and we denote by $W^\lambda$ the set of coset representatives of minimal length. Therefore $w\in W^\lambda \iff \ell(w)\leq \ell(\tilde w)\,\forall \tilde w \in wW_\lambda$.

Suppose now that $\lambda\in Y^+$ (not necessarily dominant), we denote by \index{v@$v^\lambda$}$v^\lambda$ the unique element of minimal length such that $\lambda=v^\lambda \lambda^{++}$, it is an element of $W^{\lambda^{++}}$. The coweight $\lambda$ also admits a parabolic fixator in $W^v$, which is $W_\lambda = v^\lambda W_{\lambda^{++}} (v^\lambda)^{-1}$.

A coweight $\lambda$ is \textbf{spherical} if its fixator $W_\lambda$ is finite, and it is \textbf{regular} if $W_\lambda$ is trivial. These notions only depend on the dominant part $\lambda^{++}$ of $\lambda$. We denote by \index{Y@$Y^+_{sph}$} $Y^+_{sph}$ the set of spherical coweights.
\paragraph{Affine Weyl semi-group}
Through the left action $W^v \curvearrowright Y$ we can form the semidirect product $Y\rtimes W^v$, for $(\lambda,w)\in Y\rtimes W^v$ we denote by $\qp^\lambda w$ the corresponding element of $Y\rtimes W^v$. If $W^v$ is finite, this semi-direct product is (a finite extension of) a Coxeter group, hence admits a Bruhat order, a Bruhat length and a set of simple reflections. However when $W^v$ is infinite, it does not have a Coxeter structure, and there is no known analog of the Bruhat order on it.

By definition, $Y^+\subset Y$ is stable by the action of $W^v$, therefore we can form \index{w@$W^+$}$W^+=Y^+ \rtimes W^v$ which is a sub semi-group of $Y\rtimes W^v$. This semi-group is called \textbf{the affine Weyl semi-group}, it is sometimes referred as "the affine Weyl group" in the literature, even though it is not a group.  

We denote by \index{p@$\pr^{Y^+}$, $\pr^{Y^{++}}$} $\pr^{Y^+}$ the projection function $W^+\rightarrow Y^+$ and $\pr^{Y^{++}}$ the projection on the dominant coweight, so if $\bx=\qp^\lambda w \in W^+$, then $\pr^{Y^+}(\bx)=\lambda$ and $\pr^{Y^{++}}(\bx)=\lambda^{++}$. We say that an element $\bx\in W^+$ is \textbf{spherical} if its coweight $\pr^{Y^+}(\bx)$ is spherical.

\paragraph{Affine roots}
Let \index{p@$\Phi^a$}$\Phi^a=\Phi \times \mathbb Z$, it is \textbf{the set of affine roots} of $\mathcal D$. Let $(\beta,n)\in \Phi^a$, we say that $(\beta,n)$ is positive if $n>0$ or ($n=0$ and $\beta \in \Phi_+$); we write $\Phi^a_+$ for the set of positive affine roots. We also have $\Phi^a = \Phi^a_+ \sqcup -\Phi^a_+$. We refer to \cite[1.2]{philippe2023grading} for a geometric interpretation of affine roots and of the affinized Bruhat order.

The semidirect product $Y\rtimes W^v$ acts on $\Phi^a$ by:
\begin{equation}\label{eq : W^+ action}
    \qp^\lambda w.(\beta,n)=(w\beta,n+\langle \lambda,w\beta\rangle).
\end{equation}

For any $n \in \mathbb Z$, its sign is denoted $\sgn(n)\in \{-1,+1\}$, with the convention that $\sgn(0)=+1$.

For $n\in \Z$ and $\beta\in \Phi_+$, we set:\begin{align}\label{eq : affine_roots} \index{b@$\beta[n]$}
&\beta[n]=(\sgn(n) \beta,|n|\qp)\in \Phi_+^a \\
\index{s@$s_{\beta[n]}$}&s_{\beta[n]}=\qp^{n\beta^\vee} s_\beta \in  Y\rtimes W^v.\end{align}

Note that the element $s_{\beta[n]}$ does not belong to $W^+$ in general, but it is an element of order $2$ in $Y\rtimes W^v$, the \textbf{affine reflection associated to $\beta[n]$}.
\paragraph{Affinized Bruhat order}
Even though $Y\rtimes W^v$ is a priori not a Coxeter group when $W^v$ is infinite, Braverman, Kazhdan and Patnaik have defined a preorder on $W^+$ in \cite[B.2]{braverman2016iwahori} using the action of $Y\rtimes W^v$ on $\Phi^a$. It is called the \textbf{affinized Bruhat order on $W^+$}. It is defined as   the transitive closure of the relation $<$ defined by, for $\bx\in W^+$ and $\beta[n]\in \Phi^a_+$: $s_{\beta[n]}\bx<\bx \iff \bx^{-1}(\beta[n]) \notin \Phi^a_+$. In some sense this means that $\beta[n]$ belongs to the inversion set of $\bx^{-1}$. Note in particular that if $\bx^{-1}(\beta[n]) \notin \Phi_+^a$ and $\bx$ lies in $W^+$ then so does $s_{\beta[n]}\bx$ .

Explicitly, if $\bx=\qp^\lambda w \in W^+$, then $\bx^{-1}(\beta[n])\notin \Phi_+^a$ if and only if $|n|<\sgn(n)\langle \lambda,\beta\rangle$, or $|n|=\sgn(n)\langle \lambda,\beta\rangle$ and $\sgn(n)w^{-1}\beta \in\Phi_-$.

D. Muthiah (see \cite{muthiah2018iwahori}) has proved that the transitive closure of $<$ is anti-symmetric, hence it is an order. Moreover, together with D. Orr, they define in \cite{muthiah2019bruhat} a $\mathbb R$-valued length function on $W^+$ strictly compatible with the affinized Bruhat order which we now introduce.
\paragraph{Affinized Bruhat length}
    If $\qp^\lambda w \in W^+$, the \textbf{affinized Bruhat length} is defined by: \index{l@$\ell^a$} \begin{equation}\label{eq: affine length formula}\ell^a(\qp^\lambda w)=2\htt(\lambda^{++})+\ell_{v^\lambda}(w).\end{equation}
    The function \index{l@$\ell_v$} $\ell_{v^\lambda}$ is the relative Bruhat length on $W^v$: $\ell_v(w)=\ell(v^{-1}w)-\ell(v)$ for all $v,w\in W^v$. The equivalence between this definition of the affinized Bruhat length and its definition in \cite{muthiah2019bruhat} is given by \cite[Proposition 2.10]{philippe2023grading}. This length is strictly compatible with the affinized Bruhat order by \cite[Theorem 3.3]{muthiah2019bruhat}. Note that it depends on the choice of height function, and since $\htt$ may not take values in $\mathbb Z$, it is not integral in general. However, since two comparable elements $\bx,\by$ satisfy $\pr^{Y^{++}}(\bx)-\pr^{Y^{++}}(\by)\in Q^\vee$, the length difference $\ell^a(\bx)-\ell^a(\by)$ is an integer independent of this choice of height function.

\paragraph{Covers and gradation of partially ordered sets}
\begin{Definition}\label{Definition: covers}Let $(X,<)$ be a partially ordered set. 
\begin{itemize}
    \item For $x,y\in X$, we say that $y$ \textbf{covers} $x$ if: $x<y$ and $\{z\in X\mid x<z<y\}=\emptyset$. We write $x\lhd y$ if $y$ covers $x$. We also say that $x$ \textbf{cocovers} $y$.

    \item A length function $\ell: X\rightarrow \mathbb R$ is said to be \textbf{strictly compatible} with $<$ if, for a comparable pair $(x,y)\in X^2$, we have $x<y\iff \ell(x)<\ell(y)$. Moreover it is a \textbf{$\mathbb Z$-grading} of $(X,<)$ if: $y$ covers $x$ if and only if $y>x$ and $\ell(y)=\ell(x)+1$. 
\end{itemize}

\end{Definition}
The Bruhat length on $W^v$ is an example of a grading (actually $\mathbb N$-valued) of $(W^v,<)$, where $<$ is the Bruhat order, we refer to \cite[Chapter 2]{bjorner2005combinatorics} for general properties of the Bruhat order.

The affinized Bruhat length is a $\mathbb Z$-grading of $(W^+,<)$ (by \cite[Theorem 6.2]{muthiah2019bruhat} in particular cases and \cite{philippe2023grading} in the general case): for all $\bx,\by\in W^+$ such that $\bx<\by$, \begin{equation}\label{e_characterization_cover}
    \bx \lhd \by \Leftrightarrow \ell^a(\by)=\ell^a(\bx)+1.
\end{equation}

The next lemma is a weaker reformulation of \cite[Proposition 3.20]{philippe2023grading}, which is a more explicit classification of covers in $W^+$. In the reductive case, this is a reformulation of the classification of covers given by Schremmer in terms of quantum Bruhat graph (see \cite[Proposition 4.5]{schremmer2024affinebruhat}).

\begin{Lemma}\label{Lemma : form of covers}
    Let $\bx,\by \in W^+$ and suppose that $\by$ covers $\bx$. Then there are unique elements $\lambda \in Y^{++},\,v, w \in W^v,\, \beta\in \Phi_+$ and $n\in \mathbb Z$ such that:
    \begin{itemize}
        \item $v$ is a minimal coset representative: $v\in W^\lambda$
        \item $\bx = \qp^{v\lambda} w$
        \item $\by=s_{v(\beta)[n]}\bx = \qp^{v(s_\beta \lambda +n\beta^\vee)}s_{v(\beta)} w$.
    \end{itemize}
    Moreover $n \in \{0,\langle \lambda,\beta\rangle,-1, \langle \lambda,\beta\rangle +1\}$, and $\pr^{Y^{++}}(\by)=\pr^{Y^{++}}(\bx)$ if and only if $n \in \{0,\langle \lambda,\beta\rangle\}$.

    Furthermore:
    \begin{enumerate}[i)]
        \item If $n=\langle \lambda,\beta\rangle$ (possibly equal to zero), then $\ell(s_\beta v^{-1}w)=\ell(v^{-1}w)+1$, so $v^{-1}w \lhd s_\beta v^{-1}w$
        
        \item If $n=0\neq\langle \lambda,\beta\rangle$, then $\ell(vs_\beta)=\ell(v)-1$, so $vs_\beta \lhd v$.

        \item If $n\in \{-1,\langle \lambda,\beta\rangle+1\}$ then $\beta$ is a quantum root and there exists an element $u \in W^v$ such that $\pr^{Y^+}(\by)=vu.\pr^{Y^{++}}(\by)$ and $\ell(vu)=\ell(v)+\ell(u)$, in particular $v\leq vu$. 
    \end{enumerate}
\end{Lemma}
The first part is a direct consequence of the definition of the affinized Bruhat order and of \cite[Prop. 3.1]{philippe2023grading}, items $i)$ and $ii)$ correspond to items $1)$ and $2)$ of \cite[Proposition 3.20]{philippe2023grading} whereas $iii)$ corresponds to subitems $(a),(c)$ of items $3)$ and $4)$ of the same reference.

 In particular, we use finiteness of the set of quantum roots to show that any element $\bx \in W^+$ has a finite number of covers, and if it is spherical, then it also has a finite number of co-covers. This extends to the general Kac-Moody case a result obtained by Welch in the affine ADE case (\cite[Theorem 1.1]{welch2022classification}). We provide examples of non-spherical elements which admit an infinite number of co-covers, and we give a characterization of the irreducible Kac-Moody root systems for which every element admits a finite number of co-covers. At the end of this section, we construct explicit covers in $W^+$ associated to any given quantum root, therefore showing that the notion of quantum root is the most refined one in this context.

\subsection{Finiteness results for covers and co-covers}

\begin{Theorem}\label{Theorem : finiteness of covers}
    Suppose that $\bx$ is an element of an affine Weyl semi-group $W^+$. Then $\bx$ has a finite number of covers for the affinized Bruhat order.
\end{Theorem}

\begin{proof}
For any $u\in W^v$, let $n_{cov}(u)$ (resp. $n_{cocov}(u)$) denote the number of covers (resp. cocovers) of $u$ in $W^v$. Since $W^v$ is a Coxeter group of finite rank, these numbers are finite. 

Let $\bx$ be an element of $W^+$. Then, by item $iii)$ of Lemma \ref{Lemma : form of covers}, the number of covers $\by$ of $\bx$ such that $\pr^{Y^{++}}(\by)\neq \pr^{Y^{++}}(\bx)$ is bounded by twice the number $N_{asr}(\Phi)$ of quantum roots of $\Phi$, which is finite by Theorem~\ref{t_finiteness_almost_simple_roots}. Moreover, write $\bx=\qp^{v\lambda}w$ with $\lambda = \pr^{Y^{++}}(\bx)$ and $v\in W^v$ minimal, then the number of covers $\by$ such that $\pr^{Y^{++}}(\by)= \pr^{Y^{++}}(\bx)$ is exactly $n_{cov}(v^{-1}w)+n_{cocov}(v)$ using items $i),ii)$ of the same lemma. Hence $\bx$ has at most $2N_{asr}(\Phi)+n_{cov}(v^{-1}w)+n_{cocov}(v)$ covers.
\end{proof}

\begin{Corollary}\label{c_finiteness_chains}
    Let $n\in \Z_{\geq 0}$ and $\bx\in W^+$. Then \[\{\by\in W^+\mid \exists (\bx_1,\ldots,\bx_n)\in (W^+)^n, \bx=\bx_1\lhd \bx_2\lhd\ldots \bx_{n-1}\lhd \bx_{n}=\by\}\] is finite.
\end{Corollary}

\begin{Corollary}
    Let $\bx,\by\in W^+$. Then $[\bx,\by]:=\{\bz\in W^+\mid \bx\leq \bz\leq \by\}$ is finite.
\end{Corollary}

\begin{proof}
Assume $\bx\leq \by$. 
    As the Bruhat order is strictly compatible with $\ell^a$ (by \cite[Theorem 3.3]{muthiah2019bruhat}), for every $\bz\in [\bx,\by]$, there exist $k\in \llbracket 0,\ell^a(\by)-\ell^a(\bx)\rrbracket$ and $\bx_1,\ldots,\bx_k\in W^+$ such that $\bx=\bx_1\lhd\bx_2\lhd \ldots \lhd \bx_k=\bz$ and thus the result follows from Corollary~\ref{c_finiteness_chains}. 
\end{proof}

\begin{Lemma}\label{l_characterization_bruhat_order_Yin}
    Let $Y_{in}=\{\lambda\in Y\mid \langle \lambda,\alpha_i\rangle=0,\forall i \in I\}$. Let $\lambda\in Y_{in}$, $\mu\in Y^+$ and $v,w\in W^v$. Then $\qp^{\mu}w \leq \qp^{\lambda}v$ if and only if $\mu=\lambda$ and $w\leq v$ (for the Bruhat order on $W^v$). 
\end{Lemma}

\begin{proof}
Let $\beta\in \Phi_+$ and $n\in \Z$ be such that $s_{\beta[n]} \qp^\lambda v<\qp^\lambda v$, note that $s_{\beta[n]} \qp^\lambda v=\qp^{\lambda+n\beta^\vee} s_\beta v$. Then by \cite[Lemma 1.7]{philippe2023grading} we have $\lambda+n\beta^\vee\in [\lambda,s_\beta.\lambda]_Y=\{\lambda\}$ and hence $n=0$.

    As $\lambda\in Y_{in}$, we have $W^\lambda=\{1\}$. Now $s_\beta \qp^\lambda v=\qp^\lambda s_\beta v< \qp^\lambda v$ if and only if $\ell^a(\qp^\lambda s_\beta v)<\ell^a(\qp^\lambda v)$ if and only if $\ell(s_\beta v)< \ell(v)$ if and only if $s_\beta v< v$ for the Bruhat order on $W^v$. Lemma follows. 
\end{proof}

\begin{Theorem}\label{thm_finiteness_co-covers}
    Suppose that $\by\in W^+$ is spherical, or that $\pr^{Y^+}(\by)\in Y_{in}$. Then it admits a finite number of co-covers.
\end{Theorem}
\begin{proof}
    Suppose that $\by$ is spherical, let us write $\by=\qp^\mu w$ with $\mu\in Y^+_{sph}$. Let $C(\by)$ denote the set of co-covers of $\by$ in $W^+$. We have $C(\by)=C_1(\by)\sqcup C_2(\by)$, where $C_1(\by)=\{\bx \in C(\by)\mid \pr^{Y^{++}}(\bx)=\pr^{Y^{++}}(\by)=\mu^{++}\}$ and $C_2(\by)=C(\by)\setminus C_1(\by)$. We show separately that both of these sets are finite. Recall that $v^\mu$ denotes the minimal element of $\{v\in W^v \mid \mu=v(\mu^{++})\}$.
    \begin{itemize}
        \item Suppose that $\bx \in C_1(\by)$, then by Lemma \ref{Lemma : form of covers}, we can write $\bx=\qp^{v(\mu^{++})}\tilde w$ and $\by=s_{v(\beta)[n]}(\bx)$ for some $\beta \in \Phi_+$, $v\in W^{\mu^{++}},\tilde w\in W^v$ and $n\in \{0,\langle \mu^{++},\beta\rangle\}$. The element $\tilde w$ is entirely determined by $(v,\beta)$ since $\tilde w =s_{v(\beta)}w$. So we need to prove that there is a finite number of viable choices for $(v,\beta)$. We have:
        \begin{itemize}
            \item If $n=0\neq \langle  \mu^{++},\beta\rangle$, then $\mu=vs_\beta(\mu^{++})$, so $vs_\beta$ lies in $v^\mu.W_{\mu^{++}}$, where $W_{\mu^{++}}$ is the stabilizer of $\mu^{++}$ in $W^v$. By assumption, $\mu$ is spherical so $W_{\mu^{++}}$ is finite, and so is $v^\mu.W_{\mu^{++}}$. Moreover, by item $ii)$ of Lemma \ref{Lemma : form of covers}, $v$ covers $vs_\beta$. Since $W^v$ is a Coxeter group of finite rank, any element of $v^\mu W_{\mu^{++}}$ admits a finite number of covers, hence the pair $(v,\beta)$ can only take a finite number of values.
            \item If $n\neq 0$, then $\bx$ and $\by$ have the same co-weight $\mu$, so $v=v^\mu$. Moreover, $\tilde w = s_{v(\beta)}w=vs_\beta v^{-1} w$ and therefore by item $i)$ of Lemma \ref{Lemma : form of covers}, $v^{-1}w=s_\beta v^{-1}\tilde w$ covers $v^{-1}\tilde w = s_\beta v^{-1} w$. So the number of viable choices of $\beta$ is equal to the number of co-covers of $v^{-1}w$, which is finite.
        \end{itemize}
        Either way, we thus have shown that $C_1(\by)$ is finite, let us now turn to $C_2(\by)$.
        \item Suppose that $\bx \in C_2(\by)$, by Lemma \ref{Lemma : form of covers} we can write $\bx=\qp^{v\lambda}\tilde w$ with $\lambda \in Y^{++}, \,\tilde w\in W^v, \,v\in W^\lambda$ and $\by=s_{v(\beta)[n]}\bx$ for $\beta \in \Phi_+$ and $n\in \{-1,\langle \lambda,\beta\rangle+1\}$. The coweight $\lambda$ is either equal to $v^{-1}\mu-\beta^\vee$ (if $n\neq1$) or to $s_\beta v^{-1}\mu -\beta^\vee$ (if $n=-1$). On top of that, $\tilde w=s_{v(\beta)}w$ and thus it suffices to show that the set of pairs $(v,\beta)$ such that $\by\in \{\qp^{v(\lambda+\beta^\vee)}w,\qp^{vs_\beta(\lambda+\beta^\vee)}w\}$ and is a cover of $\qp^{v\lambda}s_{v(\beta)}w$ (for one of these two choices of $\lambda$) is finite. Note that, since $\mu$ is spherical, the set $\{\tilde v \in W^v\mid \mu=\tilde v \mu^{++}\}=v^\mu.W_{\mu^{++}}$ is finite so we can define $L=\max\{\ell(\tilde v)\mid \tilde v\in W^v,\tilde v\mu^{++}=\mu\}$. Then by item $iii)$ of Lemma \ref{Lemma : form of covers}, we have $\ell(v)<L$, and $\beta$ is a quantum root. Therefore, since $\{v\in W^v\mid \ell(v)<L\}$ is finite and the set of quantum roots is finite as well by Theorem~\ref{t_finiteness_almost_simple_roots}, we deduce that $C_2(\by)$ is finite.
    \end{itemize}
    If $\by\in W^+$ is such that $\pr^{Y^+}(\by)\in Y_{in}=\bigcap_{i\in I}\ker(\alpha_i)$, then by Lemma~\ref{l_characterization_bruhat_order_Yin} $\pr^{Y^+}(\bx)=\pr^{Y^+}(\by)$ for all $\bx \in W^+$ such that $\bx<\by$, and  there is as many co-covers of $\by$ as co-covers of $w$ in $W^v$, hence a finite number. 
\end{proof}

\begin{Remark}
    When $Y$ is associated with an affine generalized Cartan matrix, then every element of $Y^+$ is either spherical or in $Y_{in}$ (by \cite[Proposition 5.8 b)]{kac1994infinite}). Thus the theorem above generalizes \cite[Corollary 24]{welch2022classification}.
\end{Remark}

Recall that $A=(\alpha_j(\alpha_i^\vee))_{i,j\in I}$ is the generalized Cartan matrix of the root datum $\cD$. Following \cite[1]{kac1994infinite}, we say that the matrix $A$ is \textbf{indecomposable} if for every non-empty subsets $J,K\subset I$ such that $J\sqcup K=I$, there exists $(j,k)\in J\times K$ such that  $\alpha_j(\alpha_k^\vee)\neq 0$.

 \begin{Lemma}\label{l_reciprocal_thm_cocovers}
     Assume that $A$ is indecomposable  and that there exists $\mu\in Y^{++}$ such that $W_{\mu}$ is infinite and $W^v$ does not fix $\mu$. Then there exists $\bx\in W^+$ admitting   infinitely many co-covers.
 \end{Lemma}

 \begin{proof}
 Let $J=\{i\in I\mid \alpha_i(\mu)=0\}$. Then $J\subsetneq I$. Let $J'\subset J$ be such that $W_{J'}$ is infinite and  for every $J''\subsetneq J'$, $W_{J''}$ is finite. Let $\nu\in Y$ be such that 
 $\alpha_j(\nu)=3$ for every $j\in I\setminus J'$ and $\alpha_j(\nu)=0$ for every $j\in J'$. Then $\nu\in  Y^{++}$ and $W_\nu=W_{J'}$.

 As $A$ is indecomposable, there exists $(i,j)\in (I\setminus J')\times J'$ such that $\alpha_i(\alpha_j^\vee)\neq 0$.  Set $\lambda=\nu-\alpha_i^\vee\in ]\nu,r_i.\nu[_{Y}$. Then $\alpha_j(\lambda)=\alpha_j(\nu)-\alpha_j(\alpha_i^\vee)=-\alpha_j(\alpha_i^\vee)>0$, $\alpha_i(\lambda)=\alpha_i(\nu)-2>0$ and for every $i'\in I\setminus \{j\}$, we have $\alpha_{i'}(\lambda)=\alpha_{i'}(\nu)-\alpha_{i'}(\alpha_j^\vee)\geq \alpha_{i'}(\nu)\geq 0$. Therefore $\lambda\in Y^{++}$ and  $\{k\in I\mid \alpha_k(\lambda)=0\}\subset J'\setminus\{j\}$. Thus $W_{\lambda}\subsetneq W_\nu$ and by assumption on $J'$, $W_{\lambda}$ is finite.

     Let $w\in W_\nu\cap W^\lambda$ (we prove below that this set is non-empty).  Set $\beta=w(\alpha_i)$ and $\bx_w=\qp^{w.\lambda}s_{\beta}$.  We have: \[s_{\beta[\langle w.\lambda,\beta\rangle+1]}\bx_{w}=\qp^{(\langle w.\lambda,\beta\rangle+1)\beta^\vee} \qp^{s_\beta w.\lambda}=\qp^\nu\] and thus $\bx_w<\qp^\nu$ if and only if $\ell^a(\bx_w)<\ell^a (\qp^\nu)$. Moreover  we have: 
     \[\begin{aligned} \ell^a(\bx_w)&=2\htt(\lambda)+\ell_{w}(s_{w(\alpha_i)})\\ 
     &=2(\htt(\nu)-1)+\ell(w^{-1} wr_iw^{-1})-\ell(w)\\
     &=\ell^a(\qp^\nu)-2+\ell(r_iw^{-1})-\ell(w).\end{aligned}\]

     Therefore  \begin{equation}\label{e_difference_length}
         \ell^a(\bx_w)\in \{\ell^a(\qp^\nu)-1,\ell^a(\qp^{\nu})-3\},
     \end{equation}
for every $w\in W^\lambda\cap W_\nu$.

Let $w\in W_\nu$. Then we can write uniquely $w=vu$, with $v\in W^\lambda$, $u\in W_\lambda$. But then $u\in W_\nu$ and hence $v\in W^\lambda\cap W_\nu$. Therefore $W_\nu.\lambda=(W_\nu\cap W^\lambda).\lambda$. As $W_\lambda$ is finite,  $W_\nu.\lambda$ is infinite and thus $(W^\lambda\cap W_\nu).\lambda$  and $\{\bx_w\mid w\in W^\lambda\cap W_\nu\}$ are  infinite.

 For $\by\in W^+$ denote by $C'(\by)$ the union of the set of co-covers of $\by$ with $\{\by\}$. By compatibility of $\ell^a$ with $<$  and by \eqref{e_difference_length},  for every $w\in W^\nu\cap W^\lambda$, we have $\bx_w\in C'(C'(C'(\qp^\nu)))$. Therefore \[C'(C'(C'(\qp^\nu)))\supset \{\bx_w\mid w\in W^\lambda\cap W_\nu\},\]  and the right hand side is infinite. Consequently there exists $\bx\in W^+$ for which $C'(\bx)$ is infinite.
 \end{proof}

\begin{Remark}
The situation of Lemma~\ref{l_reciprocal_thm_cocovers} can happen only when $A$ is of indefinite type. Indeed, if $A$ has finite type, $W^v$ is finite and if $A$ has affine type, then every element of $Y^{++}$ with infinite fixator is fixed by $W^v$, by \cite[Proposition 5.8 b) and Proposition 3.12 f)]{kac1994infinite}.
\end{Remark}

\begin{Example}
 Assume that $I$ has cardinality $3$ and write $I=\{i,j,k\}$. Assume that $\alpha_m(\alpha_{m'}^\vee)\alpha_{m'}(\alpha_{m}^\vee)\geq 4$ for every $m,m'\in \{i,j,k\}$. By \cite[1.3.21 Proposition]{kumar2002kac}, this implies that $\langle r_i,r_j\rangle$, $\langle r_j,r_k\rangle$ and $\langle r_i,r_k\rangle$ are infinite. Take $\nu\in Y$ such that $\alpha_i(\nu)\geq 3$, $\alpha_j(\nu)=\alpha_k(\nu)=0$. Set $\lambda=\nu-\alpha_i^\vee\in Y^{++}$. Then $W_\lambda=\{1\}$ and with the same proof as in Lemma~\ref{l_reciprocal_thm_cocovers}, we can prove that $\qp^{w.\lambda}wr_iw^{-1}$ is covered by $\qp^\nu$ for every $w\in \langle r_j,r_k\rangle$.

\end{Example}

Let $\A=Y\otimes \R$\index{a@$\A$} be the \textbf{standard apartment}. The elements of $\Phi$ can naturally be regarded as elements of the dual $\A^*$ of $\A$. The Tits cone $\sT$\index{t@$\sT$} is then $W^v.\overline{C^v_f}$, where  $\overline{C^v_f}=\{x\in \A \mid \langle \lambda,\alpha_i\rangle\geq 0,\forall i\in I\}$.

\begin{Proposition}\label{Proposition: nonspherical infinite cocovers}
We assume that $(\alpha_i(\alpha_j^\vee))_{i,j\in I}$ is indecomposable. Let $\A_{in}=\bigcap_{i\in I} \ker(\alpha_i)_\R$, where $\ker(\alpha_i)_\R=\{x\in \A\mid \alpha_i(x)=0\}$, for $i\in I$.
The following properties are equivalent:
\begin{enumerate}
    \item The set $\sT\setminus \A_{in}$ is open.

     \item Every proper parabolic subgroup of $W^v$ is finite.

    \item For every $\bx\in W^+$, the set of co-covers of $\bx$ is finite.

    \end{enumerate}
\end{Proposition}

\begin{proof}
(1) $\Rightarrow$ (2).  Assume (1). Let $J$ be a proper subset of $I$. Let $\lambda\in Y$ be such that $\alpha_j(\lambda)=0$ for all $j\in J$ and $\alpha_j(\lambda)>0$ for all $j\in I\setminus J$. Then $\lambda\in \overline{C^v_f}\setminus \A_{in}\subset \sT\setminus \A_{in}\subset \mathring{\sT}$. Then $W_\lambda=W_J$ and by \cite[Proposition 3.12 f)]{kac1994infinite}, $W_J$ is finite.

(2) $\Rightarrow$ (1). Assume (2). Let $x\in \sT\setminus \A_{in}$. We can write $x=w.y$, for some $w\in W^v$ and $y\in \overline{C^v_f}\setminus \A_{in}$. Let $J=\{i\in I\mid \alpha_i(y)=0\}$. Then $J\subsetneq I$ and thus $W_y=W_J$ is finite. By \cite[Proposition 3.12 f)]{kac1994infinite}, there exists an open neighbourhood $V$ of $y$ contained in $\sT$. Then the fixator of every element of $V$ is finite and thus $V\subset \sT\setminus \A_{in}$. Then $x\in w^{-1}.V\subset \sT\setminus \A_{in}$, which proves that $\sT\setminus \A_{in}$ is open. 

(3) $\Rightarrow$ (1) Assume (3). Then by Lemma~\ref{l_reciprocal_thm_cocovers}, for every $\mu\in Y^+$, we have either $W_\mu$ is finite or $W_\mu=W^v$. Let $x\in \sT$. Write $x=w.y$, with $w\in W^v$ and $y\in \overline{C^v_f}$. Let $J=\{i\in I\mid \alpha_i(y)=0\}$. Let $\mu\in Y$ be such that for $i\in I$, $\alpha_i(\mu)=1$ if $i\in I\setminus J$ and $\alpha_i(\mu)=0$ if $i\in J$. Then $W_\mu=W_x$. If $W_x=W^v$, then $x\in \A_{in}$ and if $W_x$ is finite, then $x\in \mathring{\sT}$ by \cite[Proposition 3.12 f)]{kac1994infinite}, which proves that $\sT=\A_{in}\sqcup \mathring{\sT}$. Thus we have 1).

(1) $\Rightarrow$ (3) is a consequence of Theorem~\ref{thm_finiteness_co-covers}.
\end{proof}
 
\subsection{Explicit covers for a given quantum root}
In this section, we construct, for any quantum root, explicit covers where it appears, hence the  restriction on $\beta\in \Phi_+$ appearing  in Lemma~\ref{Lemma : form of covers} is optimal. 
\begin{Proposition}\label{p_almost_simple_roots_covers}
    Let $\beta$ be a quantum root, and let $\lambda \in Y^{++}_{reg}$ be such that $\lambda+\beta^\vee \in Y^{++}_{reg}$ (explicitly, $\langle \lambda,\alpha_i\rangle > -\langle \beta^\vee,\alpha_i\rangle$ for all $i\in I$). Then:
    \begin{enumerate}
       
        \item Let $v$ be any element of $W^v$ and let $w\in W^v$ be such that $\ell(s_\beta w)=\ell(s_\beta)+\ell(w)$ (in particular this is always verified for $w=1_{W^v}$). Then $\qp^{v(\lambda+\beta^\vee)}vw$ covers $\qp^{v(\lambda)}vs_\beta w$.
        
        \item Let $v\in W^v$ be such that $\ell(vs_\beta)=\ell(v)+\ell(s_\beta)$ and let $w$ be any element of $W^v$. Then $\qp^{vs_\beta(\lambda+\beta^\vee)}s_{v(\beta)}w$ covers $\qp^{v(\lambda)}w$.
    \end{enumerate}
In particular, for any quantum root $\beta$ there exist covers of the form given in Lemma \ref{Lemma : form of covers} with varying dominance class.
\end{Proposition}

\begin{proof}
    Recall the expression of the affinized Bruhat length obtained in \cite[Prop. 2.10]{philippe2023grading}: For any $\lambda \in Y^{++}_{reg}$, $v,w\in W^v$:
    \begin{equation}
        \ell^a(\qp^{v(\lambda)}w)=2\htt(\lambda)+\ell(v^{-1}w)-\ell(v).
    \end{equation}
    Thus, for any $\lambda \in Y^{++}_{reg}$ such that $\lambda+\beta^\vee \in Y^{++}_{reg}$:
    \begin{align*}
        \ell^a(s_{v(\beta)[\langle \lambda,\beta\rangle +1]}\qp^{v(\lambda)}v s_\beta w)&=\ell^a(\qp^{v(\lambda+\beta^\vee)}v w)\\ &=2\htt(\lambda+\beta^\vee)+\ell(w)-\ell(v)\\
        &=\ell^a(\qp^{v(\lambda)}vs_\beta w)+2 \htt(\beta^\vee)+\ell(w)-\ell(s_\beta w).\\
        \ell^a(s_{v(\beta)[-1]}\qp^{v(\lambda)}w) &= \ell^a(\qp^{vs_\beta(\lambda+\beta^\vee)}vs_\beta v^{-1}w)\\ &=2\htt(\lambda+\beta^\vee)+\ell(v^{-1}w)-\ell(vs_\beta)\\ &=\ell^a(\qp^{v(\lambda)}w)+2\htt(\beta^\vee)+\ell(v)-\ell(vs_\beta).
    \end{align*}
    Therefore, since quantums root are characterized by the fact that $2\htt(\beta^\vee)=\ell(s_\beta)+1$ (by Lemma~\ref{l_characterization_almost_simpleness}), we obtained that the length difference $ \ell^a(s_{v(\beta)[\langle \lambda,\beta\rangle +1]}\qp^{v(\lambda)}v s_\beta w)-\ell^a(\qp^{v(\lambda)}vs_\beta w)$ (resp. $\ell^a(s_{v(\beta)[-1]}\qp^{v(\lambda)}w)-\ell^a(\qp^{v(\lambda)}w)$) is exactly one if and only if condition 1. (resp. condition 2.) is satisfied. 
\end{proof}

\printindex

\bibliography{bibliographie.bib}

\newcommand{\etalchar}[1]{$^{#1}$}
\begin{thebibliography}{LNS{\etalchar{+}}12}

\bibitem[Bar96]{bardy1996systemes}
Nicole Bardy.
\newblock Syst\`emes de racines infinis.
\newblock {\em M\'em. Soc. Math. Fr. (N.S.)}, (65):vi+188, 1996.

\bibitem[BB05]{bjorner2005combinatorics}
Anders Bj\"{o}rner and Francesco Brenti.
\newblock {\em Combinatorics of {C}oxeter groups}, volume 231 of {\em Graduate
  Texts in Mathematics}.
\newblock Springer, New York, 2005.

\bibitem[BFP99]{brenti1999mixedbruhat}
Francesco Brenti, Sergey Fomin, and Alexander Postnikov.
\newblock Mixed {B}ruhat operators and {Y}ang-{B}axter equations for {W}eyl
  groups.
\newblock {\em Internat. Math. Res. Notices}, (8):419--441, 1999.

\bibitem[BKP16]{braverman2016iwahori}
Alexander Braverman, David Kazhdan, and Manish~M. Patnaik.
\newblock Iwahori-{H}ecke algebras for {$p$}-adic loop groups.
\newblock {\em Invent. Math.}, 204(2):347--442, 2016.

\bibitem[BL00]{billey2000singular}
Sara Billey and V.~Lakshmibai.
\newblock {\em Singular loci of {Schubert} varieties}, volume 182 of {\em Prog.
  Math.}
\newblock Boston, MA: Birkh{\"a}user, 2000.

\bibitem[BMO11]{braverman2011quantum}
Alexander Braverman, Davesh Maulik, and Andrei Okounkov.
\newblock Quantum cohomology of the springer resolution.
\newblock {\em Advances in Mathematics}, 227(1):421--458, 2011.

\bibitem[Bou81]{bourbaki1981elements}
Nicolas Bourbaki.
\newblock {\em \'{E}l\'{e}ments de math\'{e}matique}.
\newblock Masson, Paris, 1981.
\newblock Groupes et alg\`ebres de Lie. Chapitres 4, 5 et 6. [Lie groups and
  Lie algebras. Chapters 4, 5 and 6].

\bibitem[CP12]{chaput2012littlewood}
Pierre-Emmanuel Chaput and Nicolas Perrin.
\newblock Towards a {Littlewood}-{Richardson} rule for {Kac}-{Moody}
  homogeneous spaces.
\newblock {\em J. Lie Theory}, 22(1):17--80, 2012.

\bibitem[Dea25]{dean2025quantum}
Lewis Dean.
\newblock The quantum {B}ruhat graph for $\widehat{SL}_2$ and double affine
  {D}emazure products.
\newblock {\em Mathematische Zeitschrift}, 311, 09 2025.

\bibitem[HP24]{hebert2024affine}
Auguste H{\'e}bert and Paul Philippe.
\newblock On affine {Kazhdan}-{Lusztig} {R}-polynomials for {Kac}-{Moody}
  groups.
\newblock Preprint, {arXiv}:2410.04872 [math.{RT}] (2024), 2024.

\bibitem[Kac94]{kac1994infinite}
Victor~G Kac.
\newblock {\em Infinite-dimensional {L}ie algebras}, volume~44.
\newblock Cambridge university press, 1994.

\bibitem[Kum02]{kumar2002kac}
Shrawan Kumar.
\newblock {\em Kac-{M}oody groups, their flag varieties and representation
  theory}, volume 204 of {\em Progress in Mathematics}.
\newblock Birkh\"auser Boston, Inc., Boston, MA, 2002.

\bibitem[LNS{\etalchar{+}}12]{lenart2012uniform}
Cristian Lenart, Satoshi Naito, Daisuke Sagaki, Anne Schilling, and Mark
  Shimozono.
\newblock A uniform model for {K}irillov-{R}eshetikhin crystals {I}: {L}ifting
  the {P}arabolic {Q}uantum {B}ruhat {G}raph.
\newblock {\em International Mathematics Research Notices}, 2015, 11 2012.

\bibitem[MO19]{muthiah2019bruhat}
Dinakar Muthiah and Daniel Orr.
\newblock On the double-affine {B}ruhat order: the {$\varepsilon=1$} conjecture
  and classification of covers in {ADE} type.
\newblock {\em Algebr. Comb.}, 2(2):197--216, 2019.

\bibitem[Mut18]{muthiah2018iwahori}
Dinakar Muthiah.
\newblock On {I}wahori-{H}ecke algebras for {$p$}-adic loop groups: double
  coset basis and {B}ruhat order.
\newblock {\em Amer. J. Math.}, 140(1):221--244, 2018.

\bibitem[Mut19]{muthiah2019double}
Dinakar Muthiah.
\newblock Double-affine {K}azhdan-{L}usztig polynomials via masures.
\newblock {\em arXiv preprint arXiv:1910.13694}, 2019.

\bibitem[Pet97]{peterson1997lectures}
Dale Peterson.
\newblock Quantum {C}ohomology of {$G/P$}.
\newblock Lecture notes, MIT, Spring 1997. Notes by J. Lu and K. Rietsch, 1997.

\bibitem[Phi24]{philippe2023grading}
Paul Philippe.
\newblock Grading of affinized {W}eyl semigroups of {K}ac–{M}oody type.
\newblock {\em Pacific Journal of Mathematics}, 332:273--320, 12 2024.

\bibitem[Pos04]{postnikov2004quantum}
Alexander Postnikov.
\newblock Quantum {B}ruhat graph and {S}chubert polynomials.
\newblock {\em Proceedings of the American Mathematical Society}, 133:699--709,
  09 2004.

\bibitem[Pro99a]{proctor1999dynkin}
Robert~A. Proctor.
\newblock Dynkin diagram classification of {{\(\lambda\)}}-minuscule {Bruhat}
  lattices and of {{\(d\)}}-complete posets.
\newblock {\em J. Algebr. Comb.}, 9(1):61--94, 1999.

\bibitem[Pro99b]{proctor1999minuscule}
Robert~A. Proctor.
\newblock Minuscule elements of {Weyl} groups, the numbers game, and
  {{\(d\)}}-complete posets.
\newblock {\em J. Algebra}, 213(1):272--303, 1999.

\bibitem[R{\'e}m02]{remy2002groupes}
Bertrand R{\'e}my.
\newblock Groupes de {K}ac-{M}oody d\'eploy\'es et presque d\'eploy\'es.
\newblock {\em Ast\'erisque}, (277):viii+348, 2002.

\bibitem[Sch22]{schremmer2022thesis}
Felix Schremmer.
\newblock {\em Understanding affine Deligne-Lusztig varieties using the quantum
  Bruhat graph}.
\newblock PhD thesis, Technische Universität München, 2022.

\bibitem[Sch24]{schremmer2024affinebruhat}
Felix Schremmer.
\newblock {A}ffine {B}ruhat order and {D}emazure products.
\newblock {\em Forum of Mathematics, Sigma}, 12, 2024.

\bibitem[Ste01]{stembridge2001minuscule}
John~R. Stembridge.
\newblock Minuscule elements of {Weyl} groups.
\newblock {\em J. Algebra}, 235(2):722--743, 2001.

\bibitem[Wel22]{welch2022classification}
Amanda Welch.
\newblock Classification of cocovers in the double affine {Bruhat} order.
\newblock {\em Electron. J. Comb.}, 29(4), 2022.

\end{thebibliography}

\end{document}